\documentclass{amsart}
\usepackage{amsmath, amssymb, amscd, amsthm, amsfonts, amstext, amsbsy, xparse,  mathrsfs, hyperref,  upgreek, mathtools, stmaryrd, enumitem, tensor, comment, lineno, xspace}
\hypersetup{colorlinks=true}
\numberwithin{equation}{section}

\renewcommand{\subset}{\subseteq}

\newcommand{\enquote}[1]{``#1''}

\newcommand{\A}{\mathcal{A}}
\newcommand{\B}{\mathcal{B}}

\newcommand{\D}{\mathcal{D}}

\newcommand{\N}{\mathcal{N}}

\renewcommand{\P}{\mathcal{P}}

\newcommand{\U}{\mathcal{U}}
\newcommand{\V}{\mathcal{V}}

\newcommand{\GG}{\mathbb{G}}

\newcommand{\LL}{\mathbb{L}}
\newcommand{\MM}{\mathbb{M}}
\newcommand{\NN}{\mathbb{N}}

\newcommand{\QQ}{\mathbb{Q}}
\newcommand{\RR}{\mathbb{R}}
\renewcommand{\SS}{\mathbb{S}}

\newcommand{\ccc}{\mathfrak{c}}

\newcommand{\MMM}{\mathfrak{M}}
\newcommand{\NNN}{\mathfrak{N}}

\newcommand{\pre}[2]{\tensor[^{#1}]{#2}{}}

\newcommand{\Nbhd}{\mathbf{N}}
\newcommand{\Bor}{\mathsf{Bor}}

\newcommand{\markdef}[1]{\textbf{#1}}
\newcommand{\imarkdef}[2][]{\textbf{#2}}
\newcommand{\Gdelta}[1]{G_\delta^{#1}}
\renewcommand{\int}[2][]{\IfNoValueTF{#1}{\mathrm{int}}{\mathrm{int}_{#1}}(#2)}

\newcommand{\supp}{\operatorname{supp}}
\newcommand{\cof}{\operatorname{cof}}

\newcommand{\leng}{\operatorname{lh}}

\DeclareMathOperator{\Deg}{Deg}

\DeclareMathOperator{\conc}{\mathbin{ {}^\smallfrown {} }}
\DeclareMathOperator{\Conc}{\mathrm{Conc}}
\DeclareMathOperator{\Sum}{\mathsf{Sum}}
\newcommand{\Exp}{\mathcal{EX}}

\newcommand{\On}{{\sf On}}

\newcommand{\Lev}{\mathrm{Lev}}

\DeclareMathOperator{\cl}{cl}

\DeclareMathOperator{\dom}{dom}

\DeclareMathOperator{\powerset}{\mathscr{P}}

\DeclareMathOperator{\weight}{\mathnormal{w}}
\DeclareMathOperator{\density}{\mathnormal{d}}
\DeclareMathOperator{\cellularity}{\mathnormal{c}}

\newcommand{\PTOMs}{positively totally ordered monoids\xspace}

\newtheorem{theorem}{Theorem}[section]
\newtheorem{lemma}[theorem]{Lemma}
\newtheorem{corollary}[theorem]{Corollary}
\newtheorem*{corollary*}{Corollary}
\newtheorem{proposition}[theorem]{Proposition}

\newtheorem*{question*}{Question}

\theoremstyle{definition}
\newtheorem{claim}{Claim}[theorem]
\newtheorem*{claim*}{Claim}
\newtheorem{definition}[theorem]{Definition}
\newtheorem{fact}[theorem]{Fact}
\newtheorem{example}[theorem]{Example}
\theoremstyle{remark}
\newtheorem{remark}[theorem]{Remark}

\newenvironment{enumerate-(a)}{\begin{enumerate}[label={\upshape (\alph*)}, leftmargin=2pc]}{\end{enumerate}}

\newenvironment{enumerate-(a)-r}{\begin{enumerate}[label={\upshape (\alph*)}, leftmargin=2pc,resume]}{\end{enumerate}}

\newenvironment{enumerate-(A)}{\begin{enumerate}[label={\upshape (\Alph*)}, leftmargin=2pc]}{\end{enumerate}}

\newenvironment{enumerate-(A)-r}{\begin{enumerate}[label={\upshape (\Alph*)}, leftmargin=2pc,resume]}{\end{enumerate}}

\newenvironment{enumerate-(i)}{\begin{enumerate}[label={\upshape (\roman*)}, leftmargin=2pc]}{\end{enumerate}}

\newenvironment{enumerate-(i)-r}{\begin{enumerate}[label={\upshape (\roman*)}, leftmargin=2pc,resume]}{\end{enumerate}}

\newenvironment{enumerate-(I)}{\begin{enumerate}[label={\upshape (\Roman*)}, leftmargin=2pc]}{\end{enumerate}}

\newenvironment{enumerate-(I)-r}{\begin{enumerate}[label={\upshape (\Roman*)}, leftmargin=2pc,resume]}{\end{enumerate}}

\newenvironment{enumerate-(1)}{\begin{enumerate}[label={\upshape (\arabic*)}, leftmargin=2pc]}{\end{enumerate}}

\newenvironment{enumerate-(1)-r}{\begin{enumerate}[label={\upshape (\arabic*)}, leftmargin=2pc,resume]}{\end{enumerate}}

\begin{document}

\title{On the problem of generalized measures: an impossibility result}

\author{Claudio Agostini}
\address[Claudio Agostini]
{HUN-REN Alfréd Rényi Institute of Mathematics, Reáltanoda utca 13-15, H-1053, Budapest}
\email[Claudio Agostini]{agostini.claudio@renyi.hu}

\author{Fernando Barrera}
\address[Fernando Barrera]
{Institut f\"ur Diskrete Mathematik und Geometrie, Technische Universit\"at Wien, Wiedner Hauptstra{\ss}e 8-10/104, 1040 Vienna, Austria}
\email[Fernando Barrera]{f.barrera.esteban@gmail.com}

\author{Vincenzo Dimonte}
\address[Vincenzo Dimonte]
{Universit\`a degli Studi di Udine Vie delle Scienze 206, 33100 Udine, Italy}
\email[Vincenzo Dimonte]{vincenzo.dimonte@uniud.it}

\begin{abstract}
This paper investigates the problem of extending measure theory to non-separable structures, from generalized descriptive set theory to a broader class of spaces beyond this framework. 
While various notions, such as the ideal of measure zero sets, have been generalized, the question of whether a satisfactory notion of $\lambda^+$-measure could be defined in generalized descriptive set theory has remained open. We introduce a broad class of $\lambda^+$-measures as functions taking values in arbitrary positively totally ordered monoids equipped with an infinitary sum. This definition relies on minimal assumptions and captures most natural generalizations of measures to this context. We then prove that, under certain cardinal assumptions, no continuous $\lambda^+$-measure of this kind exists on $\pre{\kappa}{\lambda}$, nor on any $\lambda^+$-Borel space or $T_0$ topological space of weight at most~$\lambda$. We also show the optimality of these cardinal assumptions.
\end{abstract}


\thanks{
This research was funded by the \textsf{Austrian Science Fund (FWF) P35655 and P35588}, by the \textsf{OeAD Ernst Mach grant - worldwide} and the \textsf{Austrian Science Fund (FWF) [10.55776/STA139]},
by \textsf{Universit\`a degli Studi of Udine (Italy)}, by \textsf{Progetto PRIN 2022 – Models, Sets and Classifications – Code no. 2022TECZJA CUP G53D23001890006, funded by the European Union – NextGenerationEU – PNRR M4 C2 I1.1}, by the \textsf{Gruppo Nazionale per le Strutture Algebriche, Geometriche e le loro Applicazioni (GNSAGA)} of the Istituto Nazionale di Alta Matematica (INdAM),
and by the \textsf{European Union \includegraphics[height=1em]{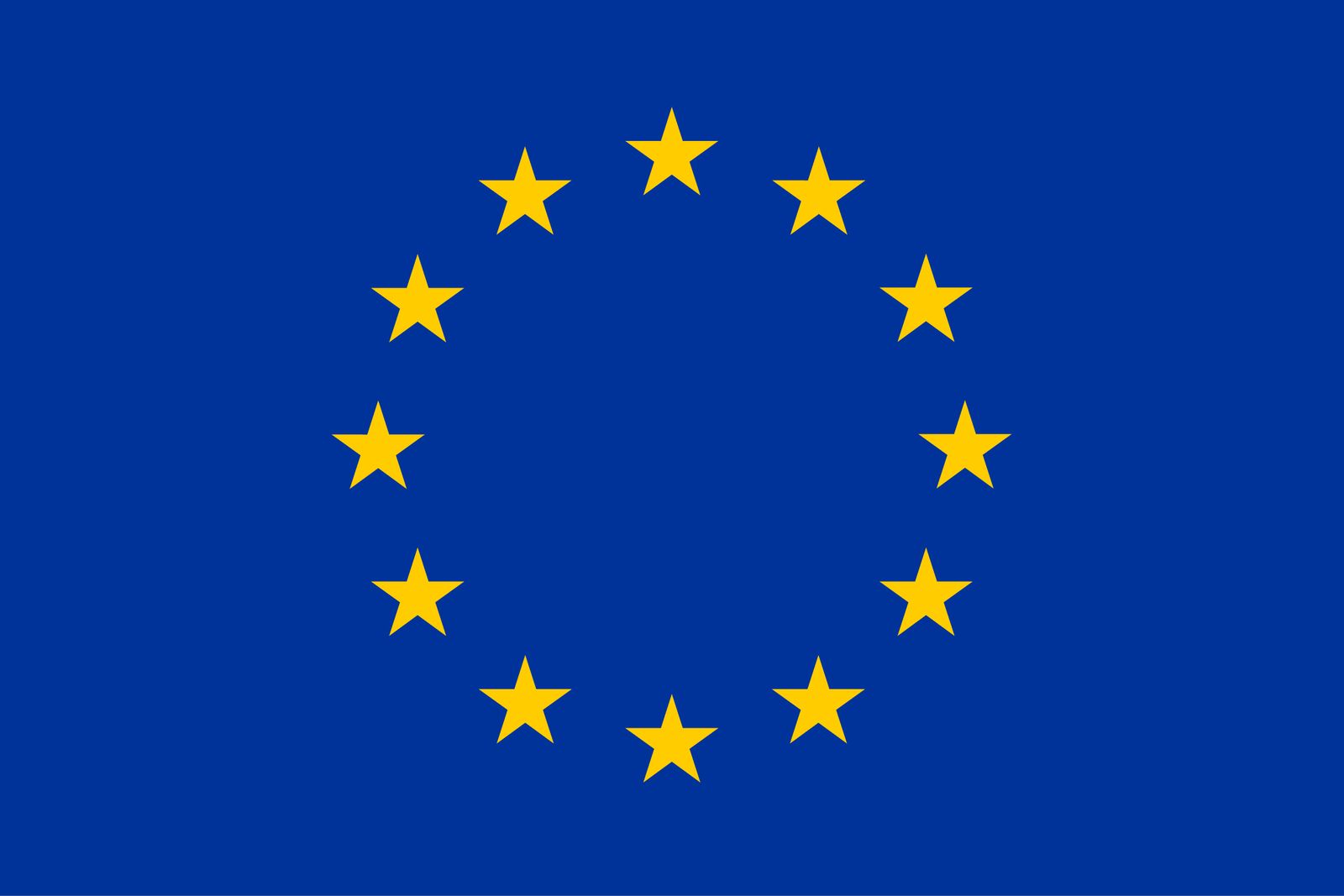} Horizon 2020 research and innovation programme, under the Marie Sklodowska-Curie action 101210902 (TopAspOfGDST)}. Views and opinions expressed are however those of the authors only and do not necessarily reflect those of the European Union or Horizon 2020 programme. Neither the European Union nor the granting authority can be held responsible for them. For open access purposes, the authors have applied a CC BY public copyright license to any author accepted manuscript version arising from this submission.
}

\maketitle

\tableofcontents

\section{Introduction}

Generalized descriptive set theory has roots dating back to the second half of the last century (see, e.g., \cite{stone1962non, MeklerVaananenMR1242054}). Over the past decade, the field has expanded rapidly, with an increasing number of notions and results being extended from classical descriptive set theory to its generalized counterpart. For instance, the reducibility of Borel equivalence relations on the generalized Baire space $\pre{\kappa}{\kappa}$, where $\kappa$ is a regular uncountable cardinal such that $2^{<\kappa} = \kappa$, was studied in \cite{FriedmanHyttKuli}, and classical results on Souslin quasi-orders were extended to this setting in \cite{AndrettaMottoRosMR4409720}. More recently, the framework was further developed (see, e.g.,  \cite{CoskSchlMR3453772, Gal19, AgosMottoSchlichtRegular}) to include topological spaces of a more general form than $\pre{\kappa}{\kappa}$, showing that it is possible to introduce a suitable generalization of metrics in this context and that the resulting theory remains well-behaved. Other projects (see, e.g., \cite{dimonte2025generalizeddescriptivesettheory, PhDThesisAgostini, agostiniGeneralizedBorelSets2025}) are aiming to further extend this theory to all cardinals $\lambda$ satisfying $2^{<\lambda} = \lambda$, including singular cardinals.

Various regularity properties, such as the Baire property, the perfect set property, the Hurewicz dichotomy, and more, have been extended and studied in the generalized setting (see, e.g., \cite{LuckeMottSchlMR3557473, Schlicht2017MR3743612, korchSpecialSubsetsGeneralized2020}). These results show that, at least consistently in certain cases, these properties can exhibit remarkably good behavior in the generalized context as well.

However, measure theory, a major topic of classical descriptive set theory with many applications to other areas of mathematics, has yet to find a proper analogue in the generalized setting. 
There have been efforts to generalize concepts related to measures. For example, \cite{friedman2017null} and \cite{cohen2019generalizing} proposed different generalizations of the ideal of null sets to the generalized Baire space. Yet, further progress has faced significant technical and conceptual obstacles. As a result, the question of whether a satisfactory notion of \enquote{measure} can be developed in this context has remained open.

In this paper, we settle this problem by showing that, under very mild assumptions and for a sufficiently large cardinal $\lambda$, it is impossible to define non-trivial generalized measures on a $\lambda^+$-Borel space. In particular, no such measure exists on the generalized Baire space.

We work in a slightly more general framework than the one usually adopted in generalized descriptive set theory (see Subsection~\ref{subsec:top_setup}).
Our results apply to any space $X$ of weight $\lambda$ of the form $X\subseteq \pre{\gamma}{A}$, for some set $A$ and ordinal $\gamma$, equipped with the bounded topology.
For spaces of this form, we argue that the appropriate notion of measure is that of a $\lambda^+$-additive measure (Proposition~\ref{prop:weight=cellularity}).

In practice, we can reduce to an equivalent and more familiar setting.
We fix two arbitrary cardinals $\lambda$ and $\kappa$, with $\kappa\leq\lambda$ and $\kappa$ regular, and work (until Section~\ref{sec:final_remarks}) with subspaces of $\pre{\kappa}{\lambda}$ of weight $\lambda$ (Fact~\ref{fct:homeo_to_subset_Baire} and Corollary~\ref{cor:embedding_into_Baire_kappa_leq_lambda}).
This way, when $\kappa=\cof(\lambda)$ and $2^{<\lambda}=\lambda$, we recover the standard setup used in other papers in generalized descriptive set theory on a cardinal $\lambda$ (see, e.g., \cite{PhDThesisAgostini, agostiniGeneralizedBorelSets2025}).

However, we do not impose any cardinal assumption a priori.
This allows us to study the spaces $\pre{\kappa}{\lambda}$ and $\pre{\lambda}{2}$ even when $\kappa\neq\cof(\lambda)$ or $2^{<\lambda}\neq\lambda$, a case that is usually not covered in generalized descriptive set theory.
In contrast, we show that the assumption $\lambda^{<\kappa}=\lambda$ can always be made without loss of generality:
up to homeomorphism, every space of the form $\pre{\kappa}{\lambda}$ is homeomorphic to a space $\pre{\kappa}{{\lambda'}}$ for some $\lambda'$ satisfying $(\lambda')^{<\kappa}=\lambda'$ (Proposition~\ref{prop:lambda<kappa=lambda_is_necessary}).
A similar statement holds for the generalized Cantor space (Proposition~\ref{prop:Cantor_always_subset_of_some_Baire}).

One of the main challenges in obtaining a meaningful impossibility theorem is to work with a class of \enquote{generalized measures} broad enough to encompass all reasonable generalizations of measure one might consider. We achieve this by relaxing the assumptions of classical measures as much as possible, introducing an \enquote{umbrella definition} that captures most (if not all) natural generalizations of measure to subspaces of $\pre{\kappa}{\lambda}$.

Recall that a \markdef{(classical) measurable space} on a set $X$ is a pair $(X,\MMM)$, where $X$ is a set and $\MMM$ is a $\sigma$-algebra on $X$. If $X$ is a topological space, it is further required that $\tau \subseteq \MMM$, i.e., that $\MMM$ extends the $\sigma$-algebra of Borel sets of $X$. In classical descriptive set theory, one typically restricts attention to \markdef{Borel spaces}, that is, pairs $(X,\MMM)$ where $X$ is a second countable $T_0$ topological space and $\MMM$ is the smallest $\sigma$-algebra generated by the topology of $X$.

These notions have straightforward generalizations to the uncountable setting.
We define a \markdef{$\lambda^+$-measurable space} as a pair $(X,\MMM)$ where $X$ is a topological space and $\MMM$ is a $\lambda^+$-algebra containing all open subsets of $X$ (that is, a family closed under unions and intersections of size at most $\lambda$, and under complements).
We call a pair $(X,\MMM)$ a \markdef{$\lambda^+$-Borel space} if $\MMM$ is a $\lambda^+$-algebra generated by a $T_0$ topology of weight at most $\lambda$, or, equivalently, if $\MMM$ separates the points of $X$ and is generated by a subfamily of size at most $\lambda$ (see, e.g., \cite[Proposition 3.24]{agostiniGeneralizedBorelSets2025}).

As our aim is to relax these assumptions as much as possible, we work with a larger class of structures that extends the classes of $\lambda^+$-measurable and $\lambda^+$-Borel spaces. In particular, we drop the requirement that  $\MMM$ is closed under complements or intersections, and we do not even require that $\Bor(X)\subseteq \MMM$.

\begin{definition}\label{def:weakly-lambda-measurable-space}
A \markdef{weakly $\lambda^+$-measurable space} is a pair $(X,\MMM)$ where $X$ is a topological space and $\MMM$ is a family of subsets of $X$ closed under unions of size $\leq \lambda$ and containing all open subsets of $X$.
\end{definition}

On a classical measurable space $(X,\MMM)$, one typically defines measures as functions on $\MMM$ taking values in $\RR_\infty = [0,\infty) \cup {\infty}$. More precisely, a \markdef{(classical) measure} $\mu$ on a measurable space $(X,\MMM)$ is a partial function $\mu:\powerset(X)\to \RR_\infty$
satisfying that $\MMM\subseteq \dom(\mu)$, $\mu(\emptyset)=0$, and
    \[
    \mu(\bigcup_{i\in\omega}A_i)=\sum_{i\in\omega}\mu(A_i)
    \]
    for any countable family $(A_i)_{i<\omega}\subset \MMM$ of disjoint sets.
We call $(X,\MMM,\mu)$ a \textbf{(classical) measure space}. To avoid trivialities, it is often required that $\mu(X)>0$ to ensure that $\mu$ is not constantly $0$.

To extend this definition to the uncountable case, one needs an ordered structure equipped with an operation extendable to sequences of the required (infinite) length. In the classical case, this role is played by the positively totally ordered monoid $\RR_\infty$ together with the infinitary operation $\sum$.

While there have been attempts to generalize the real numbers to uncountable settings while preserving as many of the properties of $\RR$ as possible (see, e.g., \cite{DalesWoodinMR1420859, Gal19, wontnerGeneralisationsDescriptiveSet}), we follow a different approach. Since our aim is to establish an impossibility theorem as general as possible, we work with an arbitrary positively totally ordered monoid $\SS = (S, 0, +, \leq)$, rather than restricting ourselves to a specific structure. Similar generalizations, where $\RR$ is replaced by monoids, have been successful in related contexts, such as that of metrics (see, e.g., \cite{ReichelTowardsUnifiedMR565844}).

As a generalization of infinite sums, we allow any partial function $\Sum : \pre{<\On}{\SS} \to \SS$ that agrees with $+$ on finite sequences in its domain and is \markdef{continuous} on sequences of limit length, that is, satisfying
\[
\Sum(s)=\sup\{\Sum(s\restriction \alpha) \mid \alpha<\leng(s),\ s\restriction \alpha\in \dom(\Sum)\}
\]
for all $s \in \dom(\Sum)$ such that $\leng(s)$ is a limit ordinal and $s\restriction \alpha \in \dom(\Sum)$ for cofinally many $\alpha < \leng(s)$ (see Definition~\ref{def:axioms_sum}). We refer to such an operation as an \markdef{infinitary operation} or \markdef{infinitary sum} on $\SS$.

Measures taking values in a monoid with a non-continuous function $\Sum$ typically fail to satisfy most of the fundamental properties that make the classical theory interesting. 
For this reason, the study of these objects is limited to Subsection~\ref{subsec:measures_with_non-continuous_sum}, while in the rest of the paper we restrict our attention to continuous infinitary sums.

This setup ensures a lot of freedom, as monoids include a much wider variety of examples than fields or groups, with much more varied behavior. For example, while a totally ordered group of uncountable degree can never be Dedekind complete, its Dedekind completion does exist within the category of monoids (Proposition~\ref{prop:completion_is_infinitary_ptom}).
At the same time, monoids can be extremely pathological.
They may fail the cancellation law and may contain idempotent elements, i.e.\ elements $a\in\SS$ such that $a+a=a$.
Moreover, the algebraic operation may interact poorly with the order structure and even be discontinuous: it can happen that there exists a bound $c\in\SS^+$ such that $a+a\geq c$ for every $a\in\SS^+$.
Since we aim to prove our impossibility theorem in full generality, we deliberately allow arbitrary monoids, including such ill-behaved examples.

The generality of this setting, however, requires us to impose additional conditions on the measures. These can be seen as analogues of properties that, in the classical case, follow from the algebraic structure of $\RR_\infty$ and the properties of $\sigma$-algebras (see, e.g., \cite[Theorem~1.19]{RudinRealAnalysisMR924157}). We incorporate a few of these properties directly into our definition. As customary in the literature, we write $x \in \MMM$ and $\mu(x)$ in place of $\{x\} \in \MMM$ and $\mu(\{x\})$.

\begin{definition}\label{def:axioms_measure}
Let $(X,\MMM)$ be a weakly $\lambda^+$-measurable space, and let $\SS$ be a positively totally ordered monoid with an infinitary sum $\Sum$. 
A \markdef{$\lambda^+$-measure} on $(X,\MMM)$ is a partial function $\mu:\powerset(X)\to (\SS,\Sum)$ such that $\MMM\subseteq \dom(\mu)$ and:
    \begin{enumerate}[label={\upshape (M\arabic*)}]      \item\label{ax_measure:emptyset} $\mu(\emptyset)=0$,
        \item\label{ax_measure:non-trivial} $\mu(X)>0$,
        \item\label{ax_measure:decreasing} $A\subseteq B$ implies $\mu(A)\leq \mu(B)$ for all $A,B\in \MMM$,
        \item\label{ax_measure:additive} for every family of disjoint sets $(A_i)_{i<\gamma}\subset \MMM$ of order type $\gamma< \lambda^+$, we have 
        \begin{itemize}
            \item $\langle \mu(A_i)\mid i<\gamma\rangle \in\dom(\Sum)$,
            \item $\mu(\bigcup_{i\in\gamma}A_i)=\Sum_{i\in\gamma}\mu(A_i)$,
        \end{itemize}
        \item\label{ax_measure:point-regular} $\mu$ is \markdef{point-regular}, i.e., for every $x\in X$ such that $x\in \MMM$ and for every (open) local basis $\{A_i\mid i<\gamma\}$ at $x$ of order type $\gamma< \lambda^{+}$, we have
        \[
        \mu(x)=\inf_{i<\gamma}\mu(A_i).
        \]
    \end{enumerate}
\end{definition}

We call $(X,\MMM, \mu)$ a \markdef{weak $\lambda^+$-measure space}. 
We say that $\mu$ is \markdef{continuous} if for every $x \in X$ we have $x \in \MMM$ and $\mu(x) = 0$.

It turns out that even under these weak conditions, it is already impossible to obtain a continuous measure on (any subspace of) the generalized Baire space.

\begin{theorem}\label{thm:impossibility_theorem}
    Let $\gamma$ be an ordinal, let $A$ be a set, and let $X\subseteq \pre{\gamma}{A}$ be a space of weight $\lambda\geq\ccc$. Then, there is no continuous weak $\lambda^+$-measure space $(X, \MMM,\mu)$.
\end{theorem}

In particular, if $\lambda^{<\kappa}=\lambda$ and $\lambda$ is at least the cardinality of the continuum $\ccc$, then there is no continuous weak $\lambda^+$-measure space $(\pre{\kappa}{\lambda},\MMM,\mu)$ on $\pre{\kappa}{\lambda}$ (Corollary~\ref{cor:impossibility_theorem_Baire}).

Working with a $\lambda^+$-measure and with a monoid of such a general form requires substantially different techniques from those used with classical measures on $\RR_\infty$. Therefore, in Sections~\ref{sec:setup}--\ref{sec:measures}, we analyze the basic properties of monoids, develop the definitions of measure and infinitary sum, and show how the problem can be reduced to a simpler setup.

In Section~\ref{sec:on-the-existence-of-trivial-measures}, we prove the optimality of the assumptions in Theorem~\ref{thm:impossibility_theorem}. 
In Subsection~\ref{subsec:Dirac}, we analyze the case $\lambda^{<\kappa}=\lambda < \ccc$, which implies $\kappa=\omega$. We show that in this case there are continuous weak $\lambda^+$-measure spaces $(X, \MMM, \mu)$ on $X = \pre{\omega}{\lambda}$ (Proposition~\ref{prop:existence_of_countable_induced_measure}). However, these measures, which we call $\lambda$-Dirac, are \enquote{trivial} in the sense that they are concentrated only on a small part of the space, in analogy with Dirac measures in the classical setting.
In Subsection~\ref{subsec:measures_with_non-continuous_sum} we show that if one further relaxes the assumptions on the monoid and allows an infinitary function $\Sum$ that is not continuous, then a continuous (non-$\lambda$-Dirac), weak $\lambda^+$-measure space $(X, \MMM, \mu)$ can again be obtained on $X = \pre{\kappa}{\lambda}$ even when $\lambda^{<\kappa} = \lambda\geq \ccc$ (Theorem~\ref{thm:possibility_theorem_non_continuous_sum}).
We argue, however, that measures obtained in this way are highly artificial, and that allowing a non-continuous function $\Sum$ generally produces pathological objects (see, e.g., Proposition~\ref{prop:trivial_measure_non_continuous_sum}); for this reason, we disregard this notion.

In Section~\ref{sec:on-the-non-existence-of-continuous-measures}, we prove Theorem~\ref{thm:impossibility_theorem}. In Proposition~\ref{prop:splitting_into_countable_induced_for_initially_Archimedean_monoids}, we establish a more technical impossibility result, showing that, even when $\lambda<\ccc$, every continuous measure can be decomposed into $\omega_1$-Dirac measures. This nicely complements Theorem~\ref{thm:impossibility_theorem} and the results of Subsection~\ref{subsec:Dirac}, showing that every continuous weak $\lambda^+$-measure is essentially obtained by starting from a countable measure on a separable subspace and then extending it to the whole space.
This rigidity phenomenon highlights the essential uniqueness of the classical countable case and reinforces the idea that the notion of measure fundamentally captures the separable situation.

In Section~\ref{sec:small_meaures_on_big_spaces}, we study continuous weak $\lambda$-measure spaces $(X,\MMM,\mu)$ on $X=\pre{\kappa}{\lambda}$.
We show that, if the structure $\SS$ is allowed to range over all positively totally ordered monoids, such $\lambda$-measures always exist (Proposition~\ref{prop:small_measures_big_space}).
In contrast, if $\SS$ is fixed to be the unit interval $[0,1]$, the existence of such $\lambda$-measure spaces is equivalent to the (real-valued-)measurability of $\lambda$ (Proposition~\ref{prop:real-valued_measurable_cardinal_iff_lambda_measure} and Corollary~\ref{cor:measurable_cardinal_iff_lambda_measure}).

Finally, in Section~\ref{sec:final_remarks}, we briefly discuss the consequences of our results.
In Corollaries~\ref{cor:impossibility_theorem_Baire}--\ref{cor:small_measures_big_space_Cantor}, we give a complete summary of the existence and non-existence of measures on the spaces $\pre{\kappa}{\lambda}$ and $\pre{\lambda}{2}$ across all possible cardinal scenarios.
We then show that, for $\lambda^+$-measure spaces, all results follow as consequences of the corresponding results for weak $\lambda^+$-measure spaces.
This allows us to extend our conclusions to a broader class of spaces, including $\lambda^+$-Borel spaces and more general topological spaces (that is, where $X$ is not necessarily a subset of $\pre{\kappa}{\lambda}$), showing for example the following.

\begin{corollary*}[{Corollary~\ref{cor:impossibility_for_all_top_spaces}}] 
    Assume $2^{<\lambda}=\lambda>\omega$. Let $X$ be a $T_0$ topological space of weight $\leq \lambda$. Then, there is no continuous $\lambda^+$-measure space $(X, \MMM,\mu)$.
\end{corollary*}

\begin{corollary*}[{Corollary~\ref{cor:impossibility_for_lambda-Borel_spaces}}]
    Assume $2^{<\lambda}=\lambda>\omega$. Let $(X,\MMM)$ be a $\lambda^+$-Borel space. Then, there is no continuous $\lambda^+$-measure on it.
\end{corollary*}

\section{Setup}\label{sec:setup}
We start by reviewing briefly the key notions we use throughout the paper.

\subsection{Topology and generalized descriptive set theory}\label{subsec:top_setup}
We refer the reader to \cite{EngelkingGenTopMR1039321} and to \cite{AndrettaMottoRosMR4409720, PhDThesisAgostini, AgosMottoSchlichtRegular} for the basic notions of, respectively, general topology and generalized descriptive set theory.

Recall that given a set $A$ and an ordinal $\gamma$, we denote by $\pre{\gamma}{A}$ the set of sequences of elements of $A$ of length $\gamma$. We also let $\pre{<\gamma}{A}=\bigcup_{\alpha<\gamma}\pre{\alpha}{A}$. 
The bounded topology on $\pre{\gamma}{A}$ is generated by the sets (called \markdef{cones}) $\Nbhd_s(\pre{\gamma}{A})=\{x\in \pre{\gamma}{A}\mid s\subseteq x\}$ for $s\in \pre{<\gamma}{A}$. 
For simplicity, we often just write $\Nbhd_s$ when the ambient space is clear from the context.
The family $\B=\{\Nbhd_s\mid s\in \pre{<\gamma}{A}\}$ consists of clopen sets and is closed under intersections of size less than $\cof(\gamma)$.
Therefore, $\B$ is a basis for the bounded topology on $\pre{\gamma}{A}$; we call it the \markdef{canonical basis}.
For any subspace $X\subseteq \pre{\gamma}{A}$, we also denote by $\B(X)$ the canonical basis of $X$ consisting of the clopen cones $\Nbhd_s(X)=\Nbhd_s\cap X$ such that $\Nbhd_s\cap X\neq\emptyset$. 
The closure of $\B$ (and $\B(X)$) under intersections of size less than $\cof(\gamma)$ ensures that the bounded topology itself is closed under intersections of size less than $\cof(\gamma)$, that is, it is \textbf{$\cof(\gamma)$-additive}.

The following well-known fact (see, e.g., \cite[Theorem~2.7]{NyikosTreeBasesMR1666807} or \cite[Proposition~2.2.24]{PhDThesisAgostini}) shows that the bounded topology is also zero-dimensional in a strong sense: every open cover can be refined by a clopen partition consisting of basic clopen cones.

\begin{fact}\label{fct:strong_zero_dimensionality}
     Let $\gamma$ be an ordinal, let $A$ be a set, and let $X\subseteq \pre{\gamma}{A}$. For every family $\V$ of open sets of $X$ there is a family $\P\subseteq \B(X)$ of disjoint cones which refines $\V$ and satisfies that $\bigcup \P=\bigcup \V$.
\end{fact}

Recall that, given a cardinal $\delta$, a space is called \markdef{$\delta$-Lindel\"of} if every open cover has a subcover of size strictly less than $\delta$.
From Fact~\ref{fct:strong_zero_dimensionality} we can obtain the following immediate corollary.

\begin{corollary}\label{cor:equivalent_conditions_lindelof}
 Let $\gamma$ be an ordinal, let $A$ be a set, and let $X\subseteq \pre{\gamma}{A}$. The following are equivalent:
\begin{enumerate-(1)}
    \item\label{cor:equivalent_conditions_lindelof-1} $X$ is not $\lambda$-Lindel\"of.
    \item\label{cor:equivalent_conditions_lindelof-2} $X$ contains a clopen partition of size at least $\lambda$.
    \item\label{cor:equivalent_conditions_lindelof-3} $X$ contains a closed discrete set of size (at least) $\lambda$.
\end{enumerate-(1)}
\end{corollary}

\begin{proof}
    \ref{cor:equivalent_conditions_lindelof-1}$\Rightarrow$\ref{cor:equivalent_conditions_lindelof-2}. Let $\U$ be an open cover that does not contain any subcover of size smaller than $\lambda$. Let $\P$ be a partition of $X$ refining $\U$ as given by Fact~\ref{fct:strong_zero_dimensionality}. If $|\P|<\lambda$, then choosing $U_P\in \U$ such that $P\subseteq U_P$ for every $P\in \P$ would lead to a contradiction. Therefore, $|\P|\geq\lambda$.

    \ref{cor:equivalent_conditions_lindelof-2}$\Rightarrow$\ref{cor:equivalent_conditions_lindelof-3}. Let $\P$ be a partition of $X$ of size (at least) $\lambda$, and choose $x_p\in P$ for every $P\in \P$. Then $\{x_P\mid P\in\P\}$ is closed and discrete, as wanted.

    \ref{cor:equivalent_conditions_lindelof-3}$\Rightarrow$\ref{cor:equivalent_conditions_lindelof-1}. It is enough to notice that every closed subspace of a $\lambda$-Lindel\"of space is again $\lambda$-Lindel\"of, while a discrete space of size $\lambda$ is not $\lambda$-Lindel\"of.
\end{proof}

Recall that the weight $\weight$ of a topological space is the minimal size of a basis for its topology; the density $\density$ is the minimal size of a dense subset; and the cellularity $\cellularity$ is the minimal cardinal $\delta$ such that every family of pairwise disjoint open sets has size at most $\delta$.

It is well known (see, e.g., \cite[Exercise 1.7.12(a)]{EngelkingGenTopMR1039321}) that for every topological space $X$ we have $\cellularity(X)\leq \density(X)\leq \weight(X)$.
In any subspace $X\subseteq \pre{\gamma}{A}$ with the bounded topology, these three cardinals coincide.

\begin{proposition}\label{prop:weight=cellularity}
    Let $\gamma$ be an ordinal, let $A$ be a set, and let $X\subseteq \pre{\gamma}{A}$. If $X$ has weight $\lambda$, then it contains a family of pairwise disjoint basic clopen sets of size $\lambda$.
    
    In particular, $\cellularity(X)= \density(X)= \weight(X)$.
\end{proposition}

\begin{proof}
If $X$ is discrete we are done, so suppose not. Then, $\gamma$ is limit and $\kappa=\cof(\gamma)$ is infinite, and we must have $\kappa\leq\lambda$, since $X$ is $\kappa$-additive and not discrete. 

If $\lambda=\kappa$ we are done: for any non-isolated point $x\in X$, we have that for cofinally many $\alpha<\gamma$ the set $\Nbhd_{x\restriction \alpha}(X) \setminus \Nbhd_{x\restriction \alpha+1}(X)$ must be non-empty. Let $\{\alpha_i\mid i<\kappa\}$ be an enumeration of (some of) these $\alpha$s. Since $\B(X)$ is a basis for $X$ and $\Nbhd_{x\restriction \alpha+1}(X) \setminus \Nbhd_{x\restriction \alpha}(X)$ is clopen, for every $i<\kappa$ there is $s_i\in \pre{<\kappa}{\lambda}$ such that $\Nbhd_{s_i}(X)\neq\emptyset$ and $\Nbhd_{s_i}(X)\subseteq \Nbhd_{x\restriction {\alpha_i}}(X) \setminus \Nbhd_{x\restriction {\alpha_i+1}}(X)$, which gives a family $\{\Nbhd_{s_i}(X)\mid i<\kappa\}$ of non-empty disjoint basic clopen sets of size $\kappa=\lambda$, as wanted.

Therefore, assume $\kappa<\lambda$. Let $\langle \gamma_\alpha\mid \alpha<\kappa\rangle$ be a cofinal sequence in $\gamma$.
Then, the family 
\[
\tilde{\B}(X)=\{\Nbhd_s(X)\mid s\in \bigcup_{\alpha<\kappa} \pre{{\gamma_\alpha}}{A}\}\subseteq \B(X)
\]
is still a basis for $X$.
Since $X$ has weight $\lambda$, we can find $T\subseteq \bigcup_{\alpha<\kappa} \pre{{\gamma_\alpha}}{A}$ such that $\Nbhd_s(X)=\Nbhd_s\cap X\neq \emptyset$ for every $s\in T$ and 
\[
\B'=\{\Nbhd_s(X)\mid s\in T\}\subseteq \tilde{\B}(X)
\]
is a basis for $X$ of minimal size $|\B'|=|T|=\lambda$.
Without loss of generality, we may assume that $s\restriction \gamma_\alpha\in T$ for every $s\in T$ and $\alpha<\kappa$.
For every $\alpha<\kappa$, let $\Lev_\alpha(T)=\pre{\gamma_\alpha}{A}\cap T$. If there is $\alpha<\kappa$ such that $|\Lev_\alpha(T)|=\lambda$ we are done, so suppose not. 
Then, $| \Lev_\alpha(T)|<\lambda$ for every $\alpha<\kappa$, and $\lambda=\sup_{\alpha<\kappa} | \Lev_\alpha(T)|$, thus $\cof(\lambda)=\kappa<\lambda$.

Proceeding recursively, we can build a cofinal set $I\subseteq \kappa$ such that 
\[
\kappa<|\Lev_\alpha(T)|<| \Lev_\beta(T)|<\lambda
\]
for every $\alpha,\beta\in I$ satisfying $\alpha<\beta$. Let $I=\{\alpha_i\mid i<\kappa\}$ be an increasing enumeration of it.
Since for every $s\in T$ and $\alpha<\kappa$ we have $s\restriction \gamma_\alpha\in T$, by the pigeonhole principle for every $i<\kappa$ there is $s(i)\in \Lev_{\alpha_i}(T)$ such that 
\[
|\{t\in \Lev_{\alpha_{i+1}}(T)\mid s(i)\subseteq t\}| >\Lev_{\alpha_i}(T).
\]
Since we assumed $|\Lev_{\alpha_{i}}(T)|>\kappa$ and since for every $i,j<\kappa$ there is at most one $t\in \Lev_{\alpha_{i+1}}(T)$ such that $t\subseteq s(j)$, we have that 
\[
Z_i=\{t\in \Lev_{\alpha_{i+1}}(T)\mid s(i)\subseteq t\text{ and for every } j<\kappa \text{ we have } t\nsubseteq s(j) \}
\]
has still size $|Z_i|>|\Lev_{\alpha_i}(T)|$ for every $i<\kappa$. Since $\lambda=\sup_{\alpha<\kappa} | \Lev_\alpha(T)|$, we have that $|\bigcup_{i<\kappa} Z_i|=\lambda$, and thus 
\[
\{\Nbhd_t(X)\mid t\in \bigcup_{i<\kappa} Z_i\}
\]
is a family of disjoint basic clopen sets of size $\lambda$, as wanted.
\end{proof}

Generalized descriptive set theory at a cardinal $\lambda$ usually works with (subspaces of) the spaces $\pre{\lambda}{2}$ and $\pre{\cof(\lambda)}{\lambda}$, equipped with the bounded topology.
It is customary in this setting to assume $2^{<\lambda}=\lambda$, which implies the weaker $\lambda^{<\cof(\lambda)}=\lambda$, in order to ensure, for example, that $\pre{\lambda}{2}$ and $\pre{\cof(\lambda)}{\lambda}$ both have weight $\lambda$.

Here, we adopt a slightly more general framework: we work with any space of weight $\lambda$ of the form $X\subseteq\pre{\gamma}{A}$, for some set $A$ and some ordinal $\gamma$, equipped with the bounded topology it inherits from $\pre{\gamma}{A}$.
This class of spaces is quite broad. For example, it includes all Lebesgue zero-dimensional metric spaces (equivalently, all ultrametric spaces) \cite[Theorem 7.3.15 and Exercise 7.3.F]{EngelkingGenTopMR1039321} and all $\mu$-metrizable spaces for $\mu>\omega$ \cite[Theorem~3.3]{NyikosTreeBasesMR1666807} of the required weight, among others (see also \cite[Theorem~2.2.1]{PhDThesisAgostini}).

Thanks to Proposition~\ref{prop:weight=cellularity}, it is evident why in this context of a topological space $X\subseteq \pre{\gamma}{A}$ of weight $\lambda$, the appropriate notion of measure is that of $\lambda^+$-measures: any such space contains an open set that can be partitioned into $\lambda$ many open sets, making any $\gamma$-measure for $\gamma<\lambda^+$ not informative enough in this context. 
This is particularly relevant when studying the space $\pre{\kappa}{\lambda}$ without the assumption $\lambda^{<\kappa}=\lambda$, and it shows that the correct objects to study on $\pre{\kappa}{\lambda}$ are $(\lambda')^+$-measures, rather than $\lambda^+$-measures, for $\lambda'=\lambda^{<\kappa}$ the weight of $\pre{\kappa}{\lambda}$.

Operationally, our setting can be reduced to a more familiar one, thanks to the following well-known fact (see, e.g., \cite[Proposition~2.2.33]{PhDThesisAgostini}).

\begin{fact}\label{fct:homeo_to_subset_Baire}
    Let $\gamma$ be an ordinal, let $A$ be a set, and let $X\subseteq \pre{\gamma}{A}$. Let $\lambda$ be the weight of $X$ and let $\kappa=\cof(\gamma)$. Then, $X$ is homeomorphic to a subspace of $\pre{\kappa}{\lambda}$.
\end{fact}

When $X$ is not discrete, the cardinal $\kappa$ obtained in this way is unique and satisfies $\kappa\leq\lambda$. Indeed, recall that the \markdef{character} of a topological space is the smallest cardinal $\kappa$ such that every point of $X$ has a local base of size at most $\kappa$, and the \markdef{additivity} of a topological space is the unique cardinal $\kappa$ such that the topology is $\kappa$-additive but not $\kappa^+$-additive\footnote{Equivalently, the smallest cardinal $\kappa$ such that the topology is not closed under intersections of size $\kappa$.}.
Every non-discrete subspace of a space of the form $\pre{\gamma}{A}$ has both character and additivity equal to $\cof(\gamma)$. Since these notions are invariant under homeomorphisms, it follows that if $X$ is a non-discrete space that embeds both into $\pre{\gamma}{A}$ and into $\pre{\delta}{B}$, then $\cof(\gamma)=\cof(\delta)$, and this number coincides with the character and additivity of $X$.
Moreover, since every $\kappa$-additive space of weight less than $\kappa$ is discrete, we also get $\kappa\leq\lambda$ for every non-discrete space $X$ of weight $\lambda$ and additivity $\kappa$.

On the other hand, if $X\subseteq \pre{\gamma}{A}$ is discrete, then $X$ embeds into $\pre{\omega}{\lambda}$, where $\lambda$ is the weight of $X$.

Putting these observations together, we obtain the following corollary of Fact~\ref{fct:homeo_to_subset_Baire}.

\begin{corollary}\label{cor:embedding_into_Baire_kappa_leq_lambda}
    Let $\gamma$ be an ordinal, let $A$ be a set, and let $X\subseteq \pre{\gamma}{A}$. Let $\lambda$ be the weight of $X$. Then, there is a regular cardinal $\kappa\leq \lambda$ such that $X$ is homeomorphic to a subspace of $\pre{\kappa}{\lambda}$.
\end{corollary}

In particular, without loss of generality we may restrict attention to subspaces $X\subseteq \pre{\kappa}{\lambda}$ of weight $\lambda$, where $\kappa$ is regular and $\kappa\leq\lambda$.
For completeness, and in line with standard practice in generalized descriptive set theory, we also allow spaces of weight at most $\lambda$, although most results of this paper become trivial when considering a $\lambda^+$-measure on a space of weight strictly less than $\lambda$.

\begin{center}
\textit{For the rest of the paper, we use the symbols $\lambda$ and $\kappa$ to denote infinite cardinals such that $\kappa$ is regular and $\kappa\leq \lambda$.}
\end{center}

\begin{center}
\textit{We denote by $(X,\tau)$ (or $X$ for short) a subspace of $\pre{\kappa}{\lambda}$ of weight (at most) $\lambda$, always equipped with the bounded topology.}
\end{center}

In the rest of this subsection, we compare our setup with those adopted in the literature, with particular emphasis on the generalized Baire and Cantor spaces $\pre{\kappa}{\lambda}$ and $\pre{\lambda}{2}$, and on the role and necessity of cardinal assumptions $2^{<\lambda}=\lambda$ and $\lambda^{<\cof(\lambda)}=\lambda$, and most importantly, $\kappa=\cof(\lambda)$.

Classically, generalized descriptive set theory at a cardinal $\lambda$ has focused on the study of the spaces $\pre{\lambda}{\lambda}$ and $\pre{\lambda}{2}$ under the assumption $\lambda^{<\lambda}=\lambda$; this condition is equivalent to $2^{<\lambda}=\lambda$ and $\lambda$ regular.
More recent work has aimed at extending the theory to the spaces $\pre{\cof(\lambda)}{\lambda}$ and $\pre{\lambda}{2}$ for arbitrary cardinals $\lambda$ (possibly singular) satisfying $2^{<\lambda}=\lambda$.
This condition implies the weaker $\lambda^{<\cof(\lambda)}=\lambda$.
Since the space $\pre{\cof(\lambda)}{\lambda}$ has weight $\lambda^{<\cof(\lambda)}$, and the space $\pre{\lambda}{2}$ has weight $2^{<\lambda}$, this setup also ensures that both spaces have weight $\lambda$.

In this framework, the same cardinal $\lambda$ plays several distinct roles simultaneously.
First, it determines the \enquote{dimension}, i.e., the weight, of the spaces under consideration.
Second, it fixes the appropriate \enquote{size} of the associated notions to be studied, such as $\lambda^+$-Borel sets, $\lambda$-compactness, $\lambda^+$-measures, and related concepts.
Finally, together with its cofinality, it is also used to define the base set and the exponent of the spaces $\pre{\cof(\lambda)}{\lambda}$ and $\pre{\lambda}{2}$ themselves.

This identification of roles is precisely what made it necessary to adopt additional cardinal assumptions.
As argued after Proposition~\ref{prop:weight=cellularity} in the context of measures, the right notions for a topological space of the form $X\subseteq \pre{\gamma}{A}$ (and in particular for $\pre{\cof(\lambda)}{\lambda}$) are arguably those defined relative to its weight.
However, this cardinal is a priori independent of the cardinals appearing in the definition of $X$, and forcing them to coincide amounts to assuming cardinal equalities such as $\lambda^{<\cof(\lambda)}=\lambda$ and $2^{<\lambda}=\lambda$.

The setup we adopt instead is strictly more general and allows us to work without cardinal assumptions.
First, when $\kappa=\cof(\lambda)$ and $2^{<\lambda}=\lambda$, we recover the usual setting described above, since in this case the space $\pre{\kappa}{\lambda}$ itself (and all its subspaces) has weight at most $\lambda$, and $\pre{\lambda}{2}$ is homeomorphic to a subspace of $\pre{\kappa}{\lambda}$.
However, we do not require that $\kappa=\cof(\lambda)$, nor that $2^{<\lambda}=\lambda$ or $\lambda^{<\cof(\lambda)}=\lambda$. As a consequence, we can study the spaces $\pre{\lambda}{2}$ and $\pre{\cof(\lambda)}{\lambda}$ and their subspaces even when both cardinal assumptions $2^{<\lambda}=\lambda$ and $\lambda^{<\cof(\lambda)}=\lambda$ fail, a situation that is usually not covered in generalized descriptive set theory.

Indeed, $\pre{\cof(\kappa)}{\lambda}$ always contains subspaces of weight at most $\lambda$, regardless of whether $\lambda^{<\cof(\kappa)}=\lambda$ holds or not, and such spaces need not be discrete.
One example is the subspace $X_0$ of $\pre{\lambda}{2}$ described in \cite[Equation~2.3]{AgosMottoSchlichtRegular} (or, similarly, the subspace of $\pre{\cof(\lambda)}{\lambda}$ consisting of sequences that are constant except at finitely many coordinates).
A more interesting example is the body $[T]$ of a $\lambda$-Kurepa tree $T$ (when such a tree exists), since the definition of a $\lambda$-Kurepa tree guarantees that $[T]$ has weight at most $\lambda$, independently of whether $\lambda^{<\cof(\kappa)}=\lambda$ holds.
The results developed in our framework apply to all these spaces.

Moreover, our setup also allows us to analyze the space $\pre{\cof(\lambda)}{\lambda}$ itself even when the relevant cardinal assumption $\lambda^{<\cof(\lambda)}=\lambda$ fails.
This already follows from Fact~\ref{fct:homeo_to_subset_Baire}, which shows that this space can always be reduced to a subspace of some $\pre{\cof(\kappa)}{{\lambda'}}$, where $\lambda'=\lambda^{<\cof(\lambda)}$ is the weight of $\pre{\cof(\lambda)}{\lambda}$.
Notice that this new $\lambda'$ need not satisfy $\kappa=\cof(\lambda')$, and thus we crucially exploit the fact that the cardinal $\kappa$ is allowed to differ from $\cof(\lambda)$.

In practice, the situation is even simpler and more elegant.
Once $\kappa\neq\cof(\lambda)$ is allowed, up to homeomorphism the space $\pre{\cof(\lambda)}{\lambda}$, and more generally any space of the form $\pre{\kappa}{\lambda}$, can be identified with a space of the form $\pre{\kappa}{{\lambda'}}$ for which the relevant cardinal assumption $(\lambda')^{<\kappa}=\lambda'$ holds, as shown by the following result.

\begin{proposition}\label{prop:lambda<kappa=lambda_is_necessary}
    Let $\lambda'=\lambda^{<\kappa}$. Then, $(\lambda')^{<\kappa}=\lambda'$, and the space $\pre{\kappa}{\lambda}$ is homeomorphic to $\pre{\kappa}{{\lambda'}}$.
\end{proposition}

\begin{proof}
    Let $\lambda'=\sup\{\lambda^{\gamma}\mid \gamma<\kappa \text{ cardinal}\}$. If $\lambda'=\lambda^{\gamma}$ for some $\gamma<\kappa$, then \[
    (\lambda')^{<\kappa}=\sup\{(\lambda^{\gamma})^{\delta}\mid \delta<\kappa \text{ cardinal}\}=\sup\{\lambda^{\delta}\mid \delta<\kappa \text{ cardinal}\}=\lambda'
    \]
    as wanted.
    Assume instead that $\lambda^{\gamma}<\lambda'$ for every $\gamma<\kappa$. Then the family $\{\lambda^{\gamma}\mid \gamma<\kappa \text{ cardinal}\}$ is cofinal in $\lambda'$, so in particular $\cof(\lambda')=\kappa$ and $\delta^\gamma<\lambda'$ for every $\gamma<\kappa$ and $\delta<\lambda'$.
    Then, by \cite[Theorem~5.20]{jechSetTheory2003} we have $(\lambda')^\gamma=\lambda'$ for every $\gamma<\kappa$, as wanted.

    To prove that $\pre{\kappa}{\lambda}$ is homeomorphic to $\pre{\kappa}{{\lambda'}}$, assume $\kappa\leq \lambda$.
    \begin{claim}\label{claim:big_clopen_partition}
        For every $s\in \pre{<\kappa}{\lambda}$ there is a partition $\P_s\subseteq\B$ of $\Nbhd_s(\pre{\kappa}{{\lambda}})$ of size $\lambda'$ made of basic clopen sets.
    \end{claim}

    \begin{proof}
        If $\lambda'=\lambda^{\gamma}$ for some $\gamma<\kappa$, then 
        \[
        \P_s=\{\Nbhd_{s\conc t}\mid t\in \pre{\gamma}{\lambda}\}
        \]
        is as wanted.
        Otherwise, assume that the family $\{\lambda^{\gamma}\mid \gamma<\kappa \text{ cardinal}\}$ is cofinal in $\lambda'$, and let $\{\gamma_\alpha\mid \alpha<\kappa\}$ be an enumeration of the cardinals below $\kappa$. For every $\alpha<\kappa$, let $\delta_\alpha=\lambda^{\gamma_\alpha}$.
        Define
        \[
        \P_s=\{\Nbhd_{s\conc \langle \alpha\rangle }\mid \kappa\leq \alpha<\lambda\}\cup \{\Nbhd_{s\conc \langle \alpha\rangle\conc t }\mid t\in \pre{\gamma_\alpha}{\lambda}, \alpha<\kappa\}.
        \]
        Since $\kappa\leq \lambda$, the family $\P_s$ has size $|\P_s|=\sup_{\alpha<\kappa} \delta_\alpha=\lambda'$, and it is a cover of $\Nbhd_s(\pre{\kappa}{{\lambda}})$ made of basic clopen sets by construction, as wanted. 
    \end{proof}
    
    Then, applying Claim~\ref{claim:big_clopen_partition} recursively we can obtain a family $\{\P_\alpha\mid \alpha<\kappa\}$ of partitions of $\pre{\kappa}{{\lambda}}$ made of basic clopen sets such that
    \begin{itemize}
        \item $\hat{\B}=\bigcup_{\alpha<\kappa}\P_\alpha$ is a basis for $\pre{\kappa}{{\lambda}}$,
        \item $\P_\beta$ refines $\P_\alpha$ for each $\alpha<\beta<\kappa$, and
        \item $\{P'\in \P_{\alpha+1}\mid P\subseteq P'\}$ has size $\lambda'$ for every $\alpha<\kappa$ and $P\in \P_\alpha$.
    \end{itemize}
    Then, it is clear that $(\hat{\B}, \supseteq)$ is isomorphic as a tree to $\pre{<\kappa}{{\lambda'}}$, and this isomorphism induces a homeomorphism of topological spaces between $\pre{\kappa}{\lambda}$ and $\pre{\kappa}{{\lambda'}}$ (see also \cite[Fact 2.1.5]{PhDThesisAgostini} or  \cite[Lemma~2.1]{HungNegrepontis2MR370463}), as wanted.
\end{proof}

Notice that Proposition~\ref{prop:lambda<kappa=lambda_is_necessary} relies on our implicit hypothesis $\kappa\leq \lambda$. 
If $\lambda=\omega$ and $\kappa>\omega$ is weakly compact (see \cite[Definition~9.8]{jechSetTheory2003}), then $2^{<\kappa}=\kappa>\lambda$; however, the space $\pre{\kappa}{\omega}$ (which is homeomorphic to $\pre{\kappa}{2}$) is not homeomorphic to $\pre{\kappa}{{\lambda'}}$ for any $\lambda'\geq \kappa$, since the former is $\kappa$-Lindel\"of whereas the latter is not (see, e.g., \cite{NyikosTreeBasesMR1666807}). 
On the other hand, if $\kappa$ is not weakly compact, then $\pre{\kappa}{\lambda}$ is homeomorphic to $\pre{\kappa}{\kappa}$ for every $\lambda$ satisfying $2\leq \lambda\leq \kappa$; thus in this case Proposition~\ref{prop:lambda<kappa=lambda_is_necessary} also holds without the hypothesis $\kappa\leq \lambda$.

We can obtain an analogue of Proposition~\ref{prop:lambda<kappa=lambda_is_necessary} for the space $\pre{\lambda}{2}$ as well.
First, it is well known that while $\pre{\lambda}{2}$ may not itself be a subspace of $\pre{\kappa}{\lambda}$, under the assumption $2^{<\lambda}=\lambda$ it is always homeomorphic to a subspace of $\pre{\kappa}{\lambda}$ (see, e.g., \cite[Proposition~2.2.54]{PhDThesisAgostini}).

\begin{fact}\label{fct:Cantor_subset_Baire_with_cardinal_assumption}
    Assume $2^{<\lambda}=\lambda$. Then $\pre{\lambda}{2}$ is homeomorphic to a closed subspace of $\pre{\cof(\lambda)}{\lambda}$.
\end{fact}

Following Proposition~\ref{prop:lambda<kappa=lambda_is_necessary}, we can obtain a similar result regardless of the assumption $2^{<\lambda}=\lambda$.

\begin{proposition}\label{prop:Cantor_always_subset_of_some_Baire}
    Let $\lambda'=2^{<\lambda}$. Then, $(\lambda')^{<\cof(\lambda)}=\lambda'$, and the space $\pre{\lambda}{2}$ is homeomorphic to a closed subspace of $\pre{\cof(\lambda)}{{\lambda'}}$.
\end{proposition}

\begin{proof}
    Let $\lambda'=\sup\{2^{\gamma}\mid \gamma<\lambda \text{ cardinal}\}$. If $\lambda'=2^{\gamma}$ for some $\gamma<\lambda$, then 
    \[
    \lambda'\leq (\lambda')^{<\cof(\lambda)}=\sup\{(2^{\gamma})^{\delta}\mid \delta<\cof(\lambda) \text{ cardinal}\}\leq \sup\{2^{\delta}\mid \delta<\lambda \text{ cardinal}\}=\lambda'
    \]
    as wanted.
    Assume instead that $2^{\gamma}<\lambda'$ for every $\gamma<\lambda$. Then the family $\{2^{\gamma}\mid \gamma<\lambda \text{ cardinal}\}$ is cofinal in $\lambda'$, so in particular $\cof(\lambda')=\cof(\lambda)$ and $\delta^\gamma<\lambda'$ for every $\gamma<\cof(\lambda)$ and $\delta<\lambda'$.
    Then, by \cite[Theorem~5.20]{jechSetTheory2003} we have $(\lambda')^\gamma=\lambda'$ for every $\gamma<\cof(\lambda)$.
    
    The remaining part of the proof proceeds essentially as in the usual argument used to embed the Cantor space into the Baire space.
    Let $\kappa=\cof(\lambda)$ and let $(\lambda_i)_{i<\kappa}$ be a cofinal sequence of ordinals in $\lambda$. Then, the family $\hat{\B}=\bigcup_{i<\kappa}\{\Nbhd_s(\pre{\lambda}{2})\mid s\in \pre{{\lambda_i}}{2}\}$ is a basis for $\pre{\lambda}{2}$. Since $|\{\Nbhd_s(\pre{\lambda}{2})\mid s\in \pre{{\lambda_i}}{2}\}|\leq \lambda'$ for every $i<\kappa$, there is a tree $T\subseteq\pre{<\kappa}{{\lambda'}}$ closed under initial segment such that $(\hat{\B},\supseteq)$ is isomorphic (as a tree) to $T$. 
    This isomorphism induces a homeomorphism of topological spaces between $\pre{\lambda}{2}$ and $[T]\subseteq\pre{\kappa}{{\lambda'}}$ (see, e.g., \cite[Fact 2.1.5]{PhDThesisAgostini} or \cite[Lemma~2.1]{HungNegrepontis2MR370463}), as wanted.
\end{proof}

This will allow us, in Section~\ref{sec:final_remarks}, to extend our work to the generalized Cantor space as well.

\subsection{Totally ordered monoids}\label{subsec:totally-ordered-monoids}

We now recall some basic notions and definitions about totally ordered monoids. Our main reference here is \cite[Chapter X]{fuchsPartiallyOrderedAlgebraic1963}.

A  \markdef{totally ordered monoid} is a structure $\SS=( S, 0, {+}, {\leq} )$ such that
\begin{enumerate}[label=(\arabic*)]
    \item $+$ is a binary associative operation on $\SS$ (i.e., $( S, {+})$ is a semigroup),
    \item $0$ is the neutral element of the operation $+$ (i.e., $( S, {+}, 0)$ is a monoid),
    \item $( S, {\leq} )$ is a total order,
    \item the order ${\leq}$ is (weakly) translation invariant (i.e., for every $a,b,c,d\in \SS$, if $a\leq b$ and $c\leq d$, then $a + c \leq b + d$).\label{ax_sum:weakly_translation_invariant}
\end{enumerate}

A \markdef{totally ordered group} $\GG = ( G, 0, +, \leq )$ is a totally ordered monoid such that $(G, +)$ is a group.

The \markdef{positive cone} $\SS^+$ is defined as $\SS^+=\{s\in \SS\mid 0< s\}$, and the \imarkdef{degree of \( \SS \)}, denoted by $\Deg(\SS)$, is the coinitiality of the positive cone $\SS^+$ with respect to \( \leq \), i.e. the minimal size of a set $A\subseteq \SS^+$ such that for every $\varepsilon\in \SS^+$ there is some $a\in A$ satisfying $a\leq \varepsilon$. By definition, \(\Deg(\SS) \) is always a regular cardinal. We say that a totally ordered monoid $\SS$ is \markdef{positively ordered} if $\SS^+\neq\emptyset$ and $\SS=\SS^+\cup\{0\}$, i.e. if $0$ is the minimum of the order and $\SS$ is not the trivial monoid $\{0\}$. Notice that for every totally ordered monoid $\SS$, we have that $\SS^+\cup \{0\}$ is a positively totally ordered monoid (or the trivial monoid $\{0\}$). Since for the purpose of this paper we only work with $\SS^+\cup \{0\}$, we can restrict our attention to positively ordered monoids.

A subset $A\subseteq \SS$ of a positively totally ordered monoid $\SS$ is called \markdef{Archimedean} if for every $a,b\in A\setminus \{0\}$ such that $a<b$ there is $n\in\omega$ such that $n\cdot a\geq b$, where $n\cdot a=a+ ... +a$ denotes the sum of $n$-many elements equal to $a$. 
We call $\SS$ Archimedean if $A=\SS$ is Archimedean.
For $a,b\in \SS$ with $a<b$, we write $a\ll b$ when, instead, $n\cdot a< b$ for every $n\in\omega$. This way, a (positively totally ordered) monoid is Archimedean if and only if $a\ll b$ implies $a=0$.

\begin{definition}
    A positively totally ordered monoid $\SS$ is said to be \markdef{initially Archimedean} if there is $a\in \SS^+$ such that the open interval $(0,a)$ is Archimedean. 
\end{definition}

\begin{definition}
    Let $\SS=(S, 0, +, \leq)$ be a positively totally ordered monoid. Let $\bar{a}=\langle a_i\mid i<\gamma\rangle$ be a sequence of elements in $\SS$. We say that $\bar{a}$ is:
    \begin{enumerate-(1)}
    \item \textbf{decreasing} if $a_i\leq a_j$ for all $j<i<\gamma$,\label{def:decreasing}
    \item \textbf{strictly decreasing} if $a_i<a_j$ for all $j<i<\gamma$,\label{def:strictly-decreasing}
    \item \textbf{strongly decreasing} if $a_i+a_i<a_j$ for all $j<i<\gamma$,\label{def:strongly-decreasing}
    \item \textbf{nowhere Archimedean (decreasing)} if $a_i\ll a_j$ for all $j<i<\gamma$.\label{def:nowhere-archimedean}
    \end{enumerate-(1)}
\end{definition}

\begin{remark}
    Note that, for a sequence $\bar{a}$ of elements in a positively totally ordered monoid $\SS$, we have 
$
\text{\ref{def:nowhere-archimedean}} \Rightarrow
\text{\ref{def:strongly-decreasing}} \Rightarrow
\text{\ref{def:strictly-decreasing}} \Rightarrow
\text{\ref{def:decreasing}}.
$
\end{remark}

The following lemma shows that every monoid is either pathological, initially Archimedean of countable degree, or non-Archimedean in a strong way.

\begin{lemma}\label{lem:three_types_of_monoids}
    Let $\SS$ be a positively totally ordered monoid.
    Then, exactly one of the following holds:
    \begin{enumerate-(1)}
        \item\label{lem:three_types_of_monoids - pathological} There is $c\in \SS^+$ such that $b+b\geq c$ for every $b\in \SS^+$.
        \item\label{lem:three_types_of_monoids - Archimedean} $\Deg(\SS)=\omega$ and $\SS^+$ contains a countable strongly decreasing Archimedean sequence $\bar{a}=\langle a_i\mid i<\omega\rangle$ coinitial in $\SS^+$.  
        \item\label{lem:three_types_of_monoids - non-Archimedean} $\Deg(\SS)=\delta\geq \omega$ and $\SS^+$ contains a nowhere Archimedean decreasing sequence $\bar{a}=\langle a_i\mid i<\delta\rangle$ coinitial in $\SS^+$.
    \end{enumerate-(1)}
\end{lemma}

\begin{proof}
    Assume that we are not in the first case. Then, for every $c\in \SS^+$ there is $b\in \SS^+$ such that $b+b<c$. This implies that $\delta=\Deg(\SS)\geq \omega$ is infinite (and regular). Let thus $\bar{a}=\langle a_i\mid i<\delta\rangle$ be coinitial in $\SS^+$, with $\delta=\Deg(\SS)$.
    
    We recursively define a strongly decreasing subsequence of $\bar{a}$. Let $\alpha_0=0$. Now consider $i<\delta$, and assume that we have already defined $\alpha_j\geq j$ for all $j<i$. Let $\gamma=\sup\{\alpha_j\mid j<i\}$. We have that $i\leq \gamma<\delta$, by the regularity of $\delta$. Since $\gamma<\delta=\Deg(\SS)$, there is $c\in \SS^+$ such that $c\leq a_j$ for every $j\leq \gamma$. Thus, by assumption, we can find some $b\in \SS^+$ such that $b+b< c$. Now, if it exists, let $b_\gamma$ be such that $b_\gamma\ll c$. Otherwise, let $b_\gamma=b$. Since $\bar{a}$ is coinitial in $\SS^+$, there is an $\alpha_i<\delta$ such that $a_{\alpha_i}\leq b_\gamma$, and thus $a_{\alpha_i} + a_{\alpha_i} \leq  b_\gamma+b_\gamma <a_j$ for every $j\leq \gamma$. This also implies $\alpha_i>\gamma$. It follows that the sequence $\bar{a}'=\langle a_{\alpha_i}\mid i<\delta\rangle$ obtained in this way is strongly decreasing and coinitial. 

    Now suppose first that 
    for every $c\in \SS^+$ there is $b\in \SS^+$ such that $b\ll c$.
    By construction, this implies that $\bar{a}'=\langle a_{\alpha_i}\mid i<\delta\rangle$ is nowhere Archimedean, as wanted.

    Conversely, assume that there is $c\in \SS^+$ such that for all $b\in \SS^+$ we have $n\cdot b\geq c$ for some $n\in\omega$.
    This implies that the whole interval $[0,c]$ is Archimedean.
    Notice that if $\bar{a}'=\langle a_{\alpha_i}\mid i<\delta\rangle$ is a strongly decreasing sequence converging to $0$ for an uncountable cardinal $\delta$, then $\bar{a}'=\langle a_{\alpha_i}\mid i<\delta, i\text{ limit}\rangle$ is a nowhere Archimedean sequence of the same length converging to $0$. Thus, in order for $[0,c]$ to be Archimedean, we must have $\Deg(\SS)=\delta=\omega$.
    Let $j<\delta$ be such that $a_{\alpha_j}<c$. Then, we have that $\bar{a}'=\langle a_{\alpha_{j+i}}\mid i<\omega\rangle$ is Archimedean, strongly decreasing, and strongly convergent to $0$, as wanted.
\end{proof}

Notice that monoids of infinite, countable degree $\Deg(\SS)=\omega$ can be of any of the three types described in  Lemma~\ref{lem:three_types_of_monoids}. In this case, the only thing that makes the three cases mutually exclusive is the behavior of the operation $+$.

Following the notation from \cite{ReichelTowardsUnifiedMR565844}, we call $0$-continuous the non-pathological monoids, i.e. those that satisfy point~\ref{lem:three_types_of_monoids - Archimedean} or point~\ref{lem:three_types_of_monoids - non-Archimedean} of Lemma~\ref{lem:three_types_of_monoids}.  
The name is based on the fact that, in a $0$-continuous monoid $(\SS,0,+,\leq)$, the operation $+$ is continuous at $0$: whenever two sequences of the same length converge to $0$ in the order topology, their point-wise sum also converges to $0$.

\begin{definition}
    We call \textbf{$0$-continuous} a totally ordered monoid $\SS$ of infinite degree such that $\SS^+$ contains a strongly decreasing coinitial sequence.
\end{definition}

From Lemma~\ref{lem:three_types_of_monoids}, we also get the following interesting corollary. 

\begin{corollary}
    Every $0$-continuous positively totally ordered monoid of uncountable degree is not Archimedean (and not even initially Archimedean).
\end{corollary}

Notice that (densely ordered) groups are always $0$-continuous, and if a monoid contains a dense, $0$-continuous submonoid, then it is $0$-continuous itself.

\section{Infinite operations on monoids}\label{sec:infinite-operatios-on-monoids}

To define $\lambda^+$-measures, we need a structure where sums of $\lambda$-many elements can be computed: ideally, a monoid with a sum of arity at least $\lambda$. 
Infinitary operations and their properties have been widely studied (mostly for the case of groups and fields), even with a specific focus on the purpose of generalizing analysis and descriptive set theory (see, e.g., \cite{wontnerGeneralisationsDescriptiveSet}). As anticipated in the introduction, we work here with a definition of infinitary sum that requires only minimal axioms.

Recall that for every set $A$, a sequence of elements of $A$ is a function from an ordinal $\gamma$ into $A$.
We denote by $\pre{<\On}{A}=\bigcup_{\gamma\in\On}\pre{\gamma}{A}$ the class of sequences of elements of $A$ of any length.
Given a sequence $s=(s_\alpha)_{\alpha<\gamma}\in\pre{<\On}{A}$, a sequence $t$ is a \markdef{subsequence} of $s$ if there is a set $I=\{\alpha_i\mid i<\delta\}\subseteq \gamma$ such that $\alpha_i<\alpha_j$ for all $i<j<\delta$ and $t=\langle s_{\alpha_i}\mid i<\delta\rangle$. 
Given $B\subseteq A$, denote by $s\restriction B$ the subsequence of $s$ obtained by removing all elements outside of $B$ from $s$, i.e. $s\restriction B=(s_{\alpha_i})_{i<\delta}$ where $\{\alpha_i\mid i<\delta\}$ is the increasing enumeration of the ordinals of $\{\alpha<\gamma\mid s_\alpha\in B\}$.
A \textbf{reordering} or \textbf{permutation} of $s$ is a sequence of the form $(s_{\pi(j)})_{j<\delta}$ for some bijection $\pi:\delta\to \gamma$.

We also denote by $\Conc$ the extension of the binary operation of concatenation of sequences  $\conc$ to all sequences of sequences in $\pre{<\On}{({\pre{<\On}{A}})}$.
Formally, we define $\Conc(\emptyset)=\emptyset$. Then, proceeding recursively on the length, given a sequence $\bar{s}=\langle s^\alpha\mid \alpha<\beta\rangle$ of sequences $s^\alpha=(a^\alpha_i)_{i<\leng(s^\alpha)}\in \pre{<\On}{A}$, if $\beta$ is limit, define
\[
\Conc (\bar{s})=\Conc_{\alpha<\beta}{s}^\alpha=\bigcup_{\varepsilon<\beta} \Conc_{\alpha<\varepsilon}{s}^\alpha=\bigcup_{\varepsilon<\beta} \Conc(\bar{s}\restriction \varepsilon),
\]
while if $\beta=\gamma+1$, define 
\[
\Conc (\bar{s})=\Conc_{\alpha<\beta}{s}^\alpha=(\Conc_{\alpha<\gamma}{s}^\alpha)\conc s^{\gamma}=\Conc(\bar{s}\restriction \gamma)\conc s^{\gamma}.
\]

\begin{definition}\label{def:axioms_sum}
   Let $\SS=(S, 0, +, \leq)$ be a positively totally ordered monoid.
   An \markdef{infinitary operation} or \markdef{infinitary sum} is a partial function $\Sum:\pre{<\On}{\SS}\to \SS$ satisfying:
   \begin{enumerate}[label={\upshape (S\arabic*)}]
    \item\label{ax_sum:extend_plus} $\Sum$ is \markdef{compatible with} $+$, i.e., 
    \begin{itemize}
    \item $\Sum(\emptyset)=0$,
    \item $\Sum(s)=s_0 + ... + s_n$ for every non-empty finite $s=(s_i)_{i\leq n}\in \dom(\Sum)$,
    \end{itemize}
    \item\label{ax_sum:continuous}  $\Sum$ is \markdef{continuous}, i.e., for all $s\in \dom(\Sum)$ such that $\leng(s)$ is limit, if $s\restriction \alpha\in \dom(\Sum)$ for cofinally many $\alpha<\leng(s)$, then
    \[
    \Sum(s)=\sup\{\Sum(s\restriction \alpha)\mid \alpha<\leng(s), s\restriction \alpha\in \dom(\Sum)\}.
    \]
    \end{enumerate}
    Furthermore, we say that $\Sum$ is \markdef{natural} if additionally
    \begin{enumerate}[resume, label={\upshape (S\arabic*)}]
    \item\label{ax_sum:dom_downward_closed}  $\dom(\Sum)$ is \markdef{downward closed}, i.e., for all $s\in \dom(\Sum)$ and $\alpha<\leng(s)$ we have
    \[
    s\restriction \alpha\in \dom(s),
    \]
    \item\label{ax_sum:associative} $\Sum$ is \markdef{associative}, i.e., for every $\bar{s}=(s^\alpha)_{\alpha<\gamma}\in\pre{<\On}{\dom(\Sum)}$ such that $\Conc(\bar{s})\in \dom(\Sum)$, we have
       \begin{itemize}
        \item $\langle \Sum(s^\alpha)\mid \alpha<\gamma\rangle\in \dom(\Sum)$, 
           \item $\Sum(\Conc (\bar{s}))=\Sum(\langle \Sum(s^\alpha)\mid \alpha<\gamma\rangle)$,
       \end{itemize}
    \item\label{ax_sum:commutative}  $\Sum$ is \markdef{commutative}, i.e., for every $s\in \dom(\Sum)$ and any reordering $s'$ of $s$ we have 
    \begin{itemize}
        \item $s'\in \dom(\Sum)$,
        \item $\Sum(s)=\Sum(s')$.
    \end{itemize}
\end{enumerate}
\end{definition}

In Definition~\ref{def:axioms_measure}, we defined measures using infinitary sums of arbitrary form. One of the goals of this and the next section is to show that, even when a measure is defined using a non-natural sum, this sum must be natural on the subset of its domain where the measure is defined (Proposition~\ref{prop:sum_is_natural_on_measurable_part}).

Every (positively totally ordered) monoid comes naturally equipped with an infinitary operation.

\begin{definition}\label{def:sum}
    Let $\SS=(S,0,+,\leq)$ be a positively totally ordered monoid. 
    The partial function $\sum:\pre{<\On}{\SS}\to \SS$ is defined by 
    \[
    \sum (s)=\sup\{s_{\alpha_0} + ... + s_{\alpha_k}\mid \alpha_0< ... < \alpha_k<\gamma\}
    \]
    for every ordinal $\gamma$ and every sequence $s=\langle s_\alpha\mid \alpha<\gamma\rangle\in \pre{\gamma}{\SS}$ such that the supremum $\sup\{s_{\alpha_0} + ... + s_{\alpha_k}\mid \alpha_0< ... < \alpha_k<\gamma\}$ exists in $\SS$.
\end{definition}

Recall that a linear order is called \markdef{Dedekind-complete} if every set has a supremum and an infimum. Every linear order $({\SS},\leq)$  has a (unique, up to isomorphism) Dedekind-completion $(\hat{\SS},\preceq)$ in which it is dense (see, e.g., \cite{macneille1937partially,fuchsPartiallyOrderedAlgebraic1963, FornasieroMaminoMR2484481}).

Given $A,B\subseteq \SS$, we write  $\sup(A)\leq \sup(B)$ (even when $\sup(A)$ or $\sup(B)$ is not defined) if for every $c\in \SS$, if $b\leq c$ for every $b\in B$, then $a\leq c$ for every $a\in A$. 
It is easy to check that when both  $\sup(A)$ and $\sup(B)$ are defined, this coincides with the usual $\leq$ relationship of $\SS$.

We also write $\sup(A)=\sup(B)$ if both $\sup(A)\leq \sup(B)$ and $\sup(A)\geq \sup(B)$, and similarly for other expressions like $\sup(A)\leq b$ for $b\in \SS$ (using $b=\sup(\{b\})$), and so on.
With the same meaning, given $s,t\in\pre{<\On}{\SS}$ we write $\sum(s)\leq \sum(t)$, $\sum(s)= \sum(t)$, $\sum(s)\leq b$ for some $b\in \SS$, and so on, even when $\sum(s)$ or $\sum(t)$ is not defined.

\begin{remark}\label{rmk:compare_sup}
This coincides with comparing $\sup$ of sets and $\sum$ of sequences in the Dedekind-completion\footnote{The Dedekind-completion $(\hat{\SS},\leq)$ is a total order, but this is sufficient to define $\sum$ on all sequences of $\SS$. We do not need to define $\sum$ on sequences of $\hat{\SS}$}.
\end{remark}

Recall that a monoid $\SS$ is \markdef{lower semi-continuous} if for every $A,B\subseteq \SS$ such that $a=\sup A$ and $b=\sup B$ exist, we have that $a+b=\sup\{x+y\mid x\in A, y\in B\}$ (\cite[Chapter XI.7]{fuchsPartiallyOrderedAlgebraic1963}).

\begin{proposition}\label{prop:sum_is_natural_infinitary_sum}
For every positively totally ordered monoid $\SS=(S,0,+,\leq)$, the operation $\sum$ is an infinitary sum satisfying additionally:
\begin{enumerate}[label=(\alph*)]
    \item\label{prop:sum_is_natural_infinitary_sum-1} $0$ is the \markdef{neutral element} of $\sum$, i.e., $\sum(s)=\sum(s\restriction (\SS\setminus \{0\}))$ for every $s\in \pre{<\On}{\SS}$.\\
    In particular, $s\in \dom(\sum)$ if and only if $s\restriction (\SS\setminus \{0\})\in \dom(\sum)$.
    \item\label{prop:sum_is_natural_infinitary_sum-2} $\sum$ is \markdef{infinitary-translation invariant}, i.e.
    \[
    \sum(s)\leq \sum(t)
    \]
    for every $s=(s_i)_{i<\gamma}, t=(t_i)_{i<\delta}\in \pre{<\On}{\SS}$ such that $\gamma\leq\delta$ and $s_i\leq t_i$ for every $i<\gamma$.
    \item\label{prop:sum_is_natural_infinitary_sum-3} If the order of $\SS$ is Dedekind-complete, then $\sum$ is \markdef{total}, i.e., it is defined on every sequence of any length.\\
    In particular, in this case $\sum$ satisfies Axiom~\ref{ax_sum:dom_downward_closed}.
    \item\label{prop:sum_is_natural_infinitary_sum-4} If the order of $\SS$ is Dedekind-complete and the operation $+$ of $\SS$ is lower-semicontinuous, then $\sum$ satisfies Axiom~\ref{ax_sum:associative}.
    \item\label{prop:sum_is_natural_infinitary_sum-5} The operation $+$ of $\SS$ is commutative if and only if $\sum$ satisfies Axiom~\ref{ax_sum:commutative}.
\end{enumerate}
\end{proposition}

\begin{proof} 
First, let us prove that $\sum$ is an infinitary operation. It is clear that $\sum$ satisfies Axioms~\ref{ax_sum:extend_plus}.

To see that $\sum$ satisfies Axiom~\ref{ax_sum:continuous}, suppose that $s\in \dom(\sum)$ has limit length, and that there is a sequence $\langle \beta_i\mid i<\gamma\rangle$ cofinal in $\leng(s)$ such that $s\restriction\beta_i\in \dom(\sum)$ for all $i<\gamma$.
For every $\alpha<\gamma$, let $i(\alpha)$ be minimal such that $\alpha<\beta_{i(\alpha)}$.
Then, for each $\alpha_0<...<\alpha_n<\gamma$ we have  $s_{\alpha_0} + ... + s_{\alpha_n}\leq \sum(s\restriction \beta_{i(\alpha_n)})\leq \sum(s)$, by definition of $\sum$.
This shows that $\sup\{\sum(s\restriction \beta_{i})\mid i<\gamma\}$ exists and
\[
\sum(s)=\sup\{s_{\alpha_0} + ... + s_{\alpha_n}\mid n<\omega\}\leq \sup\{\sum(s\restriction \beta_{i})\mid i<\gamma\} \leq \sum(s),
\]
and Axiom~\ref{ax_sum:continuous} is satisfied, as wanted.

That $\sum$ is infinitary-translation invariant and that $0$ is the neutral element follow from the definition. Thus statements~\ref{prop:sum_is_natural_infinitary_sum-1} and  \ref{prop:sum_is_natural_infinitary_sum-2} hold. 
It is also immediate that if the order of $\SS$ is Dedekind-complete, then $\sum$ is total, and that the operation $+$ of $\SS$ is commutative iff $\sum$ satisfies Axiom~\ref{ax_sum:commutative}.
Thus statements~\ref{prop:sum_is_natural_infinitary_sum-3} and  \ref{prop:sum_is_natural_infinitary_sum-5} hold.

Finally, to prove statement~\ref{prop:sum_is_natural_infinitary_sum-4} (and that $\sum$ satisfies Axiom~\ref{ax_sum:associative}), assume $+$ is lower semi-continuous and $\leq$ is Dedekind-complete. Then, by previous point $\sum$ is total.
Consider first $s,t\in \pre{<\On}{\SS}$. Then, 
\[
\begin{gathered}
\sum s + \sum t =\\
=\sup\{s_{i_0} + ... + s_{i_k}\mid i_0{<} ... {<} i_k{<}\leng(s)\} +  \sup\{t_{j_0} + ... + t_{j_k}\mid j_0{<} ... {<} j_k{<}\leng(t)\}=\\
 =\sup\{s_{i_0} + ... + s_{i_k} + t_{j_0} + ... + t_{j_k}\mid i_0< ... < i_k<\leng(s), j_0< ... < j_k<\leng(t)\} =\\ 
=\sum (s\conc t)
\end{gathered}
\]
by definition of $\sum$ and of lower semi-continuity. By induction, it is easy to see that $\sum (s^0\conc ... \conc s^k) = \sum s^0 + ... + \sum s^k$ for all finite $\bar{s}=(s^i)_{i\leq k}\in\pre{<\omega}{\dom(\Sum)}$.

Now fix $\beta\geq \omega$.
Let $\bar{s}=\langle s^{\alpha}:\alpha<\beta\rangle$ be such that $s^\alpha\in\pre{<\On}{\SS}$ for every $\alpha<\beta$. 
Notice that 
\[
\sum(\Conc(\bar{s})) \leq \sum(\langle\sum (s^{\alpha}):\alpha<\beta\rangle),
\]
by definition of $\sum$. 
On the other hand, we have 
\[
\sum(\langle\sum (s^{\alpha}):\alpha<\beta\rangle)=\sup\{\sum (s^{\alpha_0}) + ... + \sum (s^{\alpha_n}):\alpha_0<...<\alpha_n<\beta\},
\]
and by previous argument for finite sequences, we get 
\[
\sum (s^{\alpha_0}) + ... + \sum (s^{\alpha_n})=\sum (s^{\alpha_0}\conc  ... \conc s^{\alpha_n})\leq \sum(\Conc (\bar{s})).
\]
All together, these show that $\sum(\Conc(\bar{s})) = \sum(\langle\sum (s^{\alpha}):\alpha<\beta\rangle)$, as wanted.
\end{proof}  

\begin{corollary}
    For every Dedekind-complete, lower semi-continuous, commutative positively totally ordered monoid $\SS=(S,0,+,\leq)$, the operation $\sum$ is a natural infinitary sum.
\end{corollary}

In Proposition~\ref{prop:completion_is_infinitary_ptom}, we show that there are Dedekind-complete, lower semi-continuous, commutative positively totally ordered monoids of any desired degree. In particular, this shows that there are positively totally ordered monoids of any degree admitting a natural infinitary sum.

The following proposition shows that all natural infinitary sums come from restrictions of $\sum$. 

\begin{proposition}\label{prop:sum_is_the_only_natural_infinitary_sum}
   Let $\SS=(S, 0, +, \leq)$ be a positively totally ordered monoid with a natural infinitary sum $\Sum:\pre{<\On}{\SS}\to \SS$. 
   Then, we have that $\dom(\Sum)\subseteq \dom(\sum)$ and $\Sum(s)=\sum(s)$ for every $s\in \dom(\Sum)$.
\end{proposition}

\begin{proof}
Assume $\Sum:\pre{<\On}{\SS}\to \SS$ is a natural infinitary sum.
Fix $s=(s_{\alpha})_{\alpha<\gamma}\in \dom(\Sum)$.
By induction on $\gamma$, we now show that $s\in \dom(\sum)$ and 
    \[
    \Sum(s)= \sup\{s_{\alpha_0} + ... + s_{\alpha_k}\mid \alpha_0< ... < \alpha_k<\gamma\}=\sum(s),
    \]
    as wanted.

    It is clear that this is true if $\gamma$ is finite, by Axiom~\ref{ax_sum:extend_plus}.
So suppose $\gamma\geq \omega$.

Let $\alpha_0< ... < \alpha_k<\gamma$, and let $t=(s_{\alpha_i})_{i<k}$. By Axiom~\ref{ax_sum:commutative}, there is a reordering $s'\in \dom(\Sum)$ of $s$ of limit length such that $t=s'\restriction k$. This shows that $t\in \dom(\Sum)$, by Axiom~\ref{ax_sum:dom_downward_closed}, and consequently that $+$ and $\sum$ are commutative when restricted to $\{s_\alpha\mid \alpha<\leng(s)\}$, since $\Sum$ satisfies Axiom~\ref{ax_sum:commutative}.

    Now if $\gamma=\beta+1$ is infinite and the induction hypothesis holds for $1+\beta=\beta$, then 
    \[
    \Sum(s)=\Sum(s_\beta\conc (s\restriction \beta))=\sum (s_\beta\conc (s\restriction \beta))=\sum(s),
    \]
    by Axiom~\ref{ax_sum:commutative}, induction hypothesis, and commutativity of $\sum$ on $\{s_\alpha\mid \alpha<\leng(s)\}$.
    
    Suppose instead that $\gamma$ is limit and the statement holds for all $\beta<\gamma$.
    By Axioms~\ref{ax_sum:continuous} and~\ref{ax_sum:dom_downward_closed}, we have $s\restriction\alpha\in \dom(\Sum)$ for every $\alpha<\gamma$ and
    \[
\Sum(s)=\sup\{\Sum(s\restriction\alpha)\mid \alpha<\gamma\}.
    \]
    By induction hypothesis, we get $s\restriction\alpha\in \dom(\sum)$ for every $\alpha<\gamma$ and 
    \[
\Sum(s)=\sup\{\sum(s\restriction\alpha)\mid \alpha<\gamma\}.
\]
    This shows that
    \[
    \sup\{\sum(s\restriction\alpha)\mid \alpha<\gamma\}=\sup\{s_{\alpha_0} + ... + s_{\alpha_k}\mid \alpha_0< ... < \alpha_k<\gamma\}
    \]
    is defined, and thus $s\in \dom(\sum)$ and $\Sum(s)=\sup\{\sum(s\restriction\alpha)\mid \alpha<\gamma\}=\sum(s)$ as wanted.
\end{proof}

In particular, by Proposition~\ref{prop:sum_is_the_only_natural_infinitary_sum} we get the following.

\begin{remark}
Every natural infinitary sum satisfies the same properties of $\sum$ (e.g., those of Proposition~\ref{prop:sum_is_natural_infinitary_sum}), restricted to its domain.
\end{remark}

\subsection{Examples of monoids}

In the following, we provide different examples of \PTOMs. The first example is the well-known monoid $[0,\infty]$ used in classical measure theory.

\begin{example}\label{ex:R_infty}
    We denote by $\RR_\infty=([0,\infty], 0, +, \leq)$ the totally ordered monoid where $[0,\infty]\setminus \{\infty\}=[0,\infty)$ is the positive cone of the real numbers, with usual order and operation, and $\infty$ is the maximum of the order and an absorbing element for the operation (i.e. $a+\infty= \infty + a=\infty$ for every $a\in [0,\infty]$).

    Then, $\RR_\infty$ is an initially Archimedean, Dedekind-complete, lower semi-continuous, $0$-continuous, commutative, positively totally ordered monoid, thus $\sum$ is a total natural infinitary sum on it, by Proposition~\ref{prop:sum_is_natural_infinitary_sum}.

    The operation $\sum$ can be explicitly defined for any sequence $r=\langle r_\beta\mid \beta<\gamma\rangle$ of elements of $[0,\infty]$ in the following way: 
    \begin{itemize}
    \item $\sum r=0$ if $r$ is empty or all elements of $r$ are zero,
    \item $\sum r=r_{n_1} + ... + r_{n_k}$ if $r$ does not contain $\infty$ and has exactly $k$-many elements $\{r_{n_1}, ..., r_{n_k}\}$ different from $0$,
    \item $\sum r=\sum_{i=0}^{\infty} r_{n_i}$ if $r$ does not contain $\infty$ and it has countably many elements that are different from $0$ and $\{r_{n_i}\mid i<\omega\}$ is an enumeration of them (by the absolute convergence theorem, the sum $\sum_{i=0}^{\infty} r_{n_i}$ is independent of the chosen enumeration), 
    \item $\sum r=\infty$ otherwise.
    \end{itemize}
\end{example}

Another class of examples is given by Dedekind-complete total orders.

\begin{example}\label{ex:complete_linear_orders_are_infinitary_monoids}
    Let $(S,\leq)$ be a Dedekind-complete linear order with a minimum $0$.
    Then, $\SS=(S, 0, \max,\leq)$ is a Dedekind-complete, lower semi-continuous, commutative, positively totally ordered monoid, and the infinitary operation $\sum=\sup$ is a total natural infinitary sum.
    If $\Deg(\SS)$ is infinite, then $\SS$ is also $0$-continuous.
\end{example}

Notice that there are Dedekind-complete linear orders of any coinitiality (e.g., $\kappa+1$ with reverse order is a Dedekind-complete linear order of coinitiality $\kappa$ for any cardinal $\kappa$). This shows that for any $\kappa$, there exists a Dedekind-complete, positively totally ordered commutative monoid $\SS$ of degree $\Deg(\SS)=\kappa$ with a total natural infinitary sum.

Other, more interesting examples are given by the completion of groups.

\begin{proposition}\label{prop:completion_is_infinitary_ptom}
    Let $\GG=(G,0,+,\leq)$ be a commutative densely totally ordered group. Then, the operation $+$ and the order $\leq$ can be extended to the Dedekind-completion $S_\GG$ of $\GG$ so that $\SS_\GG=(S_\GG,0,+,\leq)$ is a Dedekind-complete, lower semi-continuous, totally ordered commutative monoid having $\GG$ as dense subgroup.

    In particular, $\{0\}\cup \SS_\GG^+$ is a Dedekind-complete, lower semi-continuous, $0$-con\-tinu\-ous, positively totally ordered commutative monoid such that $\Deg(\SS_\GG)=\Deg(\GG)$.
\end{proposition}

\begin{proof}
    Notice that every totally ordered group is a topological group (see, e.g., Theorem 10 and the following paragraph of \cite[Chapter II.8]{fuchsPartiallyOrderedAlgebraic1963}). In particular, this shows that the operation $+$ is lower semi-continuous on $\GG$.
    By \cite{FornasieroMaminoMR2484481}, we can extend the operation $+$ and the order $\leq$ to $\SS_\GG$ so that it is a Dedekind-complete, commutative totally ordered monoid having $\GG$ as dense subgroup, and the operation is lower semi-continuous by definition. 
\end{proof}

In general the operation $\sum$ need not be natural. The following example shows that $\sum$ does not necessarily satisfy Axiom~\ref{ax_sum:dom_downward_closed} nor Axiom~\ref{ax_sum:associative}, even among lower semi-continuous, commutative, positively totally ordered monoids.

\begin{example}
    Let $\QQ_\infty=(\QQ\cap [0,\infty))\cup\{\infty\}\subset \RR_\infty$. Then, $\QQ_\infty$ is a lower semi-continuous, commutative, positively totally ordered submonoid of $\RR_\infty$.
    Let $(s_n)_{n<\omega}$ be a sequence of rational numbers summing up to $\pi$ (or your favorite irrational number). Let also $t_{\gamma+n}=s_n$ for every limit ordinal $\gamma$ and natural number $n<\omega$.
    Let $t=(t_\alpha)_{\alpha<\omega^2}$.
    Then, we have $\sum(t)=\infty$, however $t\restriction \alpha\notin \dom(\sum)$ for every $\alpha<\leng(t)$.
\end{example}

The following examples show that $\sum$ does not necessarily satisfy Axiom~\ref{ax_sum:associative}, even among Dedekind-complete, commutative, positively totally ordered monoids.

\begin{example}
    Let $\SS=([0,\infty)\cup \{\infty_0, \infty_1\}, 0, +, \leq)$ where $[0,\infty)$ is the positive cone of the real numbers and $+$ and $\leq$ are its usual operation and order. 
    Define also $a< \infty_0<\infty_1$ for all $a\in [0,\infty)$, and $a+\infty_i=\infty_i+a=\infty_i$ for any $a\in [0,\infty)$, and $\infty_i+\infty_j= \infty_1$ for every $i,j\in\{0,1\}$. Then, $\SS$ is a Dedekind-complete, commutative, positively totally ordered monoid, yet $\sum$ is not associative, as 
    \[
    \sum(1^{(\omega)}\conc 1^{(\omega)})=\sum(1^{(\omega+\omega)})=\infty_0\neq \infty_1=\infty_0+\infty_0= \sum(1^{(\omega)})+\sum(1^{(\omega)}).
    \]
\end{example}

\begin{example}
    Let $(\MM, 0, + , \leq)$ be a commutative totally ordered monoid without maximum. 
    Let $\SS=\{0\}\times \MM \cup \{1\}\times \MM \sqcup \{\infty\}$, where the order is given by $(i,a)\leq(i,b)$ for any $i\in\{0,1\}$ and $a,b\in \MM$ satisfying $a\leq b$, and $(0,a)\leq (1,b)\leq \infty$ for any $a,b\in \MM$.
    Let $\ast$ be defined by $(i,a)\ast (j,b)=(i+j,a+b)$ if $i+j\leq 1$, $(1,a)\ast (1,b)=\infty$, and $(i,a)\ast \infty = \infty \ast (i,a)=\infty$.
    Then, $(\SS,(0,0),\ast,\leq)$ is a commutative, positively totally ordered monoid, and it is Dedekind-complete if $\MM $ plus a maximum is. Yet, $\sum$ is not associative, as if $\langle a_i\mid i<\gamma\rangle$ is a cofinal sequence in $\MM$, then 
    \[
    \sum(\langle (0,a_i)\mid i<\gamma\rangle\conc \langle (0,a_i)\mid i<\gamma\rangle)=(1,0),
    \]
    \[
    \sum(\langle (0,a_i)\mid i<\gamma\rangle) +\sum(\langle (0,a_i)\mid i<\gamma\rangle)=(1,0)+(1,0)=\infty.
    \]
\end{example}

\section{Measures}\label{sec:measures}

For the classical theory of measures, we refer the reader to, e.g., \cite{KechrisMR1321597, RudinRealAnalysisMR924157}.
Recall the definition of $\lambda^+$-measure, Definition~\ref{def:axioms_measure}. 
Notice that the notion of $\omega^+$-measure properly extends the notion of classical measure by allowing one to use monoids other than $\RR_\infty$ and measurable structures that are not $\sigma$-algebras.

Recall also that, adopting the notation of \cite[Exercise 17.4]{KechrisMR1321597},  given a weak $\lambda^+$-measure space $(X,\MMM, \mu)$, we say that $\mu$ is continuous if all points are measurable of measure $\mu(x)=0$.

Even though the definition of a weak $\lambda^+$-measure can use any kind of infinitary sum $\Sum$, the next proposition shows that this $\Sum$ must be natural at least on the range of $\mu$. 
Since the values of $\Sum$ outside the range of $\mu$ do not matter in this context, Proposition~\ref{prop:sum_is_the_only_natural_infinitary_sum} then implies that $\sum$ is essentially the only infinitary sum that can be used to define measures.

\begin{proposition}\label{prop:sum_is_natural_on_measurable_part}
    Let $\SS$ be a positively totally ordered monoid with an infinitary operation $\Sum$, and let $(X,\MMM, \mu)$ be a  weak $\lambda^+$-measure space such that $\mu$ takes values in $(\SS,\Sum)$. Let 
    \[
    \D=\{\langle \mu(A_i)\mid i<\gamma\rangle\mid (A_i)_{i<\gamma}\text{ is a family of disjoint sets of }\MMM \text{ of length } \gamma< \lambda^+\}.
    \]
    Then, we have that $\Sum\restriction \D$ is natural, and thus
    \[
    \mu(\bigcup_{i\in\gamma}A_i)=\sum_{i\in\gamma}\mu(A_i)
    \]
    for every family $(A_i)_{i<\gamma}$ of disjoint sets of $\MMM$ of length $\gamma< \lambda^+$.
\end{proposition}

\begin{proof}
We want to show that $\Sum \restriction \D$ is natural, and the rest will follow from Proposition~\ref{prop:sum_is_the_only_natural_infinitary_sum}.

First, it is clear that $\Sum \restriction \D$ satisfies Axioms~\ref{ax_sum:extend_plus} and~\ref{ax_sum:continuous} since $\Sum$ does, and it satisfies Axiom~\ref{ax_sum:dom_downward_closed} by definition of $\D$ and by the first half of Axiom~\ref{ax_measure:additive}.
It is also immediate to check that it satisfies Axiom~\ref{ax_sum:commutative}, by Axiom~\ref{ax_measure:additive} and since $\bigcup$ is commutative.
    
    Finally, for Axiom~\ref{ax_sum:associative} notice that if $\bar{s}=(s^\alpha)_{\alpha<\gamma}\in \pre{<\On}{\D}$ and $\Conc(\bar{s})\in \D$, and $\{A^\alpha_i\mid \alpha<\gamma, i<\leng(s^\alpha)\}$ are disjoint sets in $\MMM$ witnessing that $\Conc(\bar{s})\in \D$, then $\langle \bigcup_{i<\leng(s^\alpha)} A^\alpha_i\mid \alpha<\gamma\rangle$ is still a sequence of disjoint sets in $\MMM$ of length $\gamma<\lambda^+$. Thus, 
    \[
    \langle \Sum(s^\alpha)\mid \alpha<\gamma\rangle=\langle \mu(\bigcup_{i<\leng(s^\alpha)} A^\alpha_i)\mid \alpha<\gamma\rangle\in \D
    \]
    and 
\begin{align*}
\Sum(\Conc(\bar{s}))&=\mu(\bigcup_{\alpha<\gamma}\bigcup_{i<\leng(s^\alpha)}A^\alpha_i)=\\
&=\Sum(\langle \mu(\bigcup_{i<\leng(s^\alpha)} A^\alpha_i)\mid \alpha<\gamma\rangle)=\\
&=\Sum(\langle \Sum(s^\alpha)\mid \alpha<\gamma\rangle)
 \end{align*}
    as wanted.
\end{proof}

Thus, without loss of generality, 
\begin{center}
\textit{
 for the rest of the paper we fix a positively totally ordered monoid $\SS=(S, 0, +, \leq)$ and work with the infinitary sum $\Sum=\sum$ from Definition~\ref{def:sum},} 
\end{center}
unless specified otherwise.
Furthermore, by restricting the domain of $\sum$ to $\D$, in practice we can always assume that $\sum$ is natural (on the relevant part of its domain).

\subsection{Properties of measures} 
Every classical measure is also countable subadditive, i.e., $\mu(\bigcup \V) \leq \sum_{V\in \V} \mu(V)$ for every countable family of measurable sets $\V\subseteq \MMM$.
We can define an analogue property in this context.

\begin{definition}
    A weak $\lambda^+$-measure space $(X,\MMM, \mu)$ is \markdef{$\lambda^+$-subadditive} if for every family $\V\subseteq \MMM$ of size $|\V|\leq \lambda$, we have $\mu(\bigcup \V) \leq \sum_{V\in \V} \mu(V)$.
\end{definition}

In general, from the axioms of Definition~\ref{def:axioms_measure} alone we cannot conclude that every weak $\lambda^+$-measure space $(X,\MMM, \mu)$ is {$\lambda^+$-subadditive}.

\begin{example}
    Let $\tau$ be the usual euclidean topology on $\RR$. Let $\QQ+\pi=\{q+\pi\mid q\in \QQ\}$, and let
    \[
    \MMM=\{O\cup D\mid O\in \tau, D\in \{ \emptyset, \RR\setminus\QQ, \RR\setminus(\QQ+\pi)\}\}.
    \]
    Then, it is immediate to check that $\MMM$ is a weak $\omega^+$-measurable structure on $\RR$.
    Let $\mu_L:\MMM_L\to \RR_\infty$ be the usual Lebesgue measure on $\RR$.
    For every $M\in \MMM$, define 
    $\mu(M)=\mu_L(\int{M})$.
    
    It is easy to check that $\mu$ satisfies Axiom~\ref{ax_measure:emptyset}-\ref{ax_measure:non-trivial}, since $\mu_L$ does. Also, for every $A,B\in \MMM$ such that $A\subseteq B$, we have $\int{A}\subseteq \int{B}$ and $\mu(A)=\mu_L(\int{A})\leq \mu_L(\int{B})=\mu(B)$, thus Axiom~\ref{ax_measure:decreasing} holds as well.
    Since no point is measurable in $\MMM$, then $\mu$ satisfies Axiom~\ref{ax_measure:point-regular}.
    Finally, if $\A$ is a family of disjoint measurable sets of size $\geq 2$, then $\A\subseteq\tau$, since $\RR\setminus\QQ$ and $\RR\setminus(\QQ+\pi)$ are dense in $\RR$ and have non-empty intersection. Thus $\mu$ satisfies Axiom~\ref{ax_measure:additive}.
    However, 
    \[
    \mu((\RR\setminus\QQ)\cup (\RR\setminus(\QQ+\pi)))=\mu(\RR)=\infty\nleq 0=\mu(\RR\setminus\QQ)+ \mu(\RR\setminus(\QQ+\pi)),
    \]
    thus $\mu$ is not even finitely subadditive.
\end{example}

However, in every weak $\lambda^+$-measure space $(X,\MMM,\mu)$ where $\MMM$ is sufficiently nice, the $\lambda^+$-measure $\mu$ is indeed subadditive.

\begin{definition}
    Given a weak $\lambda^+$-measure space $(X,\MMM, \mu)$, we say that a family $\V\subseteq \MMM$ is \markdef{$\lambda^+$-partitionable} if there exists a partition $\P\subseteq \MMM$ of $\bigcup \V$ of size less than $\lambda^+$ refining $\V$.
    
    We say that $(X,\MMM, \mu)$ is \markdef{$\lambda^+$-partitionable} if every family $\V\subseteq \MMM$ of size less than $\lambda^+$ is $\lambda^+$-partitionable.
\end{definition}

Every weak $\lambda^+$-measure space $(X,\MMM, \mu)$ where $\MMM$ is closed under complements (and thus, a $\lambda^+$-algebra) is also $\lambda^+$-partitionable. However, these are not the only cases of $\lambda^+$-partitionable, weak $\lambda^+$-measure spaces (see Proposition~\ref{prop:minimal_measure_spaces_are_partitionable}).

In these spaces, subadditivity holds in the expected way.

\begin{proposition}\label{prop:measure_of_union_leq_sum_of_measures}
Let $(X,\MMM,\mu)$ be a weak $\lambda^+$-measure space. Then, for every $\lambda^+$-partitionable family $\V\subseteq \MMM$ we have $\mu(\bigcup \V) \leq \sum_{V\in \V} \mu(V)$.
\end{proposition}

Recall that we follow the convention introduced in Remark~\ref{rmk:compare_sup} and preceding paragraph: the fact that we write $\mu(\bigcup \V) \leq \sum_{V\in \V} \mu(V)$ does not imply that $\sum_{V\in \V} \mu(V)$ exists in $\SS$.

\begin{proof}
    Let $\P\subseteq \MMM$ be a family of disjoint sets of size $<\lambda^+$ refining $\V$ such that $\bigcup \P=\bigcup \V$. 
    Since $\P$ refines $\V$, for every $P\in\P$ there is $V\in \V$ such that $P\subseteq V$. Using the axiom of choice, we can find a partition $\{\P_V\mid V\in\V\}$ of $\P$ such that $\P=\bigcup_{V\in \V} \P_V$ and $P\subseteq V$ for every $P\in \P_V$. 
    Thus, $(\bigcup\P_V)_{V\in \V}$ is still a partition of $\bigcup \V$ refining $\V$ with $\leq \lambda$ elements different from $\emptyset$, and each $\bigcup\P_V$ is measurable since $\MMM$ is closed under unions of size $\leq \lambda$. 
    By Axiom~\ref{ax_measure:decreasing}, we get $\mu(\bigcup \P_V)\leq \mu(V)$ for every $V\in \V$, and so
    \[
    \mu(\bigcup \V)=\mu(\bigcup_{\substack{V\in \V, \\ \P_V\neq \emptyset}}\bigcup \P_V)=\sum_{\substack{ V\in \V, \\ \P_V\neq \emptyset}} \mu(\bigcup \P_V)=\sum_{V\in \V} \mu(\bigcup \P_V)\leq \sum_{V\in \V} \mu(V)
    \]
    by Axiom~\ref{ax_measure:additive} and by Proposition~\ref{prop:sum_is_natural_infinitary_sum} (using that $\sum$ is infinitary-translation invariant and $0=\mu(\emptyset)$ is the neutral element of $\sum$).
\end{proof}

\begin{corollary}\label{cor:partitionable_implies_subadditive}
Let $(X,\MMM,\mu)$ be a $\lambda^+$-partitionable, weak $\lambda^+$-measure space. Then, $\mu$ is $\lambda^+$-subadditive.
\end{corollary}

Notice that by Fact~\ref{fct:strong_zero_dimensionality} and by our assumption on the weight of $X$, every family of open sets is partitionable.

\begin{remark}\label{rmk:families_of_open_sets_are_partitionable}
Let $(X,\MMM,\mu)$ be a weak $\lambda^+$-measure space. Then every family of open sets is $\lambda^+$-partitionable.
\end{remark}

Therefore, by this and Proposition~\ref{prop:measure_of_union_leq_sum_of_measures}, we obtain the following.

\begin{corollary}\label{cor:sets_of_measure_zero_are_lambda_ideal}
Let $(X,\MMM,\mu)$ be a weak $\lambda^+$-measure space. Then, the family of measure zero open sets $\N=\{V\in \tau\mid \mu(V)=0\}$ is closed under unions of any size.
\end{corollary}

Another important property of classical measure spaces is that adding a set of measure zero should not increase the measure of a set.
Here, we get something similar under the additional assumption that the weak $\lambda^+$-measure space is $\lambda^+$-subadditive.

\begin{definition}
    Let $(X,\MMM, \mu)$ be a weak $\lambda^+$-measure space. A subset $N\subseteq X$ is said to be \markdef{essentially null} if for every $M, M'\in \MMM$, we have that 
    \[
    M\setminus N=M'\setminus N \qquad \text{ implies } \qquad \mu(M)=\mu(M').
    \]
\end{definition}
 
\begin{remark}\label{rmk:essentially_null_closed_under_subsets}
    Let $(X,\MMM, \mu)$ be a weak $\lambda^+$-measure space. If $N\subseteq X$ is {essentially null}  and $N'\subseteq N$, then $N'$ is {essentially null} as well, i.e., essentially null sets are closed under subsets. 
    
\end{remark}

\begin{lemma}\label{lem:null_are_essentially_null_if_subadditive}
    Let $(X,\MMM, \mu)$ be a $\lambda^+$-subadditive, weak $\lambda^+$-measure space.
    Then, every set of measure zero is also essentially null.
\end{lemma}

\begin{proof}
    Let $M, M', N\in \MMM$ be such that $\mu(N)=0$ and $M\setminus N=M'\setminus N$. 
    Then $M\subseteq M'\cup N$ and $M'\subseteq M\cup N$. 
    Thus, if $\mu$ is {$\lambda^+$-subadditive}, we get
    \[
    \mu(M)\leq\mu(M'\cup N)\leq\mu(M')+\mu(N)=\mu(M')+0=\mu(M')
    \]
    and viceversa, therefore $\mu(M)=\mu(M')$, as wanted.
\end{proof}

While these properties seem important enough that we might want to assume them, in practice we won't need to: they are automatically satisfied in a substructure of any weak $\lambda^+$-measure space (see Proposition~\ref{prop:minimal_measure_spaces_are_partitionable}).

\subsection{Minimal \texorpdfstring{$\lambda^+$}{lambda+}-measurable spaces}

In order to prove our main impossibility theorem, we introduce a simple class of weakly $\lambda^+$-measurable spaces, to which we can always restrict when needed.

\begin{definition}
   A \markdef{minimal $\lambda^+$-measurable space} is a weakly $\lambda^+$-measurable space $(X,\MMM)$ such that for every $M\in \MMM$ there is a partition $M=D\cup O$ satisfying that $|D|\leq \lambda$, $x\in \MMM$ for every $x\in D$, and $O\in \tau$ is open.

   A \markdef{minimal $\lambda^+$-measure space} is a weak $\lambda^+$-measure space $(X,\MMM, \mu)$ where $(X,\MMM)$ is a {minimal $\lambda^+$-measurable space}.
\end{definition}

The term \textit{minimal} comes from the fact that every weak $\lambda^+$-measure space contains a unique minimal $\lambda^+$-measure space that measures exactly the same points. This can be regarded as the \textit{core}, or \emph{kernel}, of the space.

Notice that given a weakly $\lambda^+$-measurable space $(X,\MMM)$, if $D$ and $O$ are subsets of $X$ satisfying that $|D|\leq \lambda$, $x\in \MMM$ for every $x\in D$, and $O\in \tau$, then we get $D\in \MMM$ and $O\in \MMM$, by definition of weakly $\lambda^+$-measurability, and therefore $D \cup O\in \MMM$.
Thus, in a minimal $\lambda^+$-measurable space $(X,\MMM)$
we have $M\in \MMM$ if and only if there is a partition $M=D\cup O$ satisfying that $|D|\leq \lambda$, $x\in \MMM$ for every $x\in D$, and $O\in \tau$ is open.

Given $A\subseteq X$, let $\MMM^{\lambda}_p(A)$ be the minimal family of subsets of $X$ containing all open sets and all points of $A$, and closed under unions of size at most $\lambda$.
It is easy to check that $(X,\MMM^\lambda_p(A))$ is a minimal $\lambda^+$-measurable space, and that if $(X,\MMM)$ is any other weakly $\lambda^+$-measurable space satisfying $A\subseteq\{x\in X\mid x\in \MMM\}$, then necessarily $\MMM^\lambda_p(A)\subseteq \MMM$.
Therefore, $\MMM^\lambda_p(A)$ can be equivalently described as
\[
\MMM^\lambda_p(A):=\bigcap\{\MMM:(X,\MMM)\text{ is a weakly }\lambda^{+}\text{-measurable space}\,\land\,  x\in \MMM \text{ for all }x\in A\}.
\]
In particular, every minimal $\lambda^+$-measurable space $(X,\MMM)$ is of the form $\MMM=\MMM^\lambda_p(A)$ for $A=\{x\in X\mid x\in \MMM\}$, and every weakly $\lambda^+$-measurable space $(X,\MMM)$ contains a unique minimal $\lambda^+$-measurable subspace $(X,\MMM^\lambda_p(A))$ measuring the same points $A=\{x\in X\mid x\in \MMM\}$ of $\MMM$.

Furthermore, if $(X,\MMM,\mu)$ is a weak $\lambda^+$-measure space and $\MMM'\subseteq \MMM$ is such that $(X,\MMM')$ is weakly $\lambda^+$-measurable, then $(X,\MMM',\mu)$ is also a weak $\lambda^+$-measure space.
This observation allows us to introduce the following definition.

\begin{definition}
    The \markdef{kernel} of a weak $\lambda^+$-measure space $(X,\MMM,\mu)$ is the minimal $\lambda^+$-measure space $\ker(X,\MMM,\mu)=(X,\ker(\MMM),\mu)$, where 
    \[
    \ker(\MMM)=\MMM^\lambda_p(\{x\in X\mid x\in \MMM\})\subseteq \MMM
    \]
    is the unique minimal $\lambda^+$-measurable structure measuring the same points of $\MMM$.
\end{definition}

As a consequence, without loss of generality, we can restrict our study to minimal $\lambda^+$-measurable spaces: if we show that finding such a minimal structure is impossible, then we get automatically that finding anything else is impossible.

Minimal $\lambda^+$-measurable spaces have a nicer behavior than most weakly $\lambda^+$-mea\-sur\-able spaces, which makes them ideal candidates to work with.

\begin{proposition}\label{prop:minimal_measurable_spaces_closed_under_intersections}
    For every minimal $\lambda^+$-measurable space $(X,\MMM)$ we have that $\MMM$ is closed under intersections of size less than $\kappa$.
\end{proposition}

\begin{proof}
Recall that $X\subseteq \pre{\kappa}{\lambda}$. Let $(X,\MMM)$ be a minimal $\lambda^+$-measurable space, and let $(M_i)_{i<\gamma}$ be a sequence of elements of $\MMM$ of length $\gamma<\kappa$. For every $i<\gamma$, let  $D_i,O_i\in \MMM$ be such that $M_i=D_i\cup O_i$, $|D_i|\leq \lambda$, $x\in \MMM$ for every $x\in D_i$, and $O_i\in \tau$. 
Let $M=\bigcap_{i<\gamma} M_i$. Then, if $O=\bigcap_{i<\gamma} O_i$, we have that $O$ is open, by $\kappa$-additivity of $\pre{\kappa}{\lambda}$. 
Also, 
\[
D=M\setminus O\subseteq \bigcup_{i<\gamma} D_i
\]
has size $\leq \lambda$, since $\gamma< \kappa\leq \lambda$, and $x\in \MMM$ for every $x\in D$.
Thus $M=D\cup O\in \MMM$ as well, as wanted.
\end{proof}

In minimal $\lambda^+$-measure spaces, all families of measurable sets are essentially made of open sets, modulo a set of small size. By Fact~\ref{fct:strong_zero_dimensionality}, we get then that they are $\lambda^+$-partitionable.

\begin{proposition}\label{prop:minimal_measure_spaces_are_partitionable}
    Every minimal $\lambda^+$-measure space on $X$ is $\lambda^+$-partitionable.
\end{proposition}

\begin{proof}
Let $(X,\MMM, \mu)$ be a minimal $\lambda^+$-measure space. 
Consider $\V\subseteq \MMM$ of size $\leq \lambda$.
Then, for each $M\in \V$ we can find two disjoint sets $O(M)\in \tau$ and $D(M)\in \MMM$ such that $M=O(M)\cup D(M)$, $|D(M)|\leq \lambda$, and $x\in \MMM$ for every $x\in D(M)$.

By Fact~\ref{fct:strong_zero_dimensionality}, we can find a clopen partition $\P_1$ of $O=\bigcup_{M\in \V} O(M)$ refining $\{O(M)\mid M\in\V\}$ (and thus $\V$).
Notice that $\P_1\subseteq \MMM$, since $\MMM$ contains all open subsets of $X$.
Also, $|\P_1|\leq \lambda$, since we assumed that $X$ has weight $\leq \lambda$.

Let also $D=\bigcup \V\setminus O$, and $D'=\bigcup_{M\in \V} D(M)$. Notice that $D\subseteq D'$: this implies both that $|D|\leq\lambda$, and that $x\in \MMM$ for every $x\in D$. Therefore, $\P_2=\{\{x\}\mid x\in D\}$ is a partition of $D$ refining $\V$.

Thus, $\P_1\cup \P_2$ is a partition of $\bigcup \V$ refining $\V$, as wanted.
\end{proof}

By Corollary~\ref{cor:partitionable_implies_subadditive}, we get the following.

\begin{corollary}\label{cor:minimal_measure_spaces_are_subadditive}
   Every minimal $\lambda^+$-measure space on $X$ is $\lambda^+$-subadditive.
\end{corollary}

\subsection{Extensions and restrictions of measures}

Given a weak $\lambda^+$-measure space $(Y,\MMM, \mu)$, we first define a way to induce a structure on a superspace $X\supseteq Y$.

\begin{definition}\label{def:induced_measure_on_superspace}
    Given two spaces $Y\subseteq X$ and a weak $\lambda^+$-measure space $(Y,\MMM, \mu)$, define the set
    \[
    \MMM\uparrow X=\{A\subseteq X\mid A\cap Y\in \MMM\},
    \] 
    and let
    \[
    \mu\uparrow X:\MMM\uparrow X \to \SS
    \]
    be the function given by 
    \[
    (\mu\uparrow X)(B)=\mu(B\cap Y) \text{ for all } B\in \MMM\uparrow X.
    \]
\end{definition}

Conversely, given a weak $\lambda^+$-measure space $(X,\MMM, \mu)$, we define the a way to induce a structure on a subspace $Y\subseteq X$. 

\begin{definition}
    Given two spaces $Y\subseteq X$ and a weak $\lambda^+$-measure space $(X,\MMM, \mu)$, define
    \[
    \MMM\downarrow Y=\{M\cap Y\mid M\in \MMM\}
    \]
    and let $\mu\downarrow Y$ be the partial function given by
    \[
    (\mu\downarrow Y)(A)=\inf\{\mu(M)\mid M\in \MMM, M\cap Y=A\}
    \]
    for all $A\in \MMM\downarrow Y$ such that the above $\inf$ exists.
\end{definition}

When we extend a measure from $Y$ to $X$ using $\uparrow$, the set $X\setminus Y$ is always essentially null. 
For these sets, the operations $\uparrow$ and $\downarrow$ are then the inverses of each other.

\begin{lemma}\label{lem:uparrow_and_downarrow_inverse_of_eachother}
Let $Z\subseteq Y \subseteq X$ be sets and let $(Y,\MMM, \mu:\MMM\to \SS)$ be a weak $\lambda^+$-measure space such that $Z$ is essentially null in it.
Then, we have 
\[
\MMM=(\MMM\uparrow X)\downarrow Y, \qquad \mu=(\mu\uparrow X)\downarrow Y,
\]
and
\[
\MMM\subseteq (\MMM\downarrow Z)\uparrow Y, \qquad \mu\subseteq (\mu\downarrow Z)\uparrow Y.
\]
\end{lemma}

\begin{proof}
    The relations $\MMM=(\MMM\uparrow X)\downarrow Y$ and $\MMM\subseteq (\MMM\downarrow Z)\uparrow Y$ follow directly from the definition. For the equality $\mu=(\mu\uparrow X)\downarrow Y$ it is enough to notice that
    \begin{align*}
    ((\mu\uparrow X)\downarrow Y)(A)&=\inf\{(\mu\uparrow X)(M)\mid M\in \MMM\uparrow X, M\cap Y=A\}=\\
    &=\inf\{\mu(A)\mid M\in \MMM\uparrow X, M\cap Y=A\}=\mu(A)
    \end{align*}
    for every $A\in \MMM$ (and in particular the $\inf$ exists for every $A\in \MMM$).
    Similarly, for the inclusion $\mu\subseteq (\mu\downarrow Z)\uparrow Y$, the fact that $Z$ is essentially null implies that $\mu(M)=\mu(A)$ for every $A,M\in \MMM$ such that $M\cap Z=A\cap Z$, and thus 
    \begin{align*}
    ((\mu\downarrow Z)\uparrow Y)(A)=(\mu\downarrow Z)(A\cap Z)&=\inf\{\mu(M)\mid M\in \MMM, M\cap Z=A\cap Z\}=\\
    &=\inf\{\mu(A)\mid M\in \MMM, M\cap Z=A\cap Z\}=\mu(A)
    \end{align*}
    for every $A\in \MMM$,  and in particular the $\inf$ exists for every $A\in \MMM$.
\end{proof}

The following lemmata give sufficient conditions for $(X,\MMM\uparrow X, \mu\uparrow X)$ or $(Y,\MMM \downarrow Y, \mu\downarrow Y)$ to be weak $\lambda^+$-measure spaces.

\begin{lemma}\label{lem:extending_measure_spaces}
    For every topological space $X$ and every weak $\lambda^+$-measure space $(Y,\MMM, \mu)$ such that $Y\subseteq X$, we have that $(X,\MMM\uparrow X)$ is a weakly $\lambda^+$-measurable space, and $\mu\uparrow X$ satisfies Axioms~\ref{ax_measure:emptyset}-\ref{ax_measure:additive}.
    
    Furthermore, if $Y$ is closed in $X$, then $(X,\MMM\uparrow X, \mu\uparrow X)$ is a weak $\lambda^+$-measure space.
\end{lemma}

\begin{proof}
Since $\MMM$ is closed under unions of size $\leq\lambda$, so is $\MMM\uparrow X$. Besides, $\MMM\uparrow X$ contains all open sets of $X$, since the restriction of an open set of $X$ to $Y$ is still open in $Y$ and $\MMM$ contains all open sets of $Y$.
It is also immediate to check that $\mu\uparrow X$ satisfies Axioms~\ref{ax_measure:emptyset}-\ref{ax_measure:additive}, since $\mu$ satisfies them. 

Now if $Y$ is closed in $X$, then for every $x\in X$ we have that either $x\in Y$ and we are done because $\mu$ satisfies Axiom~\ref{ax_measure:point-regular}; or $x\notin Y$ and thus 
\[(\mu\uparrow X)(x)=0=(\mu\uparrow X)(X\setminus Y),\] 
which shows once again that $\mu\uparrow X$ satisfies Axiom~\ref{ax_measure:point-regular} for $x$.
\end{proof}

\begin{lemma}\label{lem:measure_induced_on_clopen_subspace}
    Let $(X,\MMM, \mu)$ be a weak $\lambda^+$-measure space such that $\MMM$ is closed under finite intersections. Let $Y\in \MMM$ be a set of positive measure $\mu(Y)>0$.
    Then, $(Y,\MMM \downarrow Y, \mu\downarrow Y)$ is a weak $\lambda^+$-measure space.
\end{lemma}

\begin{proof}
    Since $Y\in \MMM$ and $\MMM$ is closed under intersections, we get $M\cap Y\in \MMM$ for every $M\in \MMM$, and so $\MMM\downarrow Y=\{A\in \MMM\mid A\subseteq Y\}\subseteq \MMM$.
    Also, $\mu(M\cap Y)\leq \mu(M)$ for every $M\in\MMM$ by Axiom~\ref{ax_measure:decreasing}, thus $\inf\{\mu(M)\mid M\in \MMM, M\cap Y=A\}=\mu(A)$ exists for every $A\in \MMM\downarrow Y$, and so $\mu\downarrow Y= \mu\restriction(\MMM\downarrow Y)$. Therefore, it is immediate to check that $(Y,\MMM\downarrow Y)$ is a weakly $\lambda^+$-measurable space, since $(X,\MMM)$ is, and that $\mu\downarrow Y$ is a $\lambda^+$-measure on it, since $\mu$ is.
\end{proof}

\begin{lemma}\label{lem:measure_induced_on_essentially_null_set}
    Let $(X,\MMM, \mu)$ be a $\lambda^+$-partitionable, weak $\lambda^+$-measure space, and let $Y\subseteq X$ be such that $X\setminus Y$ is essentially null.
    Then, $\MMM \downarrow Y \subseteq \dom(\mu\downarrow Y)$ and $(Y,\MMM \downarrow Y, \mu\downarrow Y)$ satisfies Axioms~\ref{ax_measure:emptyset}-\ref{ax_measure:additive}.
\end{lemma}

\begin{proof}
    By definition of essentially null set, we get that $\mu(M)=\mu(M')$ for all $M,M'\in \MMM$ such that $M\cap Y=M'\cap Y$, thus $\mu\downarrow Y$ is defined on the whole $\MMM\downarrow Y$. 
    It is then immediate to check that $(Y,\MMM\downarrow Y)$ is a weakly $\lambda^+$-measurable space, since $(X,\MMM)$ is.
    It is also easy to check that $\mu\downarrow Y$ satisfies  Axioms~\ref{ax_measure:emptyset} and \ref{ax_measure:non-trivial}, since $\mu$ does.

    Let $A,B\in \MMM \downarrow Y$ be such that $A\subseteq B$, and let $M_A, M_B\in \MMM$ be such that $A=M_A\cap Y$ and $B=M_B\cap Y$. Notice that $Y\cap(M_A\cup M_B)=B$, and $M_A\cup M_B\in \MMM$ since $\MMM$ is closed under unions. Thus, we may assume $M_A\subseteq M_B$. 
    Then, by Axiom~\ref{ax_measure:decreasing} applied to $\mu$ we have
    \[
    (\mu\downarrow Y)(A)=\mu(M_A)\leq \mu(M_B)=(\mu\downarrow Y)(B)
    \] 
    and Axiom~\ref{ax_measure:decreasing} holds for $\mu\downarrow Y$ as well.

    Now let $\A\subseteq \MMM\downarrow Y$ be a family of disjoint measurable sets of size $|\A|<\lambda^+$. Let $\A'=\{M_A\mid A\in \A\}\subseteq \MMM$ be a family of measurable sets such that $M_A\cap Y=A$ for every $A\in \A$. Since $(X,\MMM,\mu)$ is $\lambda^+$-partitionable, there is a family of disjoint sets $\P\subseteq \MMM$ refining $\A'$ and such that $\bigcup \P=\bigcup \A'$. Since $\P$ refines $\A'$, using the axiom of choice we can partition $\P$ into families $\{\P_A\mid A\in \A\}$ such that $P\in\P_A$ implies $P\subseteq M_A$.
    Notice that since $\A$ is made of disjoint sets, we have that $P\cap A\neq \emptyset$ implies $P\in \P_A$ for every $P\in \P$. 
    Since $Y\cap \bigcup \P=\bigcup \A$, then we get $Y\cap \bigcup \P_A=A$.
    Therefore, 
    \begin{align*}
        (\mu\downarrow Y) (\bigcup \A)&=\mu(\bigcup\A')=\mu(\bigcup_{A\in \A} \bigcup\P_A)= \sum_{A\in \A}\mu(\bigcup\P_A)=\\
        &=\sum_{A\in \A}(\mu\downarrow Y)(\bigcup\P_A\cap Y)=\sum_{A\in \A}(\mu\downarrow Y)(A)
    \end{align*}
    and Axiom~\ref{ax_measure:additive} is satisfied.
\end{proof}

Without additional assumptions, Axiom~\ref{ax_measure:point-regular} may fail in $(Y, \MMM \downarrow Y, \mu \downarrow Y)$. 
However, in practice, Axiom~\ref{ax_measure:point-regular} holds for all relevant pairs $Y \subseteq X$. 
For instance, if every point that is measurable in $\MMM \downarrow Y$ is also measurable in $\MMM$, the axiom follows directly from the definition of $\mu \downarrow Y$. 
Similarly, if $\MMM$ is minimal, any point that is measurable in $\MMM \downarrow Y$ but not in $\MMM$ must be isolated, and hence satisfies Axiom~\ref{ax_measure:point-regular} trivially.
Altogether, we get the following.

\begin{remark}\label{rmk:measure_induced_on_essentially_null_set}
    Let $(X,\MMM, \mu)$ be a weak $\lambda^+$-measure space, and let $Y\subseteq X$.
    Assume that either $\MMM$ is minimal, or $\MMM$  measures all points of $Y$.
    Then, $(Y,\MMM \downarrow Y, \mu\downarrow Y)$ satisfies Axiom~\ref{ax_measure:point-regular}.
\end{remark}

By  Corollary~\ref{cor:partitionable_implies_subadditive} and Lemma~\ref{lem:null_are_essentially_null_if_subadditive}, every measure zero set of a $\lambda^+$-partitionable space is essentially null. Thus, by 
Lemma~\ref{lem:measure_induced_on_essentially_null_set} and Remark~\ref{rmk:measure_induced_on_essentially_null_set}, we get the following.

\begin{corollary}\label{cor:measure_induced_on_subspace_with_measure_zero_complement}
    Let $(X,\MMM, \mu)$ be a $\lambda^+$-partitionable, weak $\lambda^+$-measure space, and let $Y\subseteq X$ be such that $X\setminus Y\in \MMM$, and $\mu(X\setminus Y)=0$.
    Assume that one of the following holds:
\begin{itemize}
    \item $\MMM$ is minimal; 
    \item $\MMM$  measures all points of $Y$.
\end{itemize}
    Then, $(Y,\MMM \downarrow Y, \mu\downarrow Y)$ is a weak $\lambda^+$-measure space.
\end{corollary}

Recall that minimal $\lambda^+$-measure spaces are closed under finite intersections, by Proposition~\ref{prop:minimal_measurable_spaces_closed_under_intersections}, and are  $\lambda^+$-partitionable, by Proposition~\ref{prop:minimal_measure_spaces_are_partitionable}. Thus, by Lemma~\ref{lem:measure_induced_on_clopen_subspace} and Corollary~\ref{cor:measure_induced_on_subspace_with_measure_zero_complement}, we get the following.

\begin{corollary}\label{cor:measure_induced_on_subspace_of_minimal_spaces}
    Let $(X,\MMM, \mu)$ be a minimal $\lambda^+$-measure space, and let $Y\subseteq X$ be such that one of the following holds: 
    \begin{itemize}
        \item $Y\in \MMM$ and $\mu(Y)>0$,
        \item $X\setminus Y\in \MMM$ and $\mu(X\setminus Y)=0$.
    \end{itemize}
    Then, $(Y,\MMM \downarrow Y, \mu\downarrow Y)$ is a minimal $\lambda^+$-measure space.
\end{corollary}

\section{On the existence of (trivial) \texorpdfstring{$\lambda^+$}{lambda+}-measures}\label{sec:on-the-existence-of-trivial-measures}

In this section, we provide explicit examples of $\lambda^+$-measures on (subspaces of) $\pre{\kappa}{\lambda}$ that are, in one way or another, not really satisfactory for the purposes of generalized descriptive set theory.
This will show the optimality of the assumptions in Theorem~\ref{thm:impossibility_theorem}.

\subsection{Dirac-like measures and the exceptionality of the countable case}\label{subsec:Dirac}
In this subsection, we study measures with \enquote{small} support. 

Recall that the support of a (classical) measure is the set of all points that do not have an open neighborhood of measure zero. A \markdef{Dirac measure} is a measure whose support has size $1$.

While useful in many applications, Dirac measures alone cannot produce a theory comparable, for example, to that of the Lebesgue measure on the reals. For this reason, when studying structural properties of Polish spaces, one usually regards all measures with finite support as trivial.

In analogy, in the generalized context we should view as trivial those $\lambda^+$-measures whose support has weight strictly less than $\lambda$. This can be formalized as follows.

\begin{definition}
Let $(X,\MMM,\mu)$ be a weak $\lambda^+$-measure space. The \markdef{support} of $\mu$ is the set  
\[
\supp(\mu)=X\setminus \bigcup\{O\mid O\text{ open}, \mu(O)=0\}.
\]
\end{definition}

\begin{definition}\label{def:Dirac}
Let $\gamma$ be a cardinal and let $(X,\MMM,\mu)$ be a weak $\lambda^+$-measure space.
We say that $\mu$ is \markdef{$\gamma$-Dirac} if $\supp(\mu)$ has weight strictly less than $\gamma$.
\end{definition}

Since for finite (Hausdorff) spaces the weight and the size of the space coincide, it follows that, under this terminology, Dirac measures are exactly the $2$-Dirac ones, and $\omega$-Dirac measures are precisely the finite sums of Dirac measures.
On the other hand, the $\omega_1$-Dirac measures already include all classical measures on (separable) Polish spaces, and in particular they may be continuous and are not necessarily countable sums of Dirac measures.

Of course, we cannot rule out the existence of ($2$-)Dirac or $\omega$-Dirac $\lambda^+$-measures on $X\subseteq \pre{\kappa}{\lambda}$; but any $\omega$-Dirac ($\lambda^+$-)measure is necessarily not continuous, so their existence does not contradict Theorem~\ref{thm:impossibility_theorem}.

However, unlike the classical case, $\omega$-Dirac measures may not be the only \enquote{trivial} (i.e., $\lambda$-Dirac) measures on $\pre{\kappa}{\lambda}$.
In the following result we show that, if $\omega<\lambda<\ccc$, then there also exist \emph{continuous} $\lambda$-Dirac measures on $\pre{\kappa}{\lambda}$, and thus the hypothesis \enquote{$\lambda\geq\ccc$} is necessary in Theorem~\ref{thm:impossibility_theorem}.

\begin{proposition}\label{prop:existence_of_countable_induced_measure}
Assume $\lambda^{<\kappa}=\lambda<\ccc$. Then, there is a continuous $\omega_1$-Dirac weak $\lambda^+$-measure space $(\pre{\kappa}{\lambda},\MMM, \mu)$ on $\pre{\kappa}{\lambda}$. 
\end{proposition}

\begin{proof}
Notice that if $\lambda<\ccc$, then $\lambda^{<\kappa}=\lambda$ is equivalent to $\kappa=\omega$.
Let $\mu'$ be a classical continuous measure on $\pre{\omega}{\omega}$.
Let $\MMM$ be the minimal $\lambda^+$-measurable structure on $\pre{\omega}{\omega}$ measuring all its points, and for every $A\in\MMM$, denote by $\int{A}$ the interior of $A$ in $\pre{\omega}{\omega}$, i.e., the union of all open subsets $U\subseteq \pre{\omega}{\omega}$ that are contained in $A$.

\begin{claim}
    The function $\mu:\MMM\to \RR_\infty$ defined by $\mu(M)=\mu'(\int{M})$ for every $M\in \MMM$ is a continuous $\lambda^+$-measure on $(\pre{\omega}{\omega},\MMM)$.
\end{claim}

\begin{proof}
Indeed, $\mu(\emptyset)=0$ and $\mu(\pre{\omega}{\omega})=\mu'(\pre{\omega}{\omega})>0$, and $\mu$ is decreasing since $\mu'$ is.
Therefore, $\mu$ satisfies Axioms~\ref{ax_measure:emptyset}--\ref{ax_measure:decreasing}.
Moreover, for every point $x\in\pre{\omega}{\omega}$ we have $\mu(x)=\mu'(\emptyset)=0=\mu'(x)$, by definition of $\mu$ and by continuity of $\mu'$.
It follows that $\mu$ is continuous and satisfies Axiom~\ref{ax_measure:point-regular}, since $\mu(U)=\mu'(U)$ for every open set $U\subseteq\pre{\omega}{\omega}$.

It only remains to prove Axiom~\ref{ax_measure:additive}. Let $(A_i)_{i<\alpha}\subset \MMM$ be a family of disjoint sets of order type $\alpha< \lambda^+$. 
Notice that either $\int{A_i}=\emptyset$ or $|\int{A_i}|\geq \ccc>\lambda$, since every non-empty open set of $\pre{\omega}{\omega}$ has size $\ccc$, and $\int{A_i}\neq \emptyset$ happens for at most countably many indices $I\subseteq \alpha$, since $\pre{\omega}{\omega}$ is second countable.

We claim $\mu'(\int{\bigcup_{i<\alpha} A_i})=\mu'(\bigcup_{i< \alpha} \int{A_i})$.
Notice that for every $i<\alpha$ we have $|A_i\setminus \int{A_i}|\leq \lambda$,  because $\MMM$ is minimal. Since 
\[
\bigcup_{i< \alpha}( \int{A_i})\subseteq \int{\bigcup_{i<\alpha} A_i}\subseteq \bigcup_{i<\alpha} A_i,
\]
we get 
\begin{align*}
|\int{\bigcup_{i<\alpha} A_i}\setminus \bigcup_{i< \alpha}\int{A_i}| \leq |\bigcup_{i<\alpha} (A_i\setminus \int{A_i})|\leq\sum_{i<\alpha} |A_i\setminus \int{A_i}|\leq \lambda\cdot \lambda=\lambda.
\end{align*}
Since $
\int{\bigcup_{i<\alpha} A_i}$ and $\bigcup_{i< \alpha}\int{A_i}$ are both open in $\pre{\omega}{\omega}$, $\lambda<\ccc$, and Borel subsets of $\pre{\omega}{\omega}$ have the perfect set property, we have 
\[
|\int{\bigcup_{i<\alpha} A_i}\setminus \bigcup_{i< \alpha} \int{A_i}|\leq \omega,
\]
thus $\mu'(\int{\bigcup_{i<\alpha} A_i}\setminus \bigcup_{i< \alpha} \int{A_i})=0$ and $\mu'(\int{\bigcup_{i<\alpha} A_i})=\mu'(\bigcup_{i< \alpha} \int{A_i})$.

Therefore, 
\begin{align*}
\mu(\bigcup_{i<\alpha} A_i)&=
\mu'(\int{\bigcup_{i<\alpha} A_i})=\\
&=\mu'(\bigcup_{i< \alpha} \int{A_i})=\\
&=\mu'(\bigcup_{i\in I} \int{A_i})=\\
&=\sum_{i\in I}
\mu'(\int{A_i})=
\sum_{i<\alpha} \mu(A_i),
\end{align*}
as wanted.
\end{proof}

This proves that if $\lambda<\ccc$ and $\kappa=\omega$, then there exists a continuous minimal $\lambda^+$-measure structure $(\pre{\omega}{\omega},\MMM,\mu)$. Notice that $\pre{\omega}{\omega}$ is closed in $\pre{\kappa}{\lambda}$. 
This implies that $(\pre{\kappa}{\lambda},\MMM\uparrow \pre{\kappa}{\lambda},\mu\uparrow\pre{\kappa}{\lambda})$ is a weak $\lambda^+$-measure space, by Lemma~\ref{lem:extending_measure_spaces}, and that $\supp(\mu\uparrow\pre{\kappa}{\lambda})=\pre{\omega}{\omega}$, thus $\mu$ is $\omega_1$-Dirac, since $\pre{\omega}{\omega}$ is second countable. It is also immediate to check that $(\pre{\kappa}{\lambda},\MMM\uparrow \pre{\kappa}{\lambda},\mu\uparrow\pre{\kappa}{\lambda})$ is continuous as well, since $(\pre{\omega}{\omega},\MMM,\mu)$ is. This concludes the proof.
\end{proof}

Remarkably, the measure obtained in Proposition~\ref{prop:existence_of_countable_induced_measure} comes from a classical measure on a Polish space that is extended to the whole space $\pre{\kappa}{\lambda}$ via the operation $\uparrow$ from Definition~\ref{def:induced_measure_on_superspace}.
Obviously, the same construction could be carried out starting from an arbitrary $\omega^+$-measure taking values in a monoid different from $\RR_\infty$. This naturally raises the question of whether there are any further possibilities, which breaks down into two separate problems:
is every $\lambda$-Dirac $\lambda^+$-measure necessarily $\omega_1$-Dirac?
And is every $\omega_1$-Dirac weak $\lambda^+$-measure space induced by some measure on a \enquote{small} subspace?

We anticipate that the results of the next sections will give a positive answer to the first question: $\omega_1$-Dirac measures are essentially the only continuous $\lambda$-Dirac measure that exist, and they exist only when $\omega<\lambda<\ccc$.  
In fact, on the one hand if $\lambda\geq\ccc$, then not even continuous \textit{$\lambda$-Dirac} $\lambda^+$-measures can exist,  by Theorem~\ref{thm:impossibility_theorem}.
On the other hand, if $\lambda<\ccc$, then every continuous $\lambda^+$-measure is a sum of disjoint $\omega_1$-Dirac $\lambda^+$-measures on its kernel, by  Corollary~\ref{cor:minimal_continuous_below_c_is_disjoint_sum_of_Dirac}.

In the following, we will give a positive answer to the second question as well, i.e., we are going to show that every $\omega_1$-Dirac weak $\lambda^+$-measure space is \enquote{essentially}
induced by some measure on a \enquote{small} subspace. To state this, however, we first need
to make every notion precise: we must specify what kind of
small space we are aiming for (separable, second countable, or Polish), what kind of
measure we consider on this small subspace (an $\omega^+$-measure or a $\lambda^+$-measure), and what we mean by \enquote{essentially}.

Let us first introduce a formal definition and compare it to the other possibilities.

\begin{definition}\label{def:countable-ind}
Given a subspace $Y\subseteq X$, a weak $\lambda^+$-measure space $(X,\MMM, \mu)$ is called \textbf{$Y$-induced} if there exists a weak $\lambda^+$-measure space $(Y,\NNN, \nu)$ such that $\MMM\subseteq \NNN \uparrow X$ and $\mu\subseteq \nu \uparrow X$.

    We say that $(X,\MMM, \mu)$ is \textbf{countable-induced} if it is $Y$-induced for some separable subspace $Y\subseteq X$.
\end{definition}

By Lemma~\ref{lem:uparrow_and_downarrow_inverse_of_eachother}, the structure $(Y,\NNN,\nu)$ can always be taken to be precisely $(Y,\MMM\downarrow Y,\mu\downarrow Y)$.
Indeed, suppose that $Y\subseteq X$ and that $(X,\MMM,\mu)$ and $(Y,\NNN,\nu)$ are weak $\lambda^+$-measure spaces satisfying $\MMM\subseteq \NNN\uparrow X$ and $\mu\subseteq \nu\uparrow X$.
Then $Y$ is essentially null in $(X,\MMM,\mu)$.
It follows that $\MMM\downarrow Y\subseteq (\NNN \uparrow X)\downarrow Y=\NNN$ and $\mu\downarrow Y\subseteq (\nu \uparrow X)\downarrow Y=\nu$, which shows that $(Y,\MMM\downarrow Y,\mu\downarrow Y)$ is itself a weak $\lambda^+$-measure space, and moreover that $\MMM\subseteq (\MMM\downarrow Y)\uparrow X$ and $\mu\subseteq (\mu\downarrow Y)\uparrow X$.

\begin{remark}
    Let $Y\subseteq X$, let $(X,\MMM,\mu)$ be weak $\lambda^+$-measure space. Then $(X,\MMM,\mu)$ is $Y$-induced if and only if $(Y,\MMM\downarrow Y,\mu\downarrow Y)$ is a weak $\lambda^+$-measure space.
\end{remark}

In particular, countable-induced weak $\lambda^+$-measure spaces $(X,\MMM,\mu)$ that come from a classical measure on $Y$ are those where $\mu$ takes values in $\SS=\RR_\infty$ and $\mu\downarrow Y$ can be extended to a $\sigma$-algebra containing $\MMM\downarrow Y$.

In Definition~\ref{def:countable-ind}, we used the density (\textit{separable}) as a way to measure smallness. However, we could have used, equivalently, the weight (\textit{second countable}), or even the cellularity,
by Proposition~\ref{prop:weight=cellularity}.
Moreover, every second countable subspace of $\pre{\kappa}{\lambda}$ is homeomorphic to a subspace of the Baire space $\pre{\omega}{\omega}$.
When $\kappa>\omega$, this is because $\pre{\kappa}{\lambda}$ is $\omega_1$-additive, and hence all its second countable subspaces are necessarily countable and discrete; when $\kappa=\omega$, then $\pre{\kappa}{\lambda}$ and all its subspaces are metrizable and zero-dimensional, and the conclusion follows, for example, from \cite[Theorem~7.8]{KechrisMR1321597}.
In summary, we obtain the following.

\begin{remark}\label{rmk:separable_iff_second_countable_iff_subspace_Baire}
    For $Y\subseteq \pre{\kappa}{\lambda}$, the following are equivalent: 
    \begin{itemize}
        \item $Y$ is separable,
        \item $Y$ is second countable,
        \item $Y$ is homeomorphic to a subspace of $\pre{\omega}{\omega}$.
    \end{itemize}
\end{remark}

In particular, a weak $\lambda^+$-measure space $(X,\MMM, \mu)$ is {countable-induced} if and only if it is $Y$-induced for some second-countable subspace $Y\subseteq X$, if and only it is $f[Y]$-induced for some subspace $Y\subseteq \pre{\omega}{\omega}$ and  embedding $f:Y\to X$.

Regarding the choice of the complexity of $Y$ in $X$, we have the following.

\begin{proposition}\label{prop:countable_induced_iff_Polish_induced}
    Let $(X,\MMM, \mu)$ be a continuous weak $\lambda^+$-measure space. 
    Then, $(X,\MMM, \mu)$  is countable-induced if and only if it is $Y$-induced for some separable closed $Y\subseteq X$.

    In particular, if $X$ is $\Gdelta{}$ in $\pre{\kappa}{\lambda}$, then $(X,\MMM, \mu)$ is countable-induced if and only if it is $Y$-induced by a \textnormal{Polish} subspace $Y\subseteq X$.
\end{proposition}

\begin{proof}
Let $Y\subseteq X$ be a separable subspace such that $(Y,\NNN, \nu)$ is a weak $\lambda^+$-measure space and $\MMM\subseteq \NNN \uparrow X$ and $\mu\subseteq \nu \uparrow X$.
Then, $\cl(Y)$ is separable as well. Also, by Definition~\ref{def:induced_measure_on_superspace} of $\uparrow$ we get $\MMM\subseteq \NNN \uparrow X=(\NNN \uparrow \cl(Y))\uparrow X$ and $\mu\subseteq \nu \uparrow X=(\nu \uparrow \cl(Y))\uparrow X$.

Thus, to prove the first part of the statement it is enough to check that $(\cl(Y),\NNN\uparrow \cl(Y), \nu\uparrow \cl(Y))$ is a weak $\lambda^+$-measure space.
By Lemma~\ref{lem:extending_measure_spaces}, we get that Axioms~\ref{ax_measure:emptyset}-\ref{ax_measure:additive} are satisfied.
Now for every $x\in \cl(Y)$, since $(X,\MMM, \mu)$ is continuous we have $x\in \MMM$ and $\mu(x)=0$, thus
        \[
        \mu(x)=\inf_{i<\gamma}\mu(A_i)=\inf_{i<\gamma}(\nu\uparrow \cl(Y))(A_i)
        \]
for every (open) local basis $\{A_i\mid i<\gamma\}$ at $x$ of order type $\gamma< \lambda^{+}$, and Axiom~\ref{ax_measure:point-regular} holds as wanted.

For the second part of the statement, recall that every countable, discrete space is Polish. Thus, when $\kappa>\omega$, any separable subspace of $X\subseteq \pre{\kappa}{\lambda}$ is Polish.
If instead $\kappa=\omega$ and $X$ is $\Gdelta{}$ in $\pre{\omega}{\lambda}$, then every closed subspace of $X$ is $\Gdelta{}$ in $\pre{\omega}{\lambda}$ as well, thus Polish (see, e.g., \cite[Theorem~4.3.23]{EngelkingGenTopMR1039321}).
\end{proof}

We note that for $\gamma$-Dirac measures as well we can relax the hypotheses on the complexity of the subset of small weight. 

\begin{proposition}
Let $(X,\MMM, \mu)$ be a weak $\lambda^+$-measure space, and let $\gamma$ be a cardinal. Then, $\mu$ is $\gamma$-Dirac if and only if there is a subspace $Y\subseteq X$ (not necessarily closed) of weight $\weight(Y)<\gamma$ such that $X\setminus Y$ is measurable and $\mu(X\setminus Y)=0$.
\end{proposition}

\begin{proof}
    If $\mu$ is $\gamma$-Dirac, then $\supp(\mu)\subseteq X$ has weight $\weight(\supp(\mu))<\gamma$ by definition, and $X\setminus \supp(\mu)$ is measurable since it is open, thus $\mu(X\setminus Y)=0$ by Corollary~\ref{cor:sets_of_measure_zero_are_lambda_ideal}.

    On the other hand, given a subspace $Y\subseteq X$ of weight $\weight(Y)<\gamma$ such that $X\setminus Y$ is measurable and $\mu(X\setminus Y)=0$, we have that $X\setminus \cl(Y)$ is measurable as well, since it is open, and $\mu(X\setminus \cl(Y))\leq \mu(X\setminus Y)=0$ by Axiom~\ref{ax_measure:decreasing}. Since for any subspace of $\pre{\kappa}{\lambda}$ the weight and the density coincide, by Proposition~\ref{prop:weight=cellularity}, and any set has the same density of its closure, we get $\weight(\cl(Y))=\density(\cl(Y))=\density(Y)=\weight(Y)<\gamma$. The fact that $X\setminus \cl(Y)$ is an open set of measure zero implies that  $\supp(\mu)\subseteq \cl(Y)$, thus $\weight(\supp(\mu))\leq \weight(\cl(Y))<\gamma$, as wanted.
\end{proof}

Finally, regarding the choice between $\omega^+$-measures and $\lambda^+$-measures in Definition~\ref{def:countable-ind}, note that we adopted the more restrictive notion, since every $\lambda^+$-measure is automatically an $\omega^+$-measure. 
Therefore, showing that every $\omega_1$-Dirac measure is $Y$-induced by a weak $\lambda^+$-measure space for some separable $Y\subseteq X$ is more informative than showing that every $\omega_1$-Dirac measure is \enquote{$Y$-induced} by a weak $\omega^+$-measure space for some separable subspace $Y\subseteq X$. 

Nevertheless, we notice that the two choices are equivalent for $\lambda^+$-partitionable weak $\lambda^+$-measure spaces.

\begin{proposition}
    Let $(X,\MMM, \mu)$ be a $\lambda^+$-partitionable weak $\lambda^+$-measure space. Then $(X,\MMM, \mu)$ is countable-induced if and only if there exists a weak $\omega^+$-measure space $(Y,\NNN, \nu)$ for some $Y\subseteq X$ such that $\MMM\subseteq \NNN \uparrow X$ and $\mu\subseteq \nu \uparrow X$. 
\end{proposition}

\begin{proof}
    One direction is clear, as every weak $\lambda^+$-measure is automatically a weak $\omega^+$-measure as well. 
    So assume there exists a weak $\omega^+$-measure space $(Y,\NNN, \nu)$ for some $Y\subseteq X$ such that $\MMM\subseteq \NNN \uparrow X$ and $\mu\subseteq \nu \uparrow X$.
    Let $\NNN'=\MMM \downarrow X$, by assumption we have $\NNN'\subseteq \NNN$, and $\MMM\subseteq \NNN' \uparrow X$. Therefore, it is enough to prove that $(Y,\NNN', \nu)$ is a weak $\lambda^+$-measure space. 
    
    Since  $\MMM$ is closed under unions of size $\lambda$, we get that $\NNN'=\MMM\downarrow Y$ is as well, and thus $(Y,\NNN')$ is a weakly $\lambda^+$-measurable space. Also, $\nu$ satisfies Axioms~\ref{ax_measure:emptyset}-\ref{ax_measure:decreasing} and Axiom~\ref{ax_measure:point-regular}, since it is an $\omega^+$-measure. 
    So we only need to check Axiom~\ref{ax_measure:additive}. Let $(A_i)_{i<\gamma}$ be a family of disjoint sets of $\NNN'$ of length $\gamma<\lambda^+$. Then, there is a family of sets $(A'_i)_{i<\gamma}$ in $\MMM$ such that $A'_i\cap Y=A_i$ for every $i<\gamma$. Since $(X,\MMM, \mu)$ is $\lambda^+$-partitionable and $\gamma<\lambda^+$, there is a family $\P$ of disjoint sets of $\MMM$ of size $|\P|<\lambda^+$ refining $\{A'_i\mid i<\gamma\}$. Let $\P_0=\{P\in \P\mid P\cap Y\subseteq A_0\}$ and for $0<i<\gamma$, let $\P_i=\{P\in \P\mid P\cap Y\subseteq A_i, P\cap Y\neq \emptyset\}$. Since both $\P$ and $\{A_i=A'_i\cap Y\mid i<\gamma\}$ are made of disjoint sets, and $\P$ refines $\{A'_i\mid i<\gamma\}$, we get that $\{\P_i\mid i<\gamma\}$ is a partition of $\P$ and $\bigcup \P_i\cap Y=A_i$.
    Therefore, 
    \begin{align*}    
    \nu(\bigcup_{i<\gamma} A_i)=\nu(\bigcup\P\cap Y)&=\mu(\bigcup\P)=\sum_{i<\gamma}\mu(\bigcup \P_i)=\sum_{i<\gamma}\nu(\bigcup \P_i\cap Y)=\sum_{i<\gamma}\nu(A_i)
    \end{align*}
    as wanted.
\end{proof}

Without any assumption on $\MMM$, we cannot expect an ($\omega_1$-Dirac) weak $\lambda^+$-measure space $(X,\MMM,\mu)$ to be countable-induced, simply because weakly $\lambda^+$-mea\-sur\-able spaces allow too much
freedom in adding sets: we can always \enquote{dirty} the measurable structure $\MMM$ by
adding sets to it in such a way that it cannot be restricted to any separable subset.

\begin{proposition}\label{prop:ugly_dirac_measure}
    Assume $\omega<\lambda<\ccc$. 
Then, there is an  $\omega_1$-Dirac, continuous weak $\lambda^+$-measure space $(\pre{\omega}{\lambda},\MMM,\mu)$ on $\pre{\omega}{\lambda}$ that is not countable-induced.
\end{proposition} 

\begin{proof}   
    By Proposition~\ref{prop:existence_of_countable_induced_measure} and its proof, there is an $\omega_1$-Dirac weak $\lambda^+$-measure space  $(X,\MMM,\mu)$ on $X=\pre{\omega}{\lambda}$ such that $\MMM=\MMM^\lambda_p(\pre{\omega}{\omega})\uparrow X$ is the extension to $X$ of the minimal $\lambda^+$-measurable structure $\MMM^\lambda_p(\pre{\omega}{\omega})$ on $\pre{\omega}{\omega}$ measuring all points of $\pre{\omega}{\omega}$, and $\mu(M)=\mu_c(\int[\pre{\omega}{\omega}]{M\cap \pre{\omega}{\omega}})$ for every $M\in\MMM$, where $\mu_c$ is a classical measure on $\pre{\omega}{\omega}$.

    Let $D_0\subseteq\pre{\omega}{\omega}$ be a (Bernstein) set such that for every uncountable closed $C\subseteq\pre{\omega}{\omega}$ we have $|D_0\cap C|=|C\setminus D_0|=\ccc$. This can be built in the same way as a classical Bernstein set (see, e.g., \cite[Example 8.24]{KechrisMR1321597}): let $\{C_\xi:\xi<\ccc\}$ be an enumeration listing $\ccc$-many times every uncountable closed subsets of $\pre{\omega}{\omega}$. Pick recursively elements $a_\xi,b_\xi\in C_\xi$ so that all these points are pairwise distinct. Then $D_0=\{a_\xi\mid \xi<\ccc\}$ is as required.

    Let $D_1=\pre{\omega}{\omega}\setminus D_0$, then we have $|D_1\cap C|=|C\setminus D_1|=\ccc$ as well for every uncountable closed $C\subseteq\pre{\omega}{\omega}$. Let $D_2=\pre{\omega}{\omega}$, and define
    \[
    \D=\{(\pre{\omega}{\lambda}\setminus Y)\cup D_i\mid i\in \{0,1,2\}, \pre{\omega}{\omega}\subseteq Y\subseteq \pre{\omega}{\lambda}, Y \text{ separable}\}.
    \]
    Then $\D$ is closed under arbitrary unions, and thus 
    \[
    \MMM'=\MMM\cup\{D\cup M\mid M\in \MMM, D\in\D\}
    \]
    is a weakly $\lambda^+$-measurable structure on $X$.
    
    Define $\mu'(M)=\mu(\int[\pre{\omega}{\omega}]{M\cap \pre{\omega}{\omega}})$ for every $M\in \MMM'$.
    Notice that this way, we have $\mu'(M)=\mu(M)$ for every $M\in\MMM$.
    It is easy to check then that $\mu'$ satisfies Axioms~\ref{ax_measure:emptyset}-\ref{ax_measure:non-trivial} and Axiom~\ref{ax_measure:point-regular}, since $\mu$ does. Also, for all $A,B\in \MMM$ such that $A\subseteq B$, we get $\int{A}\subseteq \int{B}$ as well, and thus $\mu'(A)=\mu(\int{A})\leq \mu(\int{B})=\mu'(B)$ and Axiom~\ref{ax_measure:decreasing} holds.
    
    To prove Axiom~\ref{ax_measure:additive}, first we claim that $\mu'(M\cup D_i)=\mu(M)$ for every $M\in \MMM$ and $i\in\{0,1\}$.
    By definition of $\mu'$, this amounts to check that $\mu(A)=\mu(B)$ for $A=\int[\pre{\omega}{\omega}]{(M\cup D_i)\cap \pre{\omega}{\omega}}$ and $B=\int[\pre{\omega}{\omega}]{M\cap \pre{\omega}{\omega}}$.
    Let $F$ be the set of points $x\in \pre{\omega}{\omega} \setminus B$ such that there is open neighborhood $V$ of $x$ in $\pre{\omega}{\omega}$ satisfying $|V\setminus B|\leq \omega$. Then $F$ is countable, since $\pre{\omega}{\omega}$ is second countable.
    By construction, if $x\in \pre{\omega}{\omega}\setminus (F\cup B)$, then every clopen neighborhood $V$ of $x$ in $\pre{\omega}{\omega}$ satisfies $|V\setminus B|\geq \ccc$. 
    Since $V\setminus  B$ is closed in $\pre{\omega}{\omega}$, we get that $|(V\setminus  B)\setminus D_i|\geq \ccc$. 
    Since $|(M\cap \pre{\omega}{\omega})\setminus \int[\pre{\omega}{\omega}]{M\cap \pre{\omega}{\omega}}|\leq \lambda<\ccc$ for every $M\in \MMM$, by definition of minimal $\lambda^+$-measure space, we get that $|(V\setminus  M)\setminus D_i|\geq \ccc$. Therefore, $V$ is not contained in $M\cup  D_i$. Since this holds for an arbitrary clopen neighborhood of $x$, and $\pre{\omega}{\omega}$ is zero-dimensional, we get $x\notin A$.
    This shows that $A\subseteq F\cup B$, thus $
    |A\setminus B|\leq \omega$.
    Therefore, $\mu(A)=\mu(B)$, as wanted.

    Now let $(A_i)_{i<\gamma}$ be a sequence of disjoint measurable sets of length $\gamma<\lambda^+$.
    Notice that $\lambda>\omega$ implies that $\pre{\omega}{\lambda}$ is not separable, so there are no pairwise disjoint sets from $\D$. Therefore, at most one set of $(A_i)_{i<\gamma}$ belongs to $\MMM'\setminus \MMM$: without loss of generality, assume $A_0\in \MMM'\setminus \MMM$. 
    Since $A_0\notin \MMM$ and by definition of $\MMM$ and $\MMM'$, we get that $A_0=A'_0\cup D_i$ for some $i\in\{0,1\}$ and $A_0\in \MMM$.
    Then, $\bigcup_{i<\alpha} A_i=A'_0\cup \bigcup_{0<i<\alpha} A_i \cup D_i$, and thus 
    \[
    \mu'(\bigcup_{i<\alpha} A_i)=\mu'(A'_0\cup \bigcup_{0<i<\alpha} A_i\cup D_i)=\mu(A'_0\cup \bigcup_{0<i<\alpha} A_i)
    \]
    by the previous argument and since $A'_0\cup \bigcup_{0<i<\alpha} A_i\in \MMM$. Since $\mu$ is a $\lambda^+$-measure and $\{A'_0\}\cup\{A_i\mid 0<i<\gamma\}$ is still a family of disjoint open sets in $\MMM$, then Axiom~\ref{ax_measure:additive} follows.

    Finally, notice that $\supp(\mu')=\supp(\mu)=\pre{\omega}{\omega}$, so $\mu'$ is still $\omega_1$-Dirac. However, for every separable subspace $Y\subseteq \pre{\omega}{\lambda}$ (and in particular for $Y=\supp(\mu')$), we have that $(Y,\MMM\downarrow Y, \mu\downarrow Y)$ cannot be a weak $\lambda^+$-measure space.
    This is because $A_0=(\pre{\omega}{\lambda}\setminus (\pre{\omega}{\omega}\cup Y))\cup D_0$ and $A_1=(\pre{\omega}{\lambda}\setminus (\pre{\omega}{\omega}\cup Y))\cup D_1$ are such that $A_0\downarrow Y=D_0$ and $A_1\downarrow Y=D_1$ are disjoint, while at the same time we have $\mu(A_0)=\mu(A_1)=0$ and $\mu(A_0\cup A_1)=\mu(\pre{\omega}{\omega})>0$ (and $\mu(M)=\mu(\pre{\omega}{\omega})>0$ for every $M\in \MMM$ such that $M\cap Y=(A_0\cup A_1)\cap Y=\pre{\omega}{\omega}$).
    Therefore $(X,\MMM, \mu)$ is not countable-induced.
\end{proof}

This is not a real obstacle, however, as it is clear that even the weak $\lambda^+$-measure space $(X,\MMM, \mu)$ from Proposition~\ref{prop:ugly_dirac_measure} is \enquote{essentially} countable-induced, in the sense that its \textit{kernel} is.
More in general, we are going to show that if $(X,\MMM, \mu)$ is $\lambda$-partitionable, then it is $\omega_1$-Dirac if and only if it is countable-induced.

\begin{proposition}\label{prop:countable-induced_iff_Dirac}
    Let $(X,\MMM, \mu)$ be a continuous $\lambda^+$-partitionable weak $\lambda^+$-measure space. Then, the following are equivalent:
    \begin{enumerate-(a)}
        \item\label{prop:countable-induced_iff_Dirac-1} $(X,\MMM, \mu)$ is countable-induced.
        \item\label{prop:countable-induced_iff_Dirac-2} $(X,\MMM, \mu)$ is $\omega_1$-Dirac.
    \end{enumerate-(a)}
\end{proposition}

\begin{proof}
    First, assume $(X,\MMM, \mu)$ is {countable-induced}. By Proposition~\ref{prop:countable_induced_iff_Polish_induced} there is a closed, separable $Y\subseteq X$ and a weak $\lambda^+$-measure space  $(Y,\NNN, \nu)$ such that $\MMM\subseteq \NNN \uparrow X$, and $\mu\subseteq \nu \uparrow X$.
    Then, $X\setminus Y$ is open, thus measurable, and $\mu(X\setminus Y)=0$, therefore $\supp(\mu)\subseteq Y$. By Remark~\ref{rmk:separable_iff_second_countable_iff_subspace_Baire}, we get that  $\supp(\mu)$ is second countable as wanted.

    Conversely, assume that  $Y=\supp(\mu)$ is second countable, thus separable. Since $(X,\MMM, \mu)$ is continuous, in particular $\MMM$ measures all points of $X$.
    By Corollary~\ref{cor:sets_of_measure_zero_are_lambda_ideal}, we get that $X\setminus Y$ has measure $0$, and thus $(Y,\MMM\downarrow Y, \mu\downarrow Y)$ is a weak $\lambda^+$-measure space by Corollary~\ref{cor:measure_induced_on_subspace_with_measure_zero_complement}. By Corollary~\ref{cor:partitionable_implies_subadditive} and Lemma~\ref{lem:null_are_essentially_null_if_subadditive} we have that $X\setminus Y$ is essentially null. The inclusions $\MMM\subseteq (\MMM\downarrow Y) \uparrow X$ and $\mu\subseteq (\mu\downarrow Y) \uparrow X$ follow then by Lemma~\ref{lem:uparrow_and_downarrow_inverse_of_eachother}.
\end{proof}

Since the kernel of a continuous, $\omega_1$-Dirac weak $\lambda^+$-measure space is a continuous, $\omega_1$-Dirac minimal $\lambda^+$-measure space, and every minimal $\lambda^+$-measure space is $\lambda^+$-partitionable by Proposition~\ref{prop:minimal_measure_spaces_are_partitionable}, we get the following.

\begin{corollary}
    Every continuous, $\omega_1$-Dirac weak $\lambda^+$-mea\-sure space  $(X,\MMM, \mu)$ has a countable-induced kernel $\ker(X,\MMM, \mu)$.
\end{corollary}

\subsection{A possibility theorem for monoids with non-continuous sum}\label{subsec:measures_with_non-continuous_sum}

In Theorem~\ref{thm:impossibility_theorem} we use an infinitary sum satisfying Axioms~\ref{ax_sum:extend_plus} and~\ref{ax_sum:continuous}.  Clearly, 
Axiom~\ref{ax_sum:extend_plus} cannot be dropped.
In the following we show that we cannot drop Axiom~\ref{ax_sum:continuous} either.

For the purpose of this subsection only and with a small abuse of notation, we call {weak $\lambda^+$-measure space} even a triple $(X,\MMM,\mu)$ that satisfies all axioms and requirements of a weak $\lambda^+$-measure space, except that $\mu:\powerset(X)\to (\SS,\Sum)$ takes values in a positively totally ordered  monoid $\SS$ with a function $\Sum$ that does not necessarily satisfy Axiom~\ref{ax_sum:continuous}.
Notice that for such a measure we cannot apply the tools we developed in Section~\ref{sec:measures}.

First, without Axiom~\ref{ax_sum:continuous}, there are subspaces of $\pre{\kappa}{\lambda}$ having a $\lambda^+$-measure.

\begin{proposition}\label{prop:trivial_measure_non_continuous_sum}
   There is a positively totally ordered monoid $\SS$ with an infinitary function $\Sum$ satisfying Axiom~\ref{ax_sum:extend_plus} but not Axiom~\ref{ax_sum:continuous}, and a subspace $X\subseteq \pre{\kappa}{\lambda}$ for which there exists a continuous weak $\lambda^+$-measure structure $(X,\MMM,\mu)$ on it.
\end{proposition}

\begin{proof}
    Let $X\subseteq \pre{\kappa}{\lambda}$ be an infinite discrete subspace of $\pre{\kappa}{\lambda}$. Consider the submonoid $\SS=\{0,\infty\}$ of $\RR_\infty$.
    Given $\bar{b}\in \pre{<\On}{\SS}$, define $\Sum(\bar{b})=0$ if $\bar{b}$ is finite and constantly $0$, and $\Sum(\bar{b})=\infty$ otherwise.
    Define $\mu:\powerset(X)\to\SS$ by setting $\mu(M)=0$ if $M$ is finite, and $\mu(M)=\infty$ otherwise. Then, it is clear that $\mu$ satisfies Axioms~\ref{ax_measure:emptyset}-\ref{ax_measure:decreasing}, and it satisfies trivially Axiom~\ref{ax_measure:point-regular} since every point is open. Also,  $\mu$ satisfies Axiom~\ref{ax_measure:additive} by the very definition of $\Sum$. Therefore, 
    $(X,\powerset(X),\mu)$ is a (weak) $\lambda^+$-measure space. 
\end{proof}

The measure constructed in Proposition~\ref{prop:trivial_measure_non_continuous_sum} is pathological in multiple ways, one of them being that it is $1$-Dirac: every point of $X$ has an open neighborhood of measure $0$, hence $\supp(\mu)=\emptyset$. This observation alone might be regarded as sufficient motivation to exclude monoids with non-continuous functions $\Sum$. Nevertheless, for completeness, we also analyze what happens when we restrict our attention to non-Dirac measures, or to measures on $X=\pre{\kappa}{\lambda}$.

First, the argument from Proposition~\ref{prop:trivial_measure_non_continuous_sum} cannot be extended to the whole space $\pre{\kappa}{\lambda}$, since no $1$-Dirac measure exists on this space.

\begin{proposition}\label{prop:no_1-dirac_measures_on_Baire}
Assume $\lambda^{<\kappa}=\lambda$. Let $\SS$ be a positively totally ordered monoid with an infinitary function $\Sum$ satisfying Axiom~\ref{ax_sum:extend_plus} (but not necessarily Axiom~\ref{ax_sum:continuous}).
Then there is no $1$-Dirac, continuous weak $\lambda^+$-measure structure on $\pre{\kappa}{\lambda}$.
\end{proposition}

\begin{proof}
Assume, toward a contradiction, that $(\pre{\kappa}{\lambda}, \MMM, \mu)$ is a continuous $1$-Dirac weak $\lambda^+$-measure structure.
In particular, the family $\V=\{V\in \tau \mid \mu(V)=0\}$ is a cover of $\pre{\kappa}{\lambda}$.
Let $\P$ be a partition of $\pre{\kappa}{\lambda}$ refining $\V$, as given by Fact~\ref{fct:strong_zero_dimensionality}.
Then, $\mu(P)=0$ for every $P\in\P$, by Axiom~\ref{ax_measure:decreasing}, and $|\P|\leq\lambda$ since we assumed $\lambda^{<\kappa}=\lambda$.
Without loss of generality, we may assume $|\P|=\lambda$, for example by splitting one clopen cone of $\P$ into its successors.

By Axioms~\ref{ax_measure:non-trivial} and~\ref{ax_measure:additive}, we obtain
\[
0<\mu(\pre{\kappa}{\lambda})
=\mu\Bigl(\bigcup_{P\in\P} P\Bigr)
=\Sum((0)_{i<\lambda}),
\]
and hence $\Sum((0)_{i<\lambda})>0$.
On the other hand, if $s\in\pre{<\kappa}{\lambda}$ is such that $\Nbhd_s\in\P$, then by Axiom~\ref{ax_measure:decreasing} we have $\mu(\Nbhd_{s\conc\alpha})=0$ for every $\alpha<\lambda$.
By Axiom~\ref{ax_measure:additive}, this yields
\[
0=\mu(\Nbhd_s)
=\mu\Bigl(\bigcup_{\alpha<\lambda}\Nbhd_{s\conc\alpha}\Bigr)
=\Sum((0)_{i<\lambda}),
\]
a contradiction.
\end{proof}

Notice that the contrast between Propositions~\ref{prop:trivial_measure_non_continuous_sum} and~\ref{prop:no_1-dirac_measures_on_Baire} also shows that the results we developed in Section~\ref{sec:measures} -- like Lemma~\ref{lem:extending_measure_spaces} -- cannot be applied in general to a function $\Sum$ that does not satisfy Axiom~\ref{ax_sum:continuous}.

Restricting our attention to non-Dirac measures, the next result shows that when $\lambda=\lambda^{\kappa}$, non-Dirac continuous weak $\lambda^+$-measure spaces cannot exist, even without assuming that $\Sum$ is continuous.

\begin{proposition}\label{prop:no_non-Dirac_measures_without_continuous_sum}
Assume $\lambda=\lambda^{\kappa}$. Let $X\subseteq \pre{\kappa}{\lambda}$, and let $\SS$ be a  
positively totally ordered monoid with an infinitary function $\Sum$ satisfying Axiom~\ref{ax_sum:extend_plus} (but not necessarily Axiom~\ref{ax_sum:continuous}).
Then there is no non-$1$-Dirac, continuous weak $\lambda^+$-measure structure on $X$.
\end{proposition}

\begin{proof}
    Assume by contradiction that $\mu$ is a continuous non-$1$-Dirac measure on $X$, i.e., all points of $X$ are measurable of measure zero, and $\supp(\mu)\neq \emptyset$.
    Let $x\in \supp(\mu)$. 
    Let $\delta=\min\{|A|\mid A\text{ open neighborhood of }x\}$. 
    Notice that for every open neighborhood $A$ of $x$, we have that $|A|\leq |\pre{\kappa}{\lambda}|=\lambda$, by $\lambda=\lambda^{\kappa}$. Also, $\mu(A)>0$, by $x\in \supp(\mu)$. Therefore, 
    \[
    \mu(A)=\Sum((0)_{i<|A|})>0
    \]
    by Axiom~\ref{ax_measure:additive}.
    In particular, this shows that $\Sum((0)_{i<\delta})>0$. 
    By the minimality of $\delta$, for every open neighborhood $B$ of $x$ we can find an open $B'\subseteq B$ such that $|B'|=\delta$, and thus 
    \[
    \mu(B)\geq \mu(B')= \Sum((0)_{i<\delta})>0
    \]
    by Axiom~\ref{ax_measure:decreasing}.
    Let $(A_i)_{i<\kappa}$ be a local basis of $x$ of length $\kappa$ (for example $A_i=\Nbhd_{x\restriction i}$ for every $i<\kappa$). 
    By the previous argument, we get
    \[
    \inf_{i<\kappa}\mu(A_i)=\Sum((0)_{i<\delta})>0=\mu(x).
    \]
    Since the hypothesis $\lambda=\lambda^{\kappa}$ implies $\kappa<\lambda$, this shows that $\mu$ does not satisfy Axiom~\ref{ax_measure:point-regular}, contradiction.
\end{proof}

In particular, Propositions~\ref{prop:no_1-dirac_measures_on_Baire} and~\ref{prop:no_non-Dirac_measures_without_continuous_sum} together show that, when $\lambda=\lambda^{\kappa}$, there is no continuous weak $\lambda^+$-measure structure on $\pre{\kappa}{\lambda}$ at all.

\begin{corollary}
Assume $\lambda=\lambda^{\kappa}$. Let $\SS$ be a  
positively totally ordered monoid with an infinitary function $\Sum$ satisfying Axiom~\ref{ax_sum:extend_plus} (but not necessarily Axiom~\ref{ax_sum:continuous}).
Then there is no continuous weak $\lambda^+$-measure structure on $\pre{\kappa}{\lambda}$.
\end{corollary}

Under the additional hypothesis $\lambda<\lambda^{\kappa}$ (which always holds when $\kappa=\cof(\lambda)$, our main case of interest), using a non-continuous function $\Sum$ instead makes it possible to obtain measures on the whole space $\pre{\kappa}{\lambda}$.

\begin{theorem}\label{thm:possibility_theorem_non_continuous_sum}
Assume $\lambda^{<\kappa}=\lambda<\lambda^{\kappa}$. There is a non-$\lambda$-Dirac, continuous weak $\lambda^+$-measure structure on $\pre{\kappa}{\lambda}$ for some positively totally ordered monoid $\SS$ with an infinitary function $\Sum$ satisfying Axiom~\ref{ax_sum:extend_plus} but not Axiom~\ref{ax_sum:continuous}.
\end{theorem}

\begin{proof}
Let $X=\pre{\kappa}{\lambda}$.
Let $\LL=\kappa^\ast=\{\frac{1}{\alpha}\mid \alpha\in \kappa\}$ be $\kappa$ with reverse order, i.e. the total order given by $\frac{1}{\alpha}\leq \frac{1}{\beta}$ if $\alpha\geq \beta$. Then, $\LL=(\LL,\max,\leq)$ is a totally ordered semigroup.

Let $\{s_\alpha\mid \alpha<\lambda\}$ be a well-ordering of $\pre{<\kappa}{\lambda}$.
Consider $\SS'=\LL\times\NN^\lambda$ with lexicographic order and point-wise sum (where $\NN=(\omega,0,+,\leq)$ represent as usual the positively totally ordered monoid of natural numbers), and let $\SS=\SS'\cup\{e\}$ be the monoid obtained by adding an identity $e$ (and minimum of the order) to the semigroup $\SS'$. Then 
\[
\SS=(\{\frac{1}{\alpha}\mid \alpha\in \kappa\},e, \ast, \leq)
\]
is a positively totally ordered monoid. The operation $\ast$ can be defined explicitly by 
\[
\left(\frac{1}{\alpha},(n_\epsilon)_{\epsilon<\lambda}\right)\ast \left(\frac{1}{\beta},(m_\epsilon)_{\epsilon<\lambda}\right)= \left(\frac{1}{\min(\alpha,\beta)},(n_\epsilon+m_\epsilon)_{\epsilon<\lambda}\right)
\]
and the order is the lexicographic one.
Notice that $\Deg(\SS)=\kappa$, as $(\frac{1}{\alpha}, (0)_{\epsilon<\lambda})_{\alpha<\kappa}$ is strictly decreasing and coinitial in $\SS$.

For every open set $O$, define
\[
    \Exp(O)=\Exp_{\aleph_0}(O)=\bigcup\{\Nbhd_s\mid s\in \pre{<\kappa}{\lambda}, |\{ \alpha<\lambda\mid \Nbhd_{s\conc \alpha}\cap O\neq\emptyset \}|\geq \aleph_0 \}
\]
to be the \textit{($\aleph_0$-)expansion}\footnote{The choice of $\aleph_0$ is arbitrary: the same argument would work using $\Exp_{\delta}(O)$ for any infinite regular cardinal $\delta \leq \lambda$.} of $O$. 
Notice that $\Exp(\Exp(O))=\Exp(O)\supseteq O$, and $\Exp(\Nbhd_s)=\Nbhd_s$ for every $s\in \pre{<\kappa}{\lambda}$.
Notice that 
\begin{equation}\label{eq:strange_condition-cup_exp}    
\Exp(O_0\cup...\cup O_n)=\Exp(O_0)\cup...\cup\Exp(O_n) 
\end{equation}
for every finite family of open sets $\{O_0,...,O_n\}$.
Let
\[
h(O)=\min\{\beta\mid \Nbhd_s\subseteq O \text{ for some } s\in \pre{\beta}{\lambda}\}
\]
be the {height} of $O$. 
Let 
\[
\A(O)=\{s\in \pre{<\kappa}{\lambda}\mid \Nbhd_s\subseteq O\};
\]
this way, $O=\bigcup_{s\in \A(O)}\Nbhd_s$. Define
\[
\chi_O(\alpha)=\begin{cases}
    0 \qquad &\text{ if } s_\alpha\notin \A(O),\\
    1 &\text{ otherwise}
\end{cases}
\]
to be the \textit{characteristic function} of $\A(O)$.
Notice that for every finite family of open sets $\{O_1,...,O_n\}$ and every $\Nbhd_s\subseteq \bigcup_{i\leq n} O_i$, there is $i\leq n$ such that $\Nbhd_s\subseteq \Exp(O_i)$. This and equation~\eqref{eq:strange_condition-cup_exp} together imply that
\begin{equation}\label{eq:strange_condition-sum_chi}
\chi_{O_1\cup... \cup O_n}=\chi_{O_1}+... +\chi_{O_n}
\end{equation}
for every finite family of disjoint open sets $\{O_1,...,O_n\}$ such that $\Exp(O_i)=O_i$ for every $i\leq n$.

For every $b=(\frac{1}{\alpha},(n_\epsilon)_{\epsilon<\lambda})\in \SS$, let $O(b)=\bigcup\{\Nbhd_{s_\epsilon}\mid n_\epsilon>0\}$ (where we set $\bigcup\emptyset=\emptyset$).
It is clear then that
\begin{equation}\label{eq:strange_condition-chi_inverse_O}
    O\left(\frac{1}{\alpha},\chi_U\right)=U
\end{equation}
for every open set $U\in\tau$ and every $\alpha<\kappa$.

Given $\bar{b}=(b_\epsilon)_{\epsilon<\nu}\in \pre{<\On}{\SS}$, define
\[
    \Sum(\bar{b})=
    \begin{cases}
        b_0 \ast ... \ast b_\nu  \qquad &\text{ if } \nu \text{ is finite},\\
        \left(\frac{1}{h(\Exp(\bigcup_{\epsilon<\nu} O(b_\epsilon)))},\chi_{\Exp(\bigcup_{\epsilon<\nu} O(b_\epsilon))}\right) 
        &\text{ otherwise.}
    \end{cases}
    \]

We say that $b=(\frac{1}{\alpha},(n_i)_{i<\lambda})\in \SS$ is \textbf{adequate} if it is of the form $(\frac{1}{h(O)},\chi_O)$ for some $O\in \tau$ such that $O=\Exp(O)$ (necessarily, $O=\Exp(O(b))$).
Given $\bar{b}=(b_\epsilon)_{\epsilon<\nu}\in \pre{<\On}{\SS}$, for every $\epsilon<\nu$ let $b_\epsilon=(\frac{1}{\alpha_\epsilon},(n_i^\epsilon)_{i<\lambda})\in \SS$ denote the $\epsilon$-coordinate of $\bar{b}$.
We say that $\bar{b}$ is \textbf{good} if every $b_\epsilon$ is adequate and for every $\epsilon\neq \epsilon'$ we have $O(b_\epsilon)\cap O(b_{\epsilon'})=\emptyset$. 

Then, by equations~\eqref{eq:strange_condition-cup_exp} and~\eqref{eq:strange_condition-sum_chi} and we get
\begin{equation}\label{eq:strange_condition-finite_sum}
b_0 \ast ... \ast b_\nu=\left(\frac{1}{h(\Exp(\bigcup_{\epsilon<\nu} O(b_\epsilon)))},\chi_{\Exp(\bigcup_{\epsilon<\nu} O(b_\epsilon))}\right)
\end{equation}
for every finite, good $\bar{b}=(b_\epsilon)_{\epsilon<\nu}\in \pre{<\On}{\SS}$.

Notice that $\Sum$ is not continuous. For example, assume without loss of generality that $s_i=\langle i\rangle$ for every $i<\omega$, and let $A_\alpha=\bigcup_{i<\alpha}\Nbhd_{\langle i\rangle}$ for every $\alpha\leq \omega$. Then, we have $\Exp(A_n)=A_n$ for every $i\leq n$ (by equation~\eqref{eq:strange_condition-cup_exp}), while $\Exp(A_\omega)=X$. Therefore, we get
\[
\Sum_{i<\omega}s_i=\left(\frac{1}{h(X)},\chi_{X}\right)=\left(\frac{1}{0},(1)_{i<\lambda}\right),
\]
\[
 \sup_{n<\omega}\Sum_{i<n}s_i=\sup_{n<\omega} \left(\frac{1}{h\left(A_n\right)},\chi_{A_n}\right)
 =\left(\frac{1}{h\left(A_\omega\right)},\chi_{A_\omega}\right)
 =\left(\frac{1}{1},\chi_{A_\omega}\right).
\]

We show now that there is a continuous weak $\lambda^+$-measure structure $(X,\MMM,\mu)$ where the measure takes values in $\SS$ with the  (\textit{non-continuous!}) sum $\Sum$ we defined.

Let $\MMM$ be the minimal $\lambda^+$-measurable structure on $X$ measuring all points and open sets of $X$. 
We define a measure $\mu:\MMM\to \SS$ by setting $\mu(M)=e$ for every $M\in \MMM$ with $\int{M}=\emptyset$, and 
\[
\mu\left(M\right)=\left(\frac{1}{h\left(\Exp\left(\int{M}\right)\right)},\chi_{\Exp\left(\int{M}\right)}\right)
\]
for every other $M\in\MMM$. It follows by definition that $\mu(M)$ is adequate for every $M\in \MMM$ with $\int{M}\neq\emptyset$.

We claim $\mu$ is as wanted.

Indeed, Axioms~\ref{ax_measure:emptyset}-\ref{ax_measure:non-trivial} follow by definition. 

Axiom~\ref{ax_measure:decreasing} follows from the fact that $\Exp(U)\subseteq\Exp(V)$, $\chi_U\leq \chi_V$, and $h(U)\geq h(V)$ whenever $U\subseteq V$.

In order to prove Axiom~\ref{ax_measure:additive}, notice first that given two open sets $O,U$, we have either $O\cap U=\emptyset$ or $|O\cap U|=\lambda^\kappa>\lambda$. Thus if $O$ and $U$ differ from a set of size $\leq \lambda$, we have $\Exp(O)=\Exp(U)$.
In particular, this shows that 
\begin{equation}\label{eq:strange_condition-measures_with_small_difference}
\mu(O\cup D)=\mu(O)=\mu(O)\ast e=\mu(O)\ast \mu(D)
\end{equation}
for every open $O$ and every set $D$ of size $\leq \lambda$ disjoint from $O$.

Secondly, notice that 
\begin{equation}\label{eq:strange_condition_union}
    \Exp(\bigcup_{\alpha<\gamma} O_\alpha)=\Exp(\bigcup_{\alpha<\gamma} \Exp(O_\alpha))
\end{equation} 
for every family of open sets $\{O_\alpha\mid \alpha<\gamma\}$. 
The inclusion $\subseteq$ is clear. 
For the other inclusion $\supseteq$, notice that $\Exp(\bigcup_{\alpha<\gamma} O_\alpha)\supseteq\bigcup_{\alpha<\gamma} \Exp(O_\alpha)$, thus 
\[
\Exp(\bigcup_{\alpha<\gamma} O_\alpha)=\Exp(\Exp(\bigcup_{\alpha<\gamma} O_\alpha))\supseteq\Exp(\bigcup_{\alpha<\gamma} \Exp\left(O_\alpha\right)),
\]
as wanted.

Now given a family of disjoint measurable sets $(A_\alpha)_{\alpha<\gamma}$ of length $\gamma<\lambda^+$, let $O_\alpha=\int{A_\alpha}$, let $O=\bigcup_{\alpha<\gamma}O_\alpha$ and $M=\bigcup_{\alpha<\gamma} A_\alpha$. By definition of $\MMM$, we have $|M\setminus O|\leq \lambda$, thus $\mu(M)=\mu(O)$ by equation~\eqref{eq:strange_condition-measures_with_small_difference}.
Since the $A_\alpha$ are disjoint and $\mu(A_\alpha)$ is adequate for every $\alpha<\gamma$, it follows from equation~\eqref{eq:strange_condition-chi_inverse_O} that $(\mu(A_\alpha))_{\alpha<\gamma}$ is good.
Therefore, applying the definition of $\Sum$ if $\gamma$ is infinite, or equation~\eqref{eq:strange_condition-finite_sum} if $\gamma$ is finite, and using equations~\eqref{eq:strange_condition_union} and~\eqref{eq:strange_condition-chi_inverse_O}, we get
\begin{align*}
\mu(M)&=\mu(O)=\\
&=\left(\frac{1}{h(\Exp(O))},\chi_{\Exp(O)}\right)=\\
&=\left(\frac{1}{h(\Exp(\bigcup_{\alpha<\gamma} O_\alpha))},\chi_{\Exp(\bigcup_{\alpha<\gamma} O_\alpha)}\right)= \\
&=\left(\frac{1}{h(\Exp(\bigcup_{\alpha<\gamma} \Exp(O_\alpha)))},\chi_{\Exp(\bigcup_{\alpha<\gamma} \Exp(O_\alpha))}\right)= \\
&=\Sum_{\alpha<\gamma} \left(\frac{1}{h(\Exp(O_\alpha))},\chi_{\Exp(O_\alpha)}\right)=\\
&=\Sum_{\alpha<\gamma} \mu(A_\alpha)
\end{align*}
which proves Axiom~\ref{ax_measure:additive}.

For Axiom~\ref{ax_measure:point-regular}, it is enough to notice that for every $x\in X$, we have that $(\mu(\Nbhd_{x\restriction \alpha}))_{\alpha<\kappa}=(\frac{1}{\alpha}, \chi_{\Nbhd_{x\restriction \alpha}})_{\alpha<\kappa}$ is coinitial in $\SS^+$.

Finally, notice that $\mu(O)\neq e$ for every non-empty, open set $O\in \MMM$, thus $\supp(\mu)=\pre{\kappa}{\lambda}$. Since $\pre{\kappa}{\lambda}$ has weight $\lambda$, we get that $\mu$ is not $\lambda$-Dirac, as wanted.
\end{proof}

In a certain sense, this result points in the opposite direction of Theorem~\ref{thm:impossibility_theorem}, apparently suggesting that some analogues of measures could be defined in this setting.
However, Theorem~\ref{thm:possibility_theorem_non_continuous_sum} should be compared more closely to Proposition~\ref{prop:small_measures_big_space} than to Theorem~\ref{thm:impossibility_theorem}. The similarity between the proofs of Theorem~\ref{thm:possibility_theorem_non_continuous_sum} and Proposition~\ref{prop:small_measures_big_space} shows that the measure obtained in Theorem~\ref{thm:possibility_theorem_non_continuous_sum} is essentially a $\delta$-measure\footnote{In the specific case of the measure $\mu$ constructed in the proof of Theorem~\ref{thm:possibility_theorem_non_continuous_sum}, an $\omega$-measure.} for some $\delta<\lambda^+$, which is forced to behave like a $\lambda^+$-measure by making it \enquote{jump} when needed through the non-continuity of $\Sum$.
For this reason, it should not be regarded as a natural generalization of classical measures.

More in general, we observe that without continuity of the sum, we can hardly expect a measure to satisfy any reasonable form of regularity.

\section{On the non-existence of continuous \texorpdfstring{$\lambda^+$}{lambda+}-measures}\label{sec:on-the-non-existence-of-continuous-measures}

In the following section we are going to prove the Impossibility Theorem (Theorem~\ref{thm:impossibility_theorem}), the main result of our paper.

We are going to use multiple times the following easy corollary of Corollary~\ref{cor:sets_of_measure_zero_are_lambda_ideal}, so we state it explicitly to avoid repetitions. 

\begin{corollary}\label{cor:no_partition_into_cones_of_measure_zero}
    Let $(X,\MMM,\mu)$ be a weak $\lambda^+$-measure space and let $U\subseteq X$ be an open set in $X$ such that $\mu(U)>0$. Then, for every family of open sets $\V$ of $X$ of measure zero there is $x\in U\setminus \bigcup \V$.
\end{corollary}

\begin{proof}
By Corollary~\ref{cor:sets_of_measure_zero_are_lambda_ideal} we have that $\mu(\bigcup \V)=0$, thus $U\nsubseteq \bigcup \V$ by Axiom~\ref{ax_measure:decreasing}.
\end{proof}

We split the proof in multiple cases, depending on the algebraic structure of $\SS$ and on the value of $\kappa$.

\subsection{Measures in monoids of wrong degree}

\begin{lemma}\label{lem:impossibility_for_monoids_of_wrong_degree}
Assume $\Deg(\SS)\neq \kappa$. Then,  for every weak $\lambda^+$-measure space $(X,\MMM,\mu)$ and for every open set $U$ of positive measure $\mu(U)>0$ there is $x\in U$ such that either $x$ is not measurable or $\mu(x)>0$.
\end{lemma}

\begin{proof}
Let $\V=\{V\in \tau\mid \mu(V)=0\}$ be the set of open sets of $X$ of measure $0$.  
By Corollary~\ref{cor:no_partition_into_cones_of_measure_zero}, there is $x\in U\setminus \bigcup\V$. By definition of $\V$, we get $\mu(A)>0$ for every $A\in \tau$ such that $x\in A$.
In particular, for every $\alpha<\beta<\kappa$ we have $0<\mu(\Nbhd_{x\restriction \beta}(X))\leq \mu(\Nbhd_{x\restriction \alpha}(X))$ by Axiom~\ref{ax_measure:decreasing}, thus the sequence $\langle \mu(\Nbhd_{x\restriction \alpha}(X))\mid \alpha<\kappa\rangle$ is decreasing in $\SS^+$. 
Now, if $x$ is measurable of measure $0$, this also implies that $\langle \mu(\Nbhd_{x\restriction \alpha}(X))\mid \alpha<\kappa\rangle$ is coinitial in $\SS^+$ by Axiom~\ref{ax_measure:point-regular}, hence $\Deg(\SS)=\kappa$.
\end{proof}

\subsection{Measures in non-\texorpdfstring{$0$}{0}-continuous monoids}

Recall that a positively totally ordered monoid $\SS$ is not $0$-continuous if and only if it contains an element $c$ such that $b+b\geq c$ for all $b\in \SS^+$. 

\begin{lemma}\label{lem:impossibility_for_non-0-continuous}
    Let $(X,\MMM,\mu)$ be a weak $\lambda^+$-measure space. Let $U\subseteq X$ be an open subset of $X$ of positive measure $\mu(U)>0$ whose value satisfies $b+b\geq \mu(U)$ for every $b\in \SS^+$. 
    Then, one of the following holds:
    \begin{itemize}
        \item $U$ contains a point $x\in U$ that is not measurable, 
        \item $U$ contains a measurable point $x\in U$  of measure $\mu(x)=\mu(U)$,
        \item $U$ contains two measurable points $x,y\in U$ such that $\mu(x)+\mu(y)=\mu(U)$.
    \end{itemize} 
\end{lemma}

\begin{proof}
    If $U$ contains a non-measurable point we are done, so assume instead that all points of $U$ are measurable. 
    Also, if there is $x\in U$ such that $\mu(U)=\mu(O)$ for every open neighborhood $O$ of $x$ contained in $U$,
    then we get $\mu(x)=\mu(U)$, by Axiom~\ref{ax_measure:point-regular}, and we are done. So let $c=\mu(U)$ and suppose that every $x\in U$ has an open neighborhood $O$ of measure $\mu(O)<c$.

Let $\V=\{V\in \tau\mid \mu(V)=0\}$ be the set of open sets of $X$ of measure $0$. Then, $U\setminus \bigcup \V$ is non-empty, by Corollary~\ref{cor:no_partition_into_cones_of_measure_zero}.
    Let $x_1\in U\setminus \bigcup \V$.
    Then, $\mu(x_1) < c$, by Axiom~\ref{ax_measure:decreasing} and since $\mu(O)<c$ for some neighborhood $O$ of $x_1$.
    This also implies that $\mu(U\setminus \{x_1\})>0$, since otherwise 
    \[
    \mu(U)=\mu(x_1)+\mu(U\setminus \{x_1\})=\mu(x_1)+0=\mu(x_1)<c=\mu(U),
    \]
    contradiction.
    Thus, there is $x_2\in (U\setminus \{x_1\})\setminus\bigcup \V$, by Corollary~\ref{cor:no_partition_into_cones_of_measure_zero}.
    Let $N_1, N_2\subseteq U$ be disjoint cones such that $x_i\in N_i$ and $0<\mu(N_i)<c$ for each $i\in\{1,2\}$.
    
    Fix $i\in\{1,2\}$. Notice that
    $\mu(V)=0$ for every clopen $V\subseteq N_i\setminus \{x_i\}$, for otherwise both $\mu(V)\in \SS^+$ and $\mu(N_i\setminus V)\in \SS^+$, and thus $\mu(N_i)=\mu(V)+ \mu(N_i\setminus V)\geq c>\mu(N_i)$ by assumption on $c$, contradiction. 
    Then, $\mu(N_i)=\mu(O) + \mu(N_i\setminus O)= \mu(O)$ for every clopen neighborhood $O$ of $x_i$ contained in $N_i$, and thus $\mu(x_i)=\mu(N_i)$, as before.
    
    Then, by Axiom~\ref{ax_measure:decreasing} and by the assumption on $c$, we get 
    \[
    \mu(U)=c\leq \mu(N_1)+\mu(N_2)= \mu(N_1\cup N_2)\leq \mu(U).
    \]
    and thus $\mu(U)=\mu(N_1)+\mu(N_2)=\mu(x_1)+\mu(x_2)$ as wanted.
\end{proof}

\begin{corollary}\label{cor:impossibility_for_non-0-continuous}
    Assume $\SS$ is not $0$-continuous. Then, for every weak $\lambda^+$-measure space $(X,\MMM,\mu)$ and for every open set $U$ of positive measure $\mu(U)>0$ there is $x\in U$ such that either $x$ is not measurable or $\mu(x)>0$.
\end{corollary}

\begin{proof}
    Let $c$ be such that $b+b\geq c$ for every $b\in \SS^+$, that is, a witness of the non-$0$-continuity of $\SS$.
    Assume by contradiction that every point of $U$ is measurable of measure $0$.
    Then, $\V=\{V\in \tau\mid V\subseteq U, \mu(V)\leq c\}$ is a cover of $U$, by Axiom~\ref{ax_measure:point-regular}.
    Let also $\V'=\{V\in \tau\mid \mu(V)=0\}$ be the set of open sets of $X$ of measure $0$, then we can find $U'\in \V\setminus \V'$, by Corollary~\ref{cor:no_partition_into_cones_of_measure_zero}.
    The result then follows from Lemma~\ref{lem:impossibility_for_non-0-continuous} applied to $U'$. 
\end{proof}

\subsection{Measures in non-Archimedean monoids}

\begin{lemma}\label{lem:impossibility_for_non-Archimedean_monoids}
Let $(X,\MMM,\mu)$ be a weak $\lambda^+$-measure space. 
Assume $U\in \tau$ and $a,b\in \SS$ are such that $a\ll b<\mu(U)$.
Then, there is $x\in U$ such that either $x$ is not measurable or it has measure $\mu(x)\geq a$.
\end{lemma}

\begin{proof}
Assume that $x\in \MMM$ for every $x\in U$, as otherwise we are done.
Define $\V=\{V\in \tau\mid \mu(V)\leq a\}$. 
Suppose by contradiction that $\V$ covers $U$.
By Fact~\ref{fct:strong_zero_dimensionality} we can find a clopen partition $\P$ of $U$  refining $\V$. By Axioms~\ref{ax_measure:decreasing} and~\ref{ax_measure:additive} and by Proposition~\ref{prop:sum_is_natural_infinitary_sum}\ref{prop:sum_is_natural_infinitary_sum-2}, it follows that 
\[
\mu(U)= \sum_{P\in \P} \mu(P)\leq \sum_{i<\lambda} a\leq b<\mu(U),
\]
contradiction.
Thus $\V$ cannot cover $U$. Let $x\in U\setminus \bigcup \V$. This implies $\mu(A)>a$ for every $A\in \tau$ such that $x\in A$. Therefore, we have $\mu(x)=\inf\{A\in \tau\mid x\in A\}\geq a$ by Axiom~\ref{ax_measure:point-regular}, as wanted.
\end{proof}

Notice that if $\bar{a}=\langle a_i\mid i<\delta\rangle$ is a nowhere Archimedean decreasing sequence of limit length that is coinitial in $\SS^+$, then we have that for every $c>0$ there are $i, j\in \SS^+$ such that $0<a_i\ll a_j< c$. Thus, we get the following.

\begin{corollary}\label{cor:impossibility_for_non-Archimedean_monoids}
    Assume $\SS^+$ has infinite degree and contains a nowhere Archimedean decreasing coinitial sequence. Then, for every weak $\lambda^+$-measure space $(X,\MMM,\mu)$ and for every open set $U$ of positive measure $\mu(U)>0$ there is $x\in U$ such that either $x$ is not measurable or $\mu(x)>0$.
\end{corollary}

\subsection{Dirac-like measures above the continuum}

\begin{lemma}\label{lem:impossibility_for_countable_induced}
    Assume $\lambda\geq \ccc$. Let $(X,\MMM,\mu)$ be an $\omega_1$-Dirac, weak $\lambda^+$-measure space. 
    Then, for every open set $U$ of positive measure $\mu(U)>0$ there is $x\in U$ such that either $x$ is not measurable or $\mu(x)>0$.
\end{lemma}

\begin{proof}
    Assume that $(X,\MMM,\mu)$ is $\omega_1$-Dirac, and thus $\supp(\mu)$ is second countable.
    Recall that $\supp(\mu)$ is closed, thus $U\setminus \supp(\mu)$ is measurable.
    By Corollary~\ref{cor:sets_of_measure_zero_are_lambda_ideal}, we get that $\mu(X\setminus \supp(\mu))=0$, thus $\mu(U\setminus \supp(\mu))=0$ as well by Axiom~\ref{ax_measure:decreasing}.
    By Remark~\ref{rmk:separable_iff_second_countable_iff_subspace_Baire}, we get that $|\supp(\mu)|\leq |\pre{\omega}{\omega}|\leq \ccc\leq \lambda$.
    If all points of $U$ are measurable of measure $0$, then we have that $\A=\{U\setminus \supp(\mu)\}\cup\{\{x\}\mid x\in U\cap \supp(\mu)\}$ is a cover of $U$ of size $|\A|\leq \lambda$ made of disjoint sets of measure $0$. 
    By Axiom~\ref{ax_measure:additive}, we would get
    \[
    \mu(U)=\mu(U\setminus \supp(\mu)) + \sum_{x\in U\cap \supp(\mu)}\mu(x)=0,
    \]
    contradiction.
\end{proof}

\subsection{Proof of the Impossibility Theorem}

We prove the Impossibility Theorem~\ref{thm:impossibility_theorem} in a stronger, although more technical, form.

First, the results obtained so far already allow us to conclude that when $\kappa$ is uncountable, there is no continuous weak $\lambda^+$-measure space $(X,\MMM,\mu)$ on any subspace $X\subseteq \pre{\kappa}{\lambda}$ of weight $\lambda$.

\begin{proposition}\label{lem:impossibility_for_kappa_uncountable}
    Assume $\kappa>\omega$.
    Then, there is no continuous weak $\lambda^+$-measure space $(X,\MMM,\mu)$.
\end{proposition}

\begin{proof}
If $\SS$ is not $0$-continuous or has finite degree, the conclusion follows from Corollary~\ref{cor:impossibility_for_non-0-continuous}.
If $\Deg(\SS)\neq\kappa$, the conclusion follows from Lemma~\ref{lem:impossibility_for_monoids_of_wrong_degree}.
Thus, we may assume that $\Deg(\SS)=\kappa>\omega$ and that $\SS$ is $0$-continuous.
Then, by Lemma~\ref{lem:three_types_of_monoids} $\SS$ contains a coinitial nowhere Archimedean sequence, and hence by Corollary~\ref{cor:impossibility_for_non-Archimedean_monoids} we are done.
\end{proof}

When $\kappa=\omega$, continuous weak $\lambda^+$-measure spaces may exist.
The following results show, however, that in this case the measure is always trivial, in the sense that it is a sum of $\omega_1$-Dirac measures.

\begin{lemma}\label{lem:splitting_into_dirac}
    Let $(X,\MMM,\mu)$ be a continuous weak $\lambda^+$-measure space. 
    Then, there is a clopen partition $\P$ of $X$ such that $\supp(\mu)\cap P$ is second countable for every $P\in \P$.
\end{lemma}

\begin{proof}
    By Corollary~\ref{cor:impossibility_for_non-0-continuous}, we get that $\SS$ is $0$-continuous and of infinite degree. By Corollary~\ref{cor:impossibility_for_non-Archimedean_monoids}, $\SS$ contains no coinitial nowhere Archimedean sequence.
    Thus, by Lemma~\ref{lem:three_types_of_monoids}, we have that $\Deg(\SS)=\omega$ and $\SS$ is initially Archimedean.
    By Lemma~\ref{lem:impossibility_for_monoids_of_wrong_degree}, we have that $\kappa=\Deg(\SS)=\omega$.
    
    Say that $\mu$ is \markdef{essentially countable on $V$} if for every family $\A$ of disjoint open subsets of $V$  there are at most countably many elements of $\A$ of positive measure.
    
    \begin{claim}\label{claim:every_partitionable_measure_with_few_disjoint_sets_of_positive_measure_is_countable_induced}
    Let $\kappa=\omega$. Let $(X,\MMM,\mu)$ be a weak $\lambda^+$-measure space, and let $V\subseteq X$ be open.
    Suppose that $\mu$ is essentially countable on $V$.
    Then, $\supp(\mu)\cap V$ is second countable.
\end{claim}

\begin{proof}
    Let $T=\{s\in \pre{<\omega}{\lambda}\mid \mu(\Nbhd_s)>0, \Nbhd_s\subseteq V\}$.
    By assumption, each level $\Lev_n(T)=\{s\in \pre{n}{\lambda}\mid \mu(\Nbhd_s)>0, \Nbhd_s\cap X\subseteq V\}$ must be countable, since $\{\Nbhd_s\mid s\in \pre{n}{\lambda}\}$ is a family of disjoint open sets.
    This implies that $T=\bigcup_{n\in\omega} \Lev_n(T)$ is countable as well. 
    Since $\{\Nbhd_s\mid s\in \pre{<\omega}{\lambda}\}$ is a basis for the topology of $X$, and $V$ is open, we have that $\supp(\mu)\cap V=[T]$, and that $\{\Nbhd_s\mid s\in T\}$ is a basis for $\supp(\mu)$. Therefore, $\supp(\mu)$ is second countable as wanted. 
\end{proof}

    Let $\V=\{V\in \tau\mid \mu\text{ is essentially countable on } V\}$. Assume first that
    $\V$ does not cover $X$.
    Let $a, b\in \SS^+$ be such that the interval $[0,b]$ is Archimedean and $0<a<b$.
    Let $\A=\{A\in \tau\mid \mu(A)<a\}$, then we have that $\A$ is a cover of $X$, by Axiom~\ref{ax_measure:point-regular} and since $\mu(x)=0$ for every $x\in X$.
    Therefore, there is $A\in \A\setminus \V$. Since $\mu$ is not essentially countable on $A$, there is an uncountable family $\P$ of disjoint clopen subsets of $A$ such that $0<\mu(P)$ for every $P\in \P$. By Axiom~\ref{ax_measure:decreasing}, we also have $\mu(P)\leq\mu(A)<a$ for every $P\in \P$.
    By the pigeonhole principle, since $\Deg(\SS)=\omega$, there is $c\in (0,a)$ such that $\mu(P)>c$ for uncountably many $P\in \P$.
    Since $[0,b]$ is Archimedean and $c\in [0,b]$, we can find $n\in \omega$ such that $n\cdot c\geq b$.
    But then, 
    \[
    \mu(A)\geq \mu(\bigcup \P)=\sum_{P\in \P}\mu(P)\geq n\cdot c\geq b>a,
    \]
    by Axioms~\ref{ax_measure:decreasing} and~\ref{ax_measure:additive}, contradicting that $A\in \A$.

    Thus $\V$ does cover $X$. By Fact~\ref{fct:strong_zero_dimensionality}, we can find a clopen partition $\P$ of $X$ refining $\V$.
    It is easy to check that $\mu$ is still essentially countable on every $P\in \P$, since $\P$ refines $\V$.
    By Claim~\ref{claim:every_partitionable_measure_with_few_disjoint_sets_of_positive_measure_is_countable_induced}, for every $P\in\P$ we have that $\supp(\mu)\cap P$ is second countable, as wanted. 
\end{proof}

\begin{proposition}\label{prop:splitting_into_countable_induced_for_initially_Archimedean_monoids}
    Let $(X,\MMM,\mu)$ be a continuous weak $\lambda^+$-measure space such that $\MMM$ is closed under finite intersections. 
    
    Then, there is a clopen partition $\P$ of $X$ such that $(P,\MMM\downarrow P,\mu\downarrow P)$ is an $\omega_1$-Dirac, continuous, weak $\lambda^+$-measure space for every $P\in \P$ and $\mu$ is the sum of $(\mu\downarrow P)_{P\in \P}$, i.e., for every $M\in\MMM$ we have
    \[
    \mu(M)=\sum_{P\in \P} (\mu\downarrow P)(M\cap P).
    \]
\end{proposition}

\begin{proof}
    By Lemma~\ref{lem:splitting_into_dirac}, there is a clopen partition $\P'$ of $X$ such that $\supp(\mu)\cap P$ is second countable for every $P\in \P'$. Since we assumed that $X$ has weight at most $\lambda$, we get that $|\P'|\leq \lambda$.
    Therefore, since $\P'$ covers $X$ and is made of disjoint sets, there is $A\in\P'$ such that $\mu(A)>0$, by Axioms~\ref{ax_measure:non-trivial} and~\ref{ax_measure:additive}.
    Let 
    \[
    \P=\{P\in \P'\setminus \{A\}\mid \mu(P)>0\}\cup \{ A\cup \bigcup \{P\in \P'\mid \mu(P)=0\}\}.
    \]
    This way, $\mu(P)>0$ for every $P\in P$. 
    Since $|\P'|\leq \lambda$, by Axiom~\ref{ax_measure:additive} we get that $\mu(\bigcup \{P\in \P'\mid \mu(P)=0\})=0$, and thus by Axiom~\ref{ax_measure:decreasing} we have
    \[
    \supp(\mu)\cap (A\cup \bigcup \{P\in \P'\mid \mu(P)=0\})= \supp(\mu)\cap A.
    \]
    This shows that $\supp(\mu)\cap P$ is still second countable for every $P\in \P$.
    By Lemma~\ref{lem:measure_induced_on_clopen_subspace}, we have that  
    $(P,\MMM\downarrow P,\mu\downarrow P)$ is a (continuous) weak $\lambda^+$-measure space for every $P\in \P$. Since $P$ is open, we get that $\supp(\mu\downarrow P)=\supp(\mu)\cap P$ , and thus $(P,\MMM\downarrow P,\mu\downarrow P)$  is $\omega_1$-Dirac as well.
    
    Finally, from the fact that $\MMM$ is closed under finite intersections we get that $M\cap P\in \MMM$ for every $M\in \MMM$ and $(\mu\downarrow P)(M\cap P)=\mu(M\cap P)$, and thus from Axiom~\ref{ax_measure:additive} we get 
    \[
    \mu(M)=\sum_{P\in \P} (\mu\downarrow P)(M\cap P),
    \]
    for every $M\in\MMM$, i.e. $\mu$ is the sum of $(\mu\downarrow P)_{P\in \P}$, as wanted.
\end{proof}

In particular, Proposition~\ref{prop:splitting_into_countable_induced_for_initially_Archimedean_monoids} applies to every continuous minimal $\lambda^+$-measure space, thanks to Proposition~\ref{prop:minimal_measurable_spaces_closed_under_intersections}.

\begin{corollary}\label{cor:minimal_continuous_below_c_is_disjoint_sum_of_Dirac}
    Let $(X,\MMM,\mu)$ be a continuous, minimal $\lambda^+$-measure space. Then, there is a clopen partition $\P$ of $X$ such that $(P,\MMM\downarrow P,\mu\downarrow P)$ is an $\omega_1$-Dirac, continuous weak $\lambda^+$-measure space for every $P\in \P$, and $\mu$ is the sum of the measures $(\mu\downarrow P)_{P\in \P}$.

    In particular, this holds for the kernel $\ker(X,\MMM,\mu)$ of any continuous weak $\lambda^+$-measure space $(X,\MMM,\mu)$.
\end{corollary}

Corollaries~\ref{cor:embedding_into_Baire_kappa_leq_lambda} and~\ref{cor:minimal_continuous_below_c_is_disjoint_sum_of_Dirac} and  Lemma~\ref{lem:impossibility_for_countable_induced} together lead immediately to Theorem~\ref{thm:impossibility_theorem}, as wanted.

From Theorem~\ref{thm:impossibility_theorem}, we can also obtain the following statement.

\begin{corollary}
    Assume $\lambda\geq\ccc$. Let $(X,\MMM,\mu)$ be a weak $\lambda^+$-measure space. Then, there is a partition $X=C\cup O$ such that $O$ is open, $\mu(O)=0$, and $x\in \MMM$ for every $x\in O$, while $C$ is closed and contains a dense subset $D\subseteq C$ such that for every $x\in D$, either $x$ is not measurable or $\mu(x)>0$.
\end{corollary}

\begin{proof}
    Let 
    \[
    \V=\{V\in \tau\mid \text{ There exists }x\in V, x\notin \MMM\lor \mu(x)>0 \},
    \]
    \[
    \U=\{U\in \tau\setminus \V\mid \mu(U)=0\}.
    \]
    Suppose by contradiction there is $Y\in \tau\setminus (\U\cup \V)$: then, $\mu(Y)>0$ and for every $x\in Y$, we have $x\in \MMM$ and $\mu(x)=0$.
    Without loss of generality, we may assume that $(X,\MMM,\mu)$ is minimal, up to passing to its kernel.
    By Corollary~\ref{cor:measure_induced_on_subspace_of_minimal_spaces}, we have that $(Y,\MMM\downarrow Y,\mu\downarrow Y)$ is a weak $\lambda^+$-measure space, and $\mu\downarrow Y$ is continuous since $Y\notin \V$, contradicting Theorem~\ref{thm:impossibility_theorem}.
    
    Therefore, we must have that $\U\cup \V$ is a cover of $X$. Let $O=\bigcup \U$: then, $\mu(O)=0$, by Corollary~\ref{cor:sets_of_measure_zero_are_lambda_ideal}, and furthermore $x\in \MMM$ and $\mu(x)=0$ for every $x\in O$, since by definition $U\notin \V$ for every $U\in\U$.
    Let $C=X\setminus O$.
    Now for every $V\in \tau$ such that $V\cap C\neq \emptyset$, we have that $V\in \V$, since $\U\cup \V$ covers $X$, and thus we may chose $x_V\in V$ such that $x\notin \MMM$ or $ \mu(x)>0$. Then, $D=\{x_V\mid V\in \tau, V\cap C\neq \emptyset\}$ is as wanted.
\end{proof}

\section{On the existence of \texorpdfstring{$\lambda$}{lambda}-measures and measurable cardinals}\label{sec:small_meaures_on_big_spaces}

In this section, we briefly investigate $\lambda$-measures on $\pre{\kappa}{\lambda}$ and their relationship with the measurability of $\lambda$.

The next result shows that when $\lambda$ is regular, $\lambda$-measures do exist on $\pre{\kappa}{\lambda}$.
This will allow us to complete the picture concerning the existence and non-existence of measures on $\pre{\kappa}{\lambda}$ in Section~\ref{sec:final_remarks} (see Corollary~\ref{cor:small_measures_big_space}).

The regularity of $\lambda$ is, in a certain sense, a necessary requirement: if $\lambda$ is singular, then every weak $\lambda$-measure space is also a weak $\lambda^+$-measure space. In particular, in this case the existence of a $\lambda$-measure is tied to the existence of a $\lambda^+$-measure, which we already addressed earlier.
For a singular cardinal $\lambda$, we can still show the existence of $\delta$-measures on $\pre{\kappa}{\lambda}$ for every regular cardinal $\delta\leq \lambda$.

\begin{proposition}\label{prop:small_measures_big_space}
    Let $\delta$ be a regular cardinal such that $\delta\leq \lambda$. Then there is a continuous weak $\delta$-measure space $(\pre{\kappa}{\lambda}, \MMM,\mu)$ on $\pre{\kappa}{\lambda}$ with $\supp(\mu)=\pre{\kappa}{\lambda}$.
\end{proposition}

The proof closely follows that of Theorem~\ref{thm:possibility_theorem_non_continuous_sum} and can be viewed as a more direct version of it.

\begin{proof}
Let $\LL=(\kappa+1)^\ast=\{\frac{1}{\alpha}\mid \alpha\in \kappa+1\}$ be $\kappa+1$ with reverse order, i.e. the total order given by $\frac{1}{\alpha}\leq \frac{1}{\beta}$ if $\alpha\geq \beta$. Then, $\LL=(\LL, \frac{1}{\kappa},\max,\leq)$ is a positively totally ordered monoid with identity $\frac{1}{\kappa}$. We consider the infinitary operation $\sum=\sup(=\max)$ on it,  which is total since $\LL$ is Dedekind-complete.

For every open set $O$, define
\[
    \Exp_{\delta}\left(O\right)=\bigcup\{\Nbhd_s\mid s\in \pre{<\kappa}{\lambda}, |\{ \alpha<\lambda\mid \Nbhd_{s\conc \alpha}\cap O\neq\emptyset \}|\geq \delta \}
\]
to be the \textit{$\delta$-expansion} of $O$. 
Then, we get $\Exp_{\delta}(\Exp_{\delta}(O))=\Exp_{\delta}(O)\supseteq O$, and $\Exp_{\delta}(\Nbhd_s)=\Nbhd_s$ for every $s\in \pre{<\kappa}{\lambda}$.   
Also, by construction we have
\begin{equation}\label{eq:strange_condition-cup_exp_lambda+}    
\Exp_{\delta}(\bigcup_{i<\gamma}O_i)=\bigcup_{i<\gamma}\Exp_{\delta}(O_i)
\end{equation}
for every family of open sets $\{O_i\mid i<\gamma\}$ of size $\gamma<\delta$.
 
Define
\[
h\left(O\right)=\min\{\beta<\kappa\mid \Nbhd_s\subseteq O \text{ for some } s\in \pre{\beta}{\lambda}\}.
\]
to be the \textit{height} of $O$.
We claim that
\begin{equation}\label{eq:strange_condition-h_delta}    
h(\bigcup_{i<\gamma}O_i)=\min_{i<\gamma}h(O_i)
\end{equation}
for every family of open sets $\{O_i\mid i<\gamma\}$ such that  $\gamma< \delta$ and $O_i=\Exp_{\delta}(O_i)$ for every $i<\gamma$.
Indeed, let $s\in \pre{<\kappa}{\lambda}$ be such that $\Nbhd_s\subseteq \bigcup_{i<\gamma}O_i$. 
Then, 
\[
|\{ \alpha<\lambda\mid \Nbhd_{s\conc \alpha}\cap \bigcup_{i<\gamma}O_i\neq\emptyset \}|\geq \delta,
\]
and
since $\delta=\cof(\delta)> \gamma$, there is $i<\gamma$ such that 
\[
|\{ \alpha<\lambda\mid \Nbhd_{s\conc \alpha}\cap O_i\neq\emptyset \}|\geq \delta.
\]
Thus, $\Nbhd_s\subseteq \Exp_{\delta}(O_i)=O_i$, as wanted.

Let $\MMM$ be the minimal $\delta$-measurable structure on $\pre{\kappa}{\lambda}$ measuring all points of $\pre{\kappa}{\lambda}$.
Define $\mu(M)=\frac{1}{\kappa}$ for every $M\in \MMM$ with empty interior, while for every other $M\in \MMM$ we define 
\[
\mu(M)=\frac{1}{h(\Exp_{\delta}(\int{M}))}.
\]

Then, Axioms~\ref{ax_measure:emptyset}-\ref{ax_measure:non-trivial} follow by definition. 

Axiom~\ref{ax_measure:decreasing} follows from the fact that $\Exp_{\delta}(U)\subseteq\Exp_{\delta}(V)$ and $h(U)\geq h(V)$ whenever $U\subseteq V$.

To prove Axiom~\ref{ax_measure:additive}, notice first that given two open sets $O,U$, we have either $O\cap U=\emptyset$ or $|O\cap U|\geq\lambda\geq\delta$.
Thus, if the symmetric difference between $O$ and $U$ has size $|O\triangle U|< \delta$, we have $O\cap \Nbhd_s\neq \emptyset$ if and only if $U\cap \Nbhd_s\neq \emptyset$ for every $s\in \pre{<\kappa}{\lambda}$. This shows that $|O\triangle U|< \delta$ implies $\Exp_{\delta}(O)=\Exp_{\delta}(U)$.
Since for every family of measurable sets $\{M_i\mid i<\gamma\}$ of size $\gamma< \delta$ we have 
\[
|\int{\bigcup_{\alpha<\gamma}M_i}\setminus \bigcup_{\alpha<\gamma}\int{M_i}|\leq |\bigcup_{\alpha<\gamma}M_i\setminus \bigcup_{\alpha<\gamma}\int{M_i}|\leq |\bigcup_{\alpha<\gamma}(M_i\setminus \int{M_i})|< \delta,
\]
by the definition of minimal $\delta$-measurable structure and the regularity of $\delta$, then
\begin{equation}\label{eq:exp_interiors_lambda+}
\Exp_{\delta}(\int{\bigcup_{\alpha<\gamma}M_i})=\Exp_{\delta}( \bigcup_{\alpha<\gamma}\int{M_i}).
\end{equation}

Now given a family of disjoint measurable sets $(A_\alpha)_{\alpha<\gamma}$ of length $\gamma<\delta$, by combining equations~\eqref{eq:strange_condition-cup_exp_lambda+}, \eqref{eq:strange_condition-h_delta}, and~\eqref{eq:exp_interiors_lambda+} we obtain
\begin{align*}
\mu(\bigcup_{\alpha<\gamma} A_i)&=\frac{1}{h(\Exp_{\delta}(\int{\bigcup_{\alpha<\gamma} A_i}))}=\\
&=\frac{1}{h(\Exp_{\delta}(\bigcup_{\alpha<\gamma}\int{ A_i}))}=\\
&=\frac{1}{h(\bigcup_{\alpha<\gamma}\Exp_{\delta}(\int{ A_i}))}=\\
&=\sup_{\alpha<\gamma}\frac{1}{h(\Exp_{\delta}(\int{ A_i}))}=\\
&=\sum_{\alpha<\gamma} \mu(A_i).
\end{align*}

Finally, for Axiom~\ref{ax_measure:point-regular} it is enough to notice that for every $x\in X$, we have that $(\mu(\Nbhd_{x\restriction \alpha}))_{\alpha<\kappa}=(\frac{1}{\alpha})_{\alpha<\kappa}$ is coinitial in $\SS^+$.
\end{proof}

Notably, the measure obtained in Proposition~\ref{prop:small_measures_big_space} takes values in a non-Archi\-me\-dean monoid.
Requiring a measure that takes values in $[0,1]$, instead, leads to a characterization of (real-valued) measurable cardinals.

\begin{proposition}\label{prop:real-valued_measurable_cardinal_iff_lambda_measure}
    The following are equivalent:
    \begin{enumerate-(1)}
        \item\label{prop:real-valued_measurable_cardinal_iff_lambda_measure-1} There is a non-countable-induced continuous weak $\lambda$-measure space $(X,\MMM,\mu)$ for some $X\subseteq \pre{\kappa}{\lambda}$ of weight $\lambda$ such that $\mu$ takes values in $[0,1]$.
        \item\label{prop:real-valued_measurable_cardinal_iff_lambda_measure-2} For every non-$\lambda$-Lindel\"of space $X\subseteq \pre{\kappa}{\lambda}$ of weight $\lambda$ there is $\mu:\powerset(Y)\to [0,1]$ such that $(X,\powerset(X),\mu)$ is a non-countable-induced continuous weak $\lambda$-measure space.
        \item\label{prop:real-valued_measurable_cardinal_iff_lambda_measure-3} $\lambda$ is real-valued measurable.
    \end{enumerate-(1)}
\end{proposition}

\begin{proof}
    \ref{prop:real-valued_measurable_cardinal_iff_lambda_measure-1}$\Rightarrow$\ref{prop:real-valued_measurable_cardinal_iff_lambda_measure-3}. Assume first that $(X,\MMM,\mu)$ is a non-countable-induced, continuous, weak $\lambda$-measure space such that $\mu$ takes values in $[0,1]$. 
    Without loss of generality, we may assume $\MMM$ is the minimal $\lambda$-measurable structure measuring all points of $X$.
    
    The fact that $\mu(X)=1$ and $[0,1]$ is Archimedean together imply that every family of disjoint open sets of positive measure is countable. 
    This and Proposition~\ref{prop:weight=cellularity} together, or the same argument of Claim~\ref{claim:every_partitionable_measure_with_few_disjoint_sets_of_positive_measure_is_countable_induced}, imply that $\supp(\mu)$ is second countable. 
    
    Let $Y=X\setminus \supp(\mu)$, we claim $\mu(Y)>0$. Suppose not: then, by Corollary~\ref{cor:measure_induced_on_subspace_of_minimal_spaces} we have that $(\supp(\mu),\MMM\downarrow \supp(\mu),\mu\downarrow \supp(\mu))$ is a weak $\lambda$-measure space, and by Proposition~\ref{prop:minimal_measure_spaces_are_partitionable}, Corollary~\ref{cor:partitionable_implies_subadditive}, and Lemmata~\ref{lem:null_are_essentially_null_if_subadditive} and~\ref{lem:uparrow_and_downarrow_inverse_of_eachother}, we get that $\MMM\subseteq (\MMM\downarrow \supp(\mu))\uparrow X$ and $\mu\subseteq (\mu\downarrow \supp(\mu))\uparrow X$. This shows that $(X,\MMM,\mu)$ is countable-induced, contradiction.
    
    Therefore, $\mu(Y)>0$. Since $(X,\MMM,\mu)$ is minimal and continuous, by Proposition~\ref{prop:minimal_measurable_spaces_closed_under_intersections} and Lemma~\ref{lem:measure_induced_on_clopen_subspace} we have that $(Y,\MMM\downarrow Y,\mu\downarrow Y)$ is a $\lambda$-measure on $Y$ taking values in $[0,1]$. Since $Y=X\setminus \supp(\mu)$, the family of open sets of measure zero covers $Y$, and by Fact~\ref{fct:strong_zero_dimensionality} there is a clopen partition $\P$ of $Y$ such that $\mu(P)=0$ for every $P\in \P$.
    Notice that $|\P|\leq \lambda$, since $Y\subseteq X$ has weight $\lambda$, and $|\P|=\lambda$ since $\mu$ is a $\lambda$-measure and $\mu(Y)=\mu(\bigcup \P)>0$.
    Let $\{P_i\mid i<\lambda\}$ be an enumeration of $\P$.
    Then, $\mu':\powerset(\lambda)\to [0,1]$ defined by $\mu'(A)=\mu(\bigcup_{i\in A} P)$ is a real-valued measure on $\lambda$. 
    Without loss of generality, we may assume that $\mu'(\lambda)=1$, up to rescaling $\mu'$. Thus, $\lambda$ is real-valued measurable, as wanted.

    \ref{prop:real-valued_measurable_cardinal_iff_lambda_measure-3}$\Rightarrow$\ref{prop:real-valued_measurable_cardinal_iff_lambda_measure-2}. Assume that $\mu':\powerset(\lambda)\to [0,1]$ is a continuous real-valued probabilistic measure on $\lambda$. 
    Let $Z=\{x_\alpha\mid \alpha<\lambda\}$ be a closed discrete subset of $X$ of size $\lambda$ as given by Corollary~\ref{cor:equivalent_conditions_lindelof}.
    Define 
    $\mu:\powerset(Z) \to [0,1]$ by setting 
    \[
    \mu(M)=\mu'(\{\alpha<\lambda\mid x_\alpha\in M\}).
    \]
    Then, it is clear that $(Z, \powerset(Z), \mu)$ is a continuous $\lambda$-measure space, and so is $(X, \powerset(X), \mu \uparrow X)$ by Lemma~\ref{lem:extending_measure_spaces}.
    Finally, since $(x_\alpha)_{\alpha<\lambda}$ is a discrete family, we have that $\{x_\alpha\mid \alpha<\lambda\}\cap Y$ is countable for every second countable $Y\subseteq X$, thus $\mu(Y)=0$ for every second countable $Y\subseteq X$, proving that $\mu$ cannot be countable-induced.

    \ref{prop:real-valued_measurable_cardinal_iff_lambda_measure-2}$\Rightarrow$\ref{prop:real-valued_measurable_cardinal_iff_lambda_measure-1}. It is enough to notice that for every $\kappa\leq \lambda$ there exists a non-$\lambda$-Lindel\"of subspace of $\pre{\kappa}{\lambda}$ of weight $\lambda$ (e.g., the discrete space $X=\{\alpha\conc 0^{(\kappa)}\mid \alpha<\lambda\}$).
\end{proof}

\begin{corollary}\label{cor:measurable_cardinal_iff_lambda_measure}
    Assume $\lambda>\ccc$. The following are equivalent:
    \begin{enumerate-(1)}
\item\label{cor:measurable_cardinal_iff_lambda_measure-1} There is a continuous weak $\lambda$-measure space $(X,\MMM,\mu)$ for some $X\subseteq \pre{\kappa}{\lambda}$ of weight $\lambda$ such that $\mu$ takes values in $[0,1]$.
\item\label{cor:measurable_cardinal_iff_lambda_measure-2} For every non-$\lambda$-Lindel\"of space $X\subseteq \pre{\kappa}{\lambda}$ of weight $\lambda$ there is $\mu:\powerset(Y)\to [0,1]$ such that $(X,\powerset(X),\mu)$ is a continuous weak $\lambda$-measure space.
\item\label{cor:measurable_cardinal_iff_lambda_measure-3} $\lambda$ is measurable.
\end{enumerate-(1)}
\end{corollary}

\begin{proof}
\ref{cor:measurable_cardinal_iff_lambda_measure-1}$\Rightarrow$\ref{cor:measurable_cardinal_iff_lambda_measure-3}. Let $(X,\MMM,\mu)$ be a continuous weak $\lambda$-measure space such that $\mu$ takes values in $[0,1]$.
Notice that every second countable subset $Y\subseteq X\subseteq \pre{\kappa}{\lambda}$ has size $|Y|\leq \ccc<\lambda$, thus $Y$ is measurable of measure zero. This shows that
$(X,\MMM,\mu)$ cannot be countable-induced.
Hence, the results follows from Proposition~\ref{prop:real-valued_measurable_cardinal_iff_lambda_measure} and the classical result that if $\lambda>\ccc$ is real-valued measurable, then it is measurable (see, e.g., \cite[Corollary 10.10]{jechSetTheory2003}).

\ref{cor:measurable_cardinal_iff_lambda_measure-3}$\Rightarrow$\ref{cor:measurable_cardinal_iff_lambda_measure-2}.  Follows immediately from Proposition~\ref{prop:real-valued_measurable_cardinal_iff_lambda_measure} and the fact that every measurable cardinal $\lambda$ is also real-valued measurable.

\ref{cor:measurable_cardinal_iff_lambda_measure-2}$\Rightarrow$\ref{cor:measurable_cardinal_iff_lambda_measure-1}. Obvious.
\end{proof}

\section{Final remarks}\label{sec:final_remarks}

We conclude the paper with a few final remarks and corollaries that follow easily from the theory developed so far.

First, let us consider the space $\pre{\kappa}{\lambda}$.
This space has weight exactly $\lambda^{<\kappa}$. Thus, if $\lambda^{<\kappa}=\lambda\geq \ccc$, Theorem~\ref{thm:impossibility_theorem} immediately implies that there is no weak $\lambda^+$-measure structure on $\pre{\kappa}{\lambda}$.

\begin{corollary}\label{cor:impossibility_theorem_Baire}
    Assume $\lambda^{<\kappa}=\lambda\geq \ccc$. Then, there is no continuous weak $\lambda^+$-measure space $(\pre{\kappa}{\lambda}, \MMM,\mu)$ on $\pre{\kappa}{\lambda}$.
\end{corollary}

If instead $\lambda^{<\kappa}>\lambda$, then necessarily $\kappa>\omega$, and hence $\lambda^{<\kappa}\geq \ccc$. Therefore, by the previous corollary and by Proposition~\ref{prop:lambda<kappa=lambda_is_necessary}, we obtain the following.

\begin{corollary}
    Let $\delta=\lambda^{<\kappa}$ be the weight of $\pre{\kappa}{\lambda}$, and assume $\delta>\lambda$. Then, there is no continuous weak $\delta^+$-measure space $(\pre{\kappa}{\lambda}, \MMM,\mu)$ on $\pre{\kappa}{\lambda}$.
\end{corollary}

Analogous results hold for the generalized Cantor space $\pre{\lambda}{2}$. 
This space has weight $2^{<\lambda}$. Thus, if $2^{<\lambda}=\lambda>\omega$ (which in particular implies $\lambda^{<\cof(\lambda)}=\lambda\geq \ccc$), then by Theorem~\ref{thm:impossibility_theorem} and Fact~\ref{fct:Cantor_subset_Baire_with_cardinal_assumption} we obtain the following.

\begin{corollary}\label{cor:impossibility_theorem_Cantor}
    Assume $2^{<\lambda}=\lambda>\omega$. Let $Y\subseteq \pre{\lambda}{2}$. Then, there is no continuous weak $\lambda^+$-measure space $(Y, \MMM,\mu)$.
\end{corollary}

If instead $2^{<\lambda}>\lambda$ (and thus $2^{<\lambda}\geq \ccc$), by Proposition~\ref{prop:Cantor_always_subset_of_some_Baire} and Theorem~\ref{thm:impossibility_theorem}, we obtain the following.

\begin{corollary}\label{cor:impossibility_theorem_Cantor_general}
    Let $\delta=2^{<\lambda}$ be the weight of $\pre{\lambda}{2}$, and assume $\delta>\lambda$. Then, there is no continuous weak $\delta^+$-measure space $(\pre{\lambda}{2}, \MMM,\mu)$ on $\pre{\kappa}{\lambda}$.
\end{corollary}

While we believe that the appropriate notion of measure on $\pre{\kappa}{\lambda}$ is that of a $\delta^+$-measure, where $\delta=\lambda^{<\kappa}$ is the weight of $\pre{\kappa}{\lambda}$, we notice nonetheless that if for some reason one is specifically interested in continuous weak $\lambda^+$-measure spaces on $\pre{\kappa}{\lambda}$ even when this space has weight greater than $\lambda$, then by Propositions~\ref{prop:lambda<kappa=lambda_is_necessary} and~\ref{prop:small_measures_big_space} (applied to $\delta=\lambda^+\leq \lambda^{<\kappa}$), we get the following.

\begin{corollary}\label{cor:small_measures_big_space}
    Assume $\lambda^{<\kappa}>\lambda$. Then there is a continuous weak $\lambda^+$-measure space $(\pre{\kappa}{\lambda}, \MMM,\mu)$ on $\pre{\kappa}{\lambda}$.
\end{corollary}

In a similar way, if $2^{<\lambda}>\lambda$, then by  Proposition~\ref{prop:Cantor_always_subset_of_some_Baire}, Proposition~\ref{prop:small_measures_big_space},  and Lemma~\ref{lem:extending_measure_spaces}, we get an analogue result for the generalized Cantor space.

\begin{corollary}\label{cor:small_measures_big_space_Cantor}
    Assume $2^{<\lambda}>\lambda$. Then there is a continuous weak $\lambda^+$-measure space $(\pre{\lambda}{2}, \MMM,\mu)$ on $\pre{\lambda}{2}$.
\end{corollary}

In previous sections, we chose to work with a weakening of measurable spaces, rather than with full $\lambda^+$-algebras. However, the results we obtained clearly extend to the case of $\lambda^+$-algebras as well.
Recall that a {$\lambda^+$-measurable space} is a pair $(X,\MMM)$ where $X$ is a topological space and $\MMM$ is a $\lambda^+$-algebra containing all open subsets of $X$.

As in the classical case, for measures defined on $\lambda^+$-algebras monotonicity can be recovered from the other axioms. 
Indeed, if $\MMM$ is a $\lambda^+$-measurable space and $\mu:\powerset(X)\to \SS$ is any partial function satisfying $\MMM\subseteq \dom(\mu)$ and  Axioms~\ref{ax_measure:additive}, then $\mu(B)=\mu(A\cup (B\setminus A))=\mu(A) + \mu(B\setminus A)\geq\mu(A)$ for any $A\subseteq B$, and thus $\mu$ satisfies Axiom~\ref{ax_measure:decreasing}.

\begin{remark}
Let $(X,\MMM)$ be a $\lambda^+$-measurable space, and let $\SS$ be a positively totally ordered monoid equipped with an infinitary sum $\Sum$.
Let $\mu:\powerset(X)\to \SS$ be a partial function satisfying $\MMM\subseteq \dom(\mu)$ and Axioms~\ref{ax_measure:emptyset}, \ref{ax_measure:non-trivial}, \ref{ax_measure:additive}, and~\ref{ax_measure:point-regular}.
Then, $\mu$ satisfies Axiom~\ref{ax_measure:decreasing} as well, and therefore it is a $\lambda^+$-measure on $(X,\MMM)$.
\end{remark}

Obviously, every $\lambda^+$-algebra is closed under finite intersections, and every $\lambda^+$-measure space is $\lambda^+$-partitionable. Thus, in this setting one can apply Propositions~\ref{prop:splitting_into_countable_induced_for_initially_Archimedean_monoids} and~\ref{prop:countable-induced_iff_Dirac} to any $(X,\MMM,\mu)$ and obtain the following.

\begin{proposition}
    Assume $\lambda^{<\kappa}=\lambda$. Let $(X,\MMM,\mu)$ be a continuous $\lambda^+$-measure space. 
    Then, there is a clopen partition $\P$ of $X$ such that $(P,\MMM\downarrow P,\mu\downarrow P)$ is a countable-induced, ($\omega_1$-Dirac,) continuous, $\lambda^+$-measure space for every $P\in \P$ and $\mu$ is the sum of $(\mu\downarrow P)_{P\in \P}$.
\end{proposition}

In particular, from this and Lemma~\ref{lem:impossibility_for_countable_induced}, or directly from Theorem~\ref{thm:impossibility_theorem}, we get the following.

\begin{corollary}\label{cor:impossibility_with_non-weak_structures}
    Assume $\lambda^{<\kappa}=\lambda\geq\ccc$. Let $X\subseteq \pre{\kappa}{\lambda}$. Then, there is no continuous $\lambda^+$-measure space $(X, \MMM,\mu)$.
\end{corollary}

For (non-weak) $\lambda^+$-measure spaces, this result extends to all topological spaces. Suppose $2^{<\lambda}=\lambda$ (which implies $\lambda^{<\cof(\lambda)}=\lambda$). Then, every $T_0$ space of weight at most $\lambda$ is $\lambda^+$-Borel-isomorphic to a subspace of $\pre{\lambda}{2}$, and thus of $\pre{\cof(\lambda)}{\lambda}$ (see, e.g., \cite[Proposition~3.23]{agostiniGeneralizedBorelSets2025}). Therefore, if $\lambda>\omega$ (which together with $2^{<\lambda}=\lambda$ implies $\lambda\geq \ccc$), this yields the following.

\begin{corollary}\label{cor:impossibility_for_all_top_spaces}
    Assume $2^{<\lambda}=\lambda>\omega$. Let $Y$ be a $T_0$ topological space of weight at most $\lambda$. Then, there is no continuous $\lambda^+$-measure space $(Y, \MMM,\mu)$.
\end{corollary}

As in the classical case, a $\lambda^+$-algebra $\MMM$ on $X$ is generated by a $T_0$ topology on $X$ of weight at most $\lambda$ if and only if $\MMM$ separates the points of $X$ and $\MMM$ is generated by a subfamily of size at most $\lambda$ (see, e.g.,  \cite[Proposition~3.24]{agostiniGeneralizedBorelSets2025}). A pair $(X,\MMM)$ with this property is called a \markdef{$\lambda^+$-Borel space}.
Thus, we get the following.

\begin{corollary}\label{cor:impossibility_for_lambda-Borel_spaces}
    Assume $2^{<\lambda}=\lambda>\omega$. Let $(Y,\MMM)$ be a $\lambda^+$-Borel space. Then, there is no continuous $\lambda^+$-measure on it.
\end{corollary}

\bibliographystyle{alpha}
\bibliography{Bibliography}

\end{document}